\documentclass[final]{siamart190516}
\usepackage{amssymb,amsmath, color}

\usepackage{amssymb}

\newcommand{\tcb}{\textcolor{blue}}
\newcommand{\tcr}{\textcolor{red}}

\newcommand{\R}{{\mathbb R}}
\newcommand{\N}{{\mathbb N}}
\DeclareMathOperator{\argmin}{argmin}
\newcommand{\interior}{{\rm int}\kern 0.06em}

\newcommand{\Rb}{\mathbb R\cup\{+\infty\}}

\newcommand{\cC}{{\mathcal C}}
\newcommand{\cE}{{\mathcal E}}
\newcommand{\cH}{{\mathcal H}}
\newcommand{\cO}{{\mathcal O}}

\newcommand{\demi}{\frac{1}{2}}
\newcommand{\ie}{{\it i.e.}\,\,}

\def\<{\langle}
\def\>{\rangle}
\DeclareMathOperator*\prox{prox}%

\usepackage{ulem}

\usepackage{subcaption}
\usepackage{caption}

\usepackage{geometry}
 \usepackage{pict2e}

\usepackage{graphicx,epsfig}
\usepackage{amsmath}

\def\<{\langle}
\def\>{\rangle}

\DeclareMathOperator*\crit{crit}
\DeclareMathOperator*\dist{dist}

\newsiamremark{remark}{Remark}

\newcommand{\TheTitle}{Fast optimization via  inertial dynamics\\ with closed-loop damping}

\newcommand{\TheAuthors}{H. Attouch, R.I. Bo\c t, and E.R. Csetnek}

\headers{\TheTitle}{\TheAuthors}

\title{{\TheTitle}}

\author{
  Hedy Attouch\thanks{IMAG, Universit\'e Montpellier, CNRS, Place Eug\`ene Bataillon, 34095 Montpellier CEDEX 5, France. E-mail: hedy.attouch@umontpellier.fr. Supported by COST Action: CA16228}
  \and
  Radu Ioan \sc Bo\c t\thanks{University of Vienna, Faculty of Mathematics, Oskar-Morgenstern-Platz 1, A-1090 Vienna, Austria. E-mail: radu.bot@univie.ac.at.}
  \and
   Ern\"o Robert  Csetnek\thanks{University of Vienna, Faculty of Mathematics, Oskar-Morgenstern-Platz 1, A-1090 Vienna, Austria. E-mail: robert.csetnek@univie.ac.at. Supported by the Austrian Science Fund: project P 29809-N32}
}

\begin{document}

\maketitle

\date{August 5, 2020}

\begin{abstract}
In a real Hilbert space $\cH$, in order to develop fast optimization methods, we  analyze the asymptotic behavior, as time $t$ tends to infinity, of a large class of  autonomous dissipative inertial continuous dynamics. The function $f: \cH \to \R$ to be minimized (not necessarily convex) enters the dynamic via its gradient, which is assumed to be Lipschitz continuous on the bounded subsets of $\cH$.   This results in autonomous dynamic systems with nonlinear damping and nonlinear driving force.
 We first consider the case where the damping term $\partial \phi (\dot{x}(t))$ acts as a closed-loop control of the velocity. The damping potential $\phi : \cH \to \R_+$ is a convex continuous function which reaches its minimum at the origin.
We  show the existence and  uniqueness of a global solution to  the associated Cauchy problem.
We analyze the asymptotic convergence and the convergence rates of the trajectories generated by this system.
To do this, we use techniques from optimization, control theory, and PDE's: Lyapunov analysis based on the decreasing property of an energy-like function,   quasi-gradient and Kurdyka-Lojasiewicz theory, monotone operator theory for wave-like equations. Convergence rates are obtained based on the geometric properties of the data $f$ and $\phi$. We put forward minimal hypotheses on the damping potential $\phi$ guaranteeing the convergence of trajectories, thus showing the dividing line between strong and weak damping. When $f$ is strongly convex, we give general conditions on the damping potential $\phi$ which provide exponential convergence rates.
Then, we extend the results to the case where an additional
Hessian-driven damping enters the dynamic,  which reduces the oscillations.
Finally, we consider a new inertial system where the damping jointly involves  the velocity  $\dot{x}(t)$ and the gradient  $\nabla f(x(t))$.
% All these results have natural extensions to the case of a convex lower semicontinuous function $f: \cH \to \Rb$. Just replace $f$ with its Moreau envelope.
This study naturally leads to similar results for the
proximal-gradient algorithms obtained by temporal discretization, some of them are studied in the article.
In addition to its original results, this work  surveys the numerous works devoted to the interaction between  damped inertial continuous dynamics and  numerical optimization algorithms, with an emphasis on autonomous systems, adaptive procedures, and convergence rates.
\end{abstract}



%\vspace{1cm}

\section{Introduction}

Our work is part of the active research stream which studies the close link between continuous dissipative dynamical systems and the optimization algorithms obtained by temporal discretization. In this context, second-order evolution equations provide a natural and intuitive way to speed up algorithms.
Then, the optimization properties  come from the damping term. It is the skill of the mathematician to design this term to obtain rapidly converging trajectories and algorithms (ideally, obtain optimal convergence rates). Precisely, we will consider the following system (ADIGE) which  covers a large number of situations. Let us fix the setting.
Let $\mathcal H$ be a real Hilbert space endowed with the scalar product $\< \cdot,\cdot\>$ and norm $\|\cdot\|$.
Let  $f: \cH \to \R$ be a differentiable function (not necessarily convex), whose gradient $\nabla f: \cH \to \cH$ is Lipschitz continuous on the bounded subsets of $\cH$, and such that $\inf_{\cH}  f > -\infty$ (when considering the Hessian of $f$, we will assume that $f$ is twice differentiable).
Our objective is to study from the optimization point of view the Autonomous Damped Inertial Gradient Equation
 \begin{equation*}
\mbox{\rm (ADIGE)}    \qquad \ddot{x}(t) + \mathcal G \Big( \dot{x}(t),  \nabla f({x}(t)), \nabla^2 f({x}(t))\Big)    +  \nabla f (x(t)) = 0,
\end{equation*}
where  the damping term $\mathcal G \Big( \dot{x}(t),  \nabla f (x(t)),  \nabla^2 f({x}(t)) \Big) $ acts as a closed-loop control. Under suitable assumptions, this term will induce dissipative effects, which tend to stabilize asymptotically (\ie as $t \to +\infty$) the trajectories to critical points of $f$ (minimizers in the case   where $f$ is convex).
We will use the generic terminology  \textit{damped inertial continuous dynamics}  to designate second order evolution systems which have a strict Lyapunov function. To be specific we will refer to (ADIGE) or to some of its particular cases.
From there, we can distinguish two distinct  classes of dynamics and algorithms, depending on whether  the damping term involves  coefficients which are given a priori as  functions of time (open-loop damping, non autonomous dynamic), or is a feedback of the current state of the system (closed-loop damping, adaptive methods, autonomous dynamic). We will use these terminologies indifferently, but to be precise they correspond to the case of autonomous, and non-autonomous dynamics respectively.
Indeed, one of our objective is to understand  if the closed-loop damping can do as well (and possibly improve) the fast convergence properties of the accelerated gradient method of Nesterov. Recall that, in convex optimization, the accelerated gradient method of Nesterov (which is associated with a non-autonomous damped inertial dynamic) provides convergence rate of order $1/t^2$, which is optimal for first-order methods (involving only evaluations of  $\nabla f$ at iterates). This justifies the importance of inertial dynamics for developing fast optimization methods (recall that the continuous steepest descent, which is a first order evolution equation, only guarantees  the convergence rate $1/t$ for general convex functions).
Closely related questions concern the impact of geometric properties of data (damping term, objective function) on the convergence rates of trajectories and iterations.  This is a wide subject which concerns continuous optimization, as well as the study of the stabilization of oscillating systems in mechanics and physics.
 Due to the highly nonlinear characteristics of (ADIGE) (non-linearity occurs both in the damping term and in the gradient of $f$), our convergence analysis will mainly rely on the combination of the  quasi-gradient approach for inertial systems initiated by B\'egout--Bolte--Jendoubi \cite {BBJ} with the theory of Kurdyka--Lojasiewicz. The price to pay is that some of the results are only valid in finite dimensional Hilbert spaces.
 It should be noted that the relative simplicity of the functional framework (single functional space, differentiable objective function) does not allow direct application to the corresponding PDEs.
Our objective is mainly the study of optimization problems, but the Lyapunov analysis developed in the article can be a very useful guide for its extension to the PDE framework, as it was done in \cite{AA2}, \cite{BCD}, \cite{CFr}.

 \subsection{Presentation of the results}
 For each of the following systems, we will show  existence and  uniqueness of the solution of the Cauchy problem, and  study its asymptotic behavior.

 \subsubsection{(ADIGE-V)}
Our study mainly concerns  the case
the  differential inclusion
\begin{equation}\label{closed_loop-phi_def}
\mbox{\rm (ADIGE-V)}   \quad 0\in \ddot x(t) +\partial \phi(\dot x(t))+ \nabla f(x(t)),
\end{equation}
where $\phi: \cH \to \R$ is a convex continuous function which achieves its minimum at the origin, and the  operator $\partial \phi: \cH \to 2^{\cH}$ is its convex subdifferential.
The damping term $\mathcal G$  depends only on the velocity, which is reflected by the suffix V.
This model encompasses several classic situations:

\smallskip

\noindent $\bullet$ \, The case
 $\phi (u)= \frac{\gamma}{2} \|u\|^2$ corresponds to the Heavy Ball with Friction method
\begin{equation}\label{HBF}
\mbox{\rm (HBF)} \quad \ddot x(t) + \gamma \dot x(t)+ \nabla f(x(t))=0
\end{equation}
 introduced by B. Polyak  \cite{Pol,Polyak2} and further studied by Attouch--Goudou--Redont \cite{AGR} (exploration of local minima), Alvarez \cite{Alvarez} (convergence in the convex case), Haraux-Jendoubi \cite{HJ1, HJ2} (convergence in the analytic case),
 B\'egout--Bolte--Jendoubi \cite{BBJ} (convergence based on the Kurdyka-Lojasiewicz property), to cite part of the rich literature devoted to this subject.

 \smallskip

\noindent $\bullet$ \, The case  $\phi (u)= r\|u\|$ corresponds to the dry friction
effect. Then, (ADIGE-V) is a differential inclusion (because of $\phi$ nondifferentiable), which, when $\dot x(t)$  is not equal to zero, writes as follows
$$
\quad \ddot x(t) + r \frac{\dot x(t)}{\|\dot x(t)\|}+ \nabla f(x(t))=0.
$$
The importance of this case in optimization comes from the finite time stabilization property of the trajectories, which is satisfied generically with respect to the initial data.
The rigourous mathematical treatment of this case has been considered by Adly--Attouch--Cabot \cite{AAC},  Amann--Diaz \cite{AmaDia}, see  Adly-Attouch \cite{AA-preprint-jca, AA0, AA} for recent developements.
%The importance of this case in optimization comes from the finite time stabilization property of the trajectories, which is satisfied generically with respect to the initial data
%As a specific property of dry friction, this system exhibits trajectories which converge in finite time to approximate equilibria, see Adly-Attouch-Cabot \cite{AAC}, Amann-Diaz \cite{AmaDia},  and recently  \cite{AA-preprint-jca}, \cite{AA0}, \cite{AA}.

 \smallskip

\noindent $\bullet$ \, Taking $\phi (u)= \frac{1}{p}\|u\|^p$ with $ p \geq 1$   allows to
treat these questions in a unifying way. We will pay particular attention to the role played by the parameter $p$ in the asymptotic convergence analysis. For $p>1$ the dynamic writes
$$
\ddot x(t) +  \|   \dot x(t)\|^{p-2} \dot x(t)+ \nabla f(x(t))=0.
$$
We will see that the case $p=2$  separates the weak damping ($p>2$) from the strong damping ($p<2$), hence the importance of this case.

\subsubsection{(ADIGE-VH)}
 Then,
 we will extend the previous results to the  differential inclusion
\begin{equation*}
\mbox{(ADIGE-VH)}\quad  \ddot{x}(t) + \partial \phi (\dot{x}(t))+ \beta \nabla^2 f (x(t))\dot{x}(t)  + \nabla f (x(t)) \ni 0,
\end{equation*}
which, besides   a  damping potential $\phi$ as above acting on the velocity, also involves   a geometric damping driven by the Hessian of $f$, hence the terminology.
The inertial system
\begin{equation*}
{\rm \mbox{(DIN)}}_{\gamma,\beta} \qquad \ddot{x}(t) + \gamma \dot{x}(t) + \beta \nabla^2 f (x(t)) \dot{x}(t)  + \nabla f (x(t)) = 0,
\end{equation*}
was  introduced in \cite{AABR}. In the same spirit as (HBF), the dynamic ${\rm \mbox{(DIN)}}_{\gamma,\beta} $ contains a \textit{fixed} positive viscous friction coefficient $\gamma>0$. The introduction of the Hessian-driven damping allows to damp the transversal oscillations that might arise  with (HBF), as observed in \cite{AABR} in the case of the Rosenbrock function. The need to take a geometric damping adapted to $f$ had already been observed by Alvarez \cite{Alvarez} who considered the inertial system
\[
\ddot{x}(t) + \Gamma \dot{x}(t) + \nabla f (x(t)) = 0 ,
\]
where $\Gamma: \cH \to \cH$ is a linear positive anisotropic operator (see also \cite{BotCse}). But still this damping operator is fixed. For a general convex function, the Hessian-driven damping in $\mbox{\rm (DIN)}_{\gamma,\beta}$ performs a similar operation in an adaptive way. The terminology (DIN) stands shortly for Dynamic Inertial Newton system. It  refers to the natural link between this dynamic and the continuous Newton method, see
 Attouch--Svaiter \cite{ASv}.
Recent studies have been devoted to the study of the  dynamic
\[
\ddot{x}(t) + \frac{\alpha}{t} \dot{x}(t)+  \beta \nabla^2 f (x(t)) \dot{x}(t)  + \nabla f (x(t)) = 0 ,
\]
which combines asymptotic vanishing damping  with Hessian-driven damping.
The corresponding algorithms involve a correcting term in the Nesterov accelerated gradient method which reduces the oscillatory aspects, see  Attouch--Peypouquet--Redont \cite{APR}, Attouch--Chbani--Fadili--Riahi \cite{ACFR},
 Shi--Du--Jordan--Su   \cite{SDJS}.

 \smallskip

\subsubsection{(ADIGE-VGH)}

 Finally, we will consider the new dynamical system
\begin{equation*}
\mbox{(ADIGE-VGH)}\quad \ddot{x}(t) + \partial \phi \Big(\dot{x}(t) + \beta \nabla  f (x(t)\Big)  + \beta \nabla^2  f (x(t)) \dot{x} (t) + \nabla  f (x(t)) \ni 0,
\end{equation*}
 where the damping term $\partial \phi \Big(\dot{x}(t) + \beta \nabla  f (x(t)\Big)$ involves both the velocity vector and  the gradient of the potential function $f$.
The parameter  $\beta \geq 0$ is  attached to the geometric  damping induced by the Hessian. As previously considered, $\phi$ is a damping potential function.
Assuming that $f$ is convex and $\phi$ is a sharp function at the origin, that is
$\phi (u) \geq r\|u\|$ for some $r>0$,
we will show that, for each trajectory generated by  (ADIGE-VGH), the following properties are satisfied:

\smallskip

$i)$  $x(\cdot)$ converges weakly as $t\to +\infty$,  and its limit belongs to $\argmin _{\cH} f$.

\smallskip

$ii)$  $\dot{x}(t)$ and $\nabla  f (x(t))$ converge strongly to zero as $t\to +\infty$.

\smallskip

$iii)$ After a finite time, $x(\cdot)$ follows the steepest descent dynamic.

\subsection{Contents}
The paper is organized in accordance with the above presentation.
In Section  \ref{sec:classic}, we  recall some classical facts concerning the Heavy Ball with Friction method, the Su--Boyd--Cand\`es dynamic approach to the Nesterov method,  the Hessian-driven damping, and the dry friction.
Then, we successively examine each of the cases considered above:
Sections \ref{sec: basic_1}, \ref{sec: basic_2}, \ref{rate-f-str-conv}, \ref{sec:weakdamping}, \ref{sec: basic_3} are devoted to the closed-loop control of the velocity, which is the main part of our study. We  show the existence and uniqueness of a global solution for the Cauchy problem, the  exponential convergence rate in the case $f$ strongly convex, the effect of weak damping, and finally analyze the convergence under the Kurdyka--Lojasiewicz property (KL). Section \ref{sec: basic_4} considers some first related algorithmic results. Section \ref{Sec:Hessian} is devoted to the closed-loop damping with Hessian driven damping.
Section \ref{Sec: combine} is devoted to the closed-loop damping involving the velocity and the gradient.
%Finally, in Section \ref{f_convex_quadratic}, we consider the special case where $f$ is convex quadratic, and obtain  results parallel to those obtained for the damped wave equation in PDE's.
We conclude by mentioning several lines of research for the future.

\section{Classical facts}\label{sec:classic}
Let's recall some classical facts which will serve as comparison tools.
\subsection{(HBF) dynamic system}
The  Heavy Ball with Friction system
\begin{equation*}
{\rm (HBF)}_{r} \quad \ddot{x}(t) + r \dot{x}(t)   +  \nabla f (x(t)) = 0,
\end{equation*}
was introduced by B. Polyak \cite{Pol,Polyak2}. It involves a fixed viscous friction coefficient $r>0$.
Assuming that $ f $ is a convex function such that $\argmin_{\cH} f \neq \emptyset$, we know, by Alvarez's theorem \cite{Alvarez},
 that  each trajectory of ${\rm (HBF)}_{r}$ converges weakly, and its limit belongs to $\argmin_{\cH} f$.
In addition, we have the following convergence rates, the proof of which (see \cite{AC10}) is based on the decrease property of the following Lyapunov functions
\begin{equation*}
\mathcal E (t) :=
  \frac{1}{r^2}(f(x(t))-\min_\mathcal H f)+\frac{1}{2}\|x(t)-x^* + \frac{1}{r}\dot x(t)\|^2 ,
\end{equation*}
where $x^* \in \argmin_{\cH} f$.

\begin{theorem}\label{th.HBF-rate of conv}
Let $f :\mathcal H\to \mathbb R$ be a convex function of class ${\mathcal C}^1$ such that $\mbox{\rm argmin} f \neq~\emptyset$, and let $r$ be a positive parameter. Let $x(\cdot): [0, + \infty[ \rightarrow \cal H$ be a solution trajectory of ${\rm (HBF)}_r$. Set $x(0)= x_0$ and $\dot x (0)= x_1$. Then, we have

\smallskip

$(i)$ $\displaystyle\int_{0}^{+\infty}  (f(x(t))-\min_\mathcal H f)\, dt<+\infty$, \quad $\displaystyle\int_{0}^{+\infty} t \|\dot x(t)\|^2\, dt <+\infty$.

\smallskip

 $(ii)$ \ $f(x(t))-\min_\mathcal H f \leq \displaystyle\frac{C(x_0, x_1)}{t}$,  \quad $\displaystyle\|\dot x(t)\| \leq
 \frac{\sqrt{ 2C(x_0, x_1)}}{\sqrt{t}} $, where
 $$
 C(x_0, x_1) =: \frac{3}{2r} \left( f(x_0) - \min_\mathcal H f \right) + r \mbox{\rm dist}(x_0, \mbox{\rm argmin} f)^2
 + \frac{5}{4r}\|x_1\|^2 .
 $$

 $(iii)$ $\displaystyle f(x(t))-\min_\mathcal H f = o \left(\frac{1}{t}\right)$ \quad and \quad
$\displaystyle\|\dot x(t)\|= o\,\!\left(\frac{1}{\sqrt {t}}\right)$\quad as $t\to +\infty$.
\end{theorem}

Let us now consider the case of a strongly convex function. Recall that a function  $f: \cH \to \mathbb R$ is $\mu$-strongly convex for some $\mu >0$ if   $f- \frac{\mu}{2}\| \cdot\|^2$ is convex.
We have  the following exponential convergence rate, whose proof relies on the decrease property of the following Lyapunov function
$$
\mathcal E (t):= f(x(t))-  \min_{\mathcal H}f  + \frac{1}{2} \| \sqrt{\mu} (x(t) -x^*) + \dot{x}(t)\|^2,
$$
where $x^*$ is the unique minimizer of $f$.

\begin{theorem}\label{strong-conv-thm}
Suppose that $f: \cH \to \mathbb R$ is a function of class ${\mathcal C}^1$ which is $\mu$-strongly convex for some $\mu >0$.
Let  $x(\cdot): [0, + \infty[ \to \cH$ be a solution trajectory of
\begin{equation}\label{dyn-sc-a}
\ddot{x}(t) + 2\sqrt{\mu} \dot{x}(t)  + \nabla f (x(t)) = 0.
\end{equation}
Set $x(0)= x_0$ and $\dot x (0)= x_1$.
 Then, the following property is satisfied:   for all $t\geq 0$
$$
f(x(t))-  \min_{\mathcal H}f  \leq  C e^{-\sqrt{\mu}t},
$$
where \quad
$ C:= f(x_0)-  \min_{\mathcal H}f  + \mu \mbox{\rm dist}(x_0,S)^2 +
\| x_1\|^2 .$
\end{theorem}
A recent account on the best tuning of the damping coefficient can be found in Aujol--Dossal--Rondepierre \cite{ADR}.
The above results show the important role played by the geometric properties of the data in the convergence rates.
Apart from the convex case, the first convergence result for (HBF) was obtained by Haraux--Jendoubi \cite{HJ1} in the case where $f:\R^n \to \R$ is a real-analytic function. They have shown the central role played by Lojasiewicz's inequality  (see also \cite{Chergui}).
Then, on the basis of Kurdyka's work  in real algebraic geometric,   Lojasiewicz's inequality was extended in \cite{BDLM} by Bolte--Daniilidis--Ley--Mazet to a large class of tame functions, possibly nonsmooth.
This is the Kurdyka--Lojasiewicz inequality, to which we will briefly refer (KL).
The convergence of first and second-order proximal-gradient dynamical systems in the context of the (KL) property was obtained by Bo\c t--Csetnek \cite{BotCseESAIM} and Bo\c t--Csetnek--L\'aszl\'o \cite{BotCseLaJEE}. The (KL) property will be a key tool for obtaining convergence rates based on the geometric properties of the data.
Note that this theory only works in the finite dimensional setting \footnote{In the field of PDE's, the Lojasiewicz--Simon theory \cite{CHJ} makes it possible to deal with certain classes of particular problems, such as semi-linear equations.} (the infinite dimensional setting is a difficult topic which is the subject of current research), and only for autonomous systems.
This explains why working with autonomous systems is important,  it allows us to use the powerful  theory (KL).

\subsection{Su-Boyd-Cand\`es dynamic approach to Nesterov accelerated gradient method}

\noindent The following non-autonomous system
\begin{equation*}
\mbox{\rm (AVD)}_{\alpha} \qquad \ddot{x}(t) + \frac{\alpha}{t} \dot{x}(t)  +  \nabla f (x(t)) = 0,
\end{equation*}
 will serve as a reference to compare our results
with the open-loop damping approach. It was introduced in the context of convex optimization by Su--Boyd--Cand\`es in \cite{SBC}.   As a specific feature, the viscous damping coefficient $\frac{\alpha}{t}$ vanishes (tends to zero) as  $t$ goes to infinity, hence the terminology "Asymptotic Vanishing Damping". This contrasts with (HBF) where the viscous friction coefficient is fixed, which prevents obtaining fast convergence of the values for general convex functions.
 Recall the main results concerning the asymptotic behaviour of the trajectories generated by $\mbox{\rm (AVD)}_{\alpha}$.

\smallskip

 \begin{itemize}
  \item For $\alpha \geq 3$, for each trajectory $x(\cdot)$ of $\mbox{\rm (AVD)}_{\alpha}$, \, $f(x(t)) - \inf_{\cH}f =\cO \left(1 /t^2\right)$ as $t \to +\infty$.

  \smallskip

  \item For $\alpha >3$,  each trajectory converges weakly to a minimizer of $f$, see
\cite{ACPR}. In addition, it is shown in \cite{AP} and \cite{May} that $f(x(t)) - \inf_{\cH}f = o\left(1 /t^2\right)$ as $t \to +\infty$.

  \smallskip

\item For $\alpha \leq 3$, we have  $f(x(t)) - \inf_{\cH}f = \displaystyle{\cO\Big( t^{-\frac{2\alpha}{3}}}\Big)$, see
\cite{AAD}  and  \cite{ACR-subcrit}.

\item $\alpha=3$ is a critical value.\footnote{The convergence of the trajectories is an open question in this case.} It corresponds to the historical case studied by Nesterov \cite{Nest1,Nest2}.
  \end{itemize}
The implicit time discretization
of $\mbox{\rm (AVD)}_{\alpha}$  provides an inertial proximal algorithm that enjoys the same properties as the continuous dynamics. Replacing the proximal step by a gradient step gives the
following Nesterov accelerated gradient method (illustrated in Figure \ref{Nest_picture})
$$
 \left\{
\begin{array}{rcl}
y_k&= &  x_{k} + \left(1 -\frac{\alpha}{k}\right) ( x_{k}  - x_{k-1}) \\
\rule{0pt}{15pt}
x_{k+1}& = & y_k- s\nabla f (y_k),
\end{array}\right.
$$
 which still enjoys the same properties when the step size $s$ is  less than the inverse of the Lipschitz constant of $\nabla f$.
Based on the dynamic approach above, many recent studies have been devoted to the convergence properties of the sequences $(x_k)$ and $(y_k)$, which have led to a better understanding and improvement of Nesterov's accelerated gradient algorithm \cite{AAD}, \cite{AC2}, \cite{ACFR}, \cite{ACPR}, \cite{ACR-subcrit},
\cite{AP}, \cite{CD}, \cite{SBC}, and of the Ravine algorithm \cite{GT}, \cite{PKD}.

\smallskip

\begin{figure}
\setlength{\unitlength}{6cm}
\begin{picture}(0.5,0.7)(-0.65,0.00)
\tcb{
\put(0.357,0.51){$y_k =  x_{k} + \left(1- \frac{\alpha}{k} \right)( x_{k}  - x_{k-1})$}
\put(0.328,0.493){{\tiny $\bullet$}}
\put(0.369,0.582){$x_k$}
\put(0.348,0.565){{\tiny $\bullet$}}
\put(0.37,0.657){$x_{k-1}$}
\put(0.364,0.635){{\tiny $\bullet$}}
\put(0.293,0.365){$x_{k+1} = \ y_k - s\nabla f \left( y_k\right) $}
\put(-0.1,0.29){$\argmin f$}
}
\qbezier(-0.1,0.2)(0.4,0.37)(-0.1,0.42)
\qbezier(-0.1,0.15)(0.6,0.4)(-0.1,0.48)
\qbezier(-0.1,0.02)(1.05,0.45)(-0.1,0.6)
\qbezier(0.,-0.01)(1.2,0.53)(-0.1,0.65)
\qbezier(0.16,-0.01)(1.42,0.59)(-0.1,0.709)
\put(-0.1,0.23){\line(1,1){0.16}}
\put(-0.02,0.23){\line(1,1){0.136}}
\put(-0.1,0.32){\line(1,1){0.085}}
\tcr{
\put(0.019,0.633){\line(5,-2){0.6}}
\put(0.341,0.504){\vector(-1,-2){0.06}}
\put(0.385,0.485){\line(-1,-2){0.022}}
\put(0.32,0.463){\line(2,-1){0.042}}
}
\tcr{
\put(0.318,0.505){\line(2,8){0.035}}
}
\end{picture}
\caption{Nesterov accelerated gradient method}
\label{Nest_picture}
\end{figure}
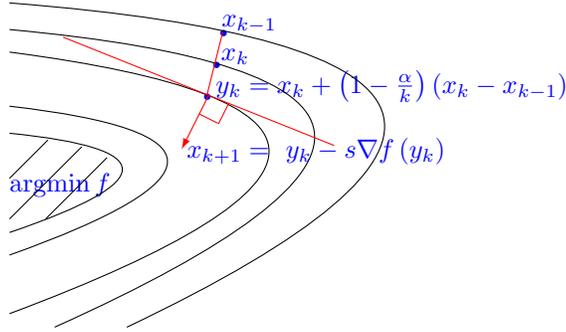

\subsubsection{Optimal convergence rates} In the above results the convergence rates are optimal, that is, they can be reached, or approached arbitrarily close, as shown by the following example from
\cite{ACPR}.
Let us  show that  $\mathcal O (1/t^2)$ is the worst possible case for the rate of convergence of the values for the $ \mbox{(AVD)}_{\alpha} $  trajectories, when $\alpha \geq 3$. It is attained as a limit in the following example.
Take $\mathcal H = \mathbb R$ and $f (x) = c|x|^{\gamma}$, where $c$ and $\gamma$ are positive parameters. We look for nonnegative solutions of $\rm{(AVD)_{\alpha}}$ of the form $x(t)= \frac{1}{t^{\theta}}$, with $\theta >0$. This means that the trajectory is not oscillating, it is a completely damped trajectory. Let us determine the values of $c$, $\gamma$ and $\theta$ that provide such solutions. We have
$$\ddot{x}(t) + \frac{\alpha}{t} \dot{x}(t) = \theta (\theta +1 -\alpha) \frac{1}{t^{\theta+ 2}}, \quad \nabla f (x(t))= c \gamma |x(t)|^{\gamma -2}x(t) = c \gamma \frac{1}{t^{\theta (\gamma -1)}}.
$$
Thus, $x(t)= \frac{1}{t^{\theta}}$ is solution of $\rm{(AVD)_{\alpha}}$  if, and only if,

\smallskip

\begin{itemize}
	\item [i)] $\theta+ 2 = \theta (\gamma -1)$, which is equivalent to $\gamma >2$ and $\theta= \frac{2}{\gamma -2}$; and
	\item [ii)] $c \gamma = \theta (\alpha -\theta -1)$, which is equivalent to $\alpha > \frac{\gamma}{\gamma -2}$ and $c= \frac{2}{\gamma(\gamma -2)}( \alpha - \frac{\gamma}{\gamma -2})$.
\end{itemize}

\smallskip

\noindent We have $\min f = 0$ and
$
f (x(t)) =\frac{2}{\gamma(\gamma -2)}( \alpha - \frac{\gamma}{\gamma -2}  )  \frac{1}{t^{\frac{2 \gamma}{\gamma -2 }}}.$\\
The speed of convergence of $f (x(t))$ to $0$ depends on the parameter $\gamma$. The exponent $\frac{2 \gamma}{\gamma -2 }$ is greater than $2$, and  tends to $2$ when $\gamma$ tends to $+\infty$. This limiting situation is obtained by taking  a function $f$ which becomes very flat around the set of its minimizers. Therefore, without any other geometric assumptions on $f$, we cannot expect a convergence rate better than $\mathcal O (1/t^2)$. This means that it is not possible to obtain a rate $\mathcal O (1/t^{r})$ with $r>2$, which holds for all convex functions. Hence, when $\alpha \geq 3$, $\mathcal O (1/t^{2})$ is sharp. This is not contradictory with the rate $o (t^{-2})$ obtained when  $\alpha  > 3$.

\subsection{Hessian-driven damping}\label{sec:Hessian_intro}

 The inertial system
\begin{equation*}
{\rm \mbox{(DIN)}}_{\gamma,\beta} \qquad \ddot{x}(t) + \gamma \dot{x}(t) + \beta \nabla^2 f (x(t)) \dot{x}(t)  + \nabla f (x(t)) = 0,
\end{equation*}
was  introduced in \cite{AABR}. In line with (HBF), the  viscous friction coefficient $\gamma$ is a \textit{fixed} positive real number. The introduction of the Hessian-driven damping makes it possible to neutralize the oscillations likely to occur with (HBF), a key property for numerical optimization purpose.\\
To accelerate this system, several  studies considered the case where the viscous damping is vanishing. As a model example, which is based on the Su--Boyd--Cand\`es continuous model for the Nesterov accelerated gradient method, we have
\begin{equation}\label{DIN-AVD}
{\rm (DIN-AVD)}_{\alpha, \beta} \qquad \ddot{x}(t) + \frac{\alpha}{t}\dot{x}(t) +  \beta  \nabla^2  f (x(t)) \dot{x} (t) + \nabla  f (x(t)) = 0.
\end{equation}
Considering this sytem, let us quote Attouch--Peypouquet--Redont\cite{APR}, Attouch--Chbani--Fadili--Riahi \cite{ACFR},
 Bo\c t--Csetnek--L\'{a}szl\'{o} \cite{BCL},
Castera--Bolte--F\'evotte--Pauwels \cite{CBFP},  Kim \cite{Kim}, Lin--Jordan \cite{LJ}, Shi--Du--Jordan--Su  \cite{SDJS}.
While preserving the convergence properties of $\mbox{\rm (AVD)}_{\alpha}$, the above system  provides fast convergence to zero of the gradients, namely $\int_{t_0}^{\infty} t^2 \|\nabla f (x(t)) \|^2  dt    < + \infty$ for $\alpha \geq 3$ and $\beta>0$, and  reduces the oscillatory aspects.

\begin{small}
\begin{figure}
\begin{center}
%\begin{tabular}{ccc}
%\includegraphics[width=53mm]{avd_2d_a3epsT20.pdf} &
%\includegraphics[width=53mm]{dinavd_2d_a3epsb1T20.pdf} &\\
%\AVD{\alpha} & \DINAVD{\alpha,\beta}
%\end{tabular}
\includegraphics[width=\textwidth]{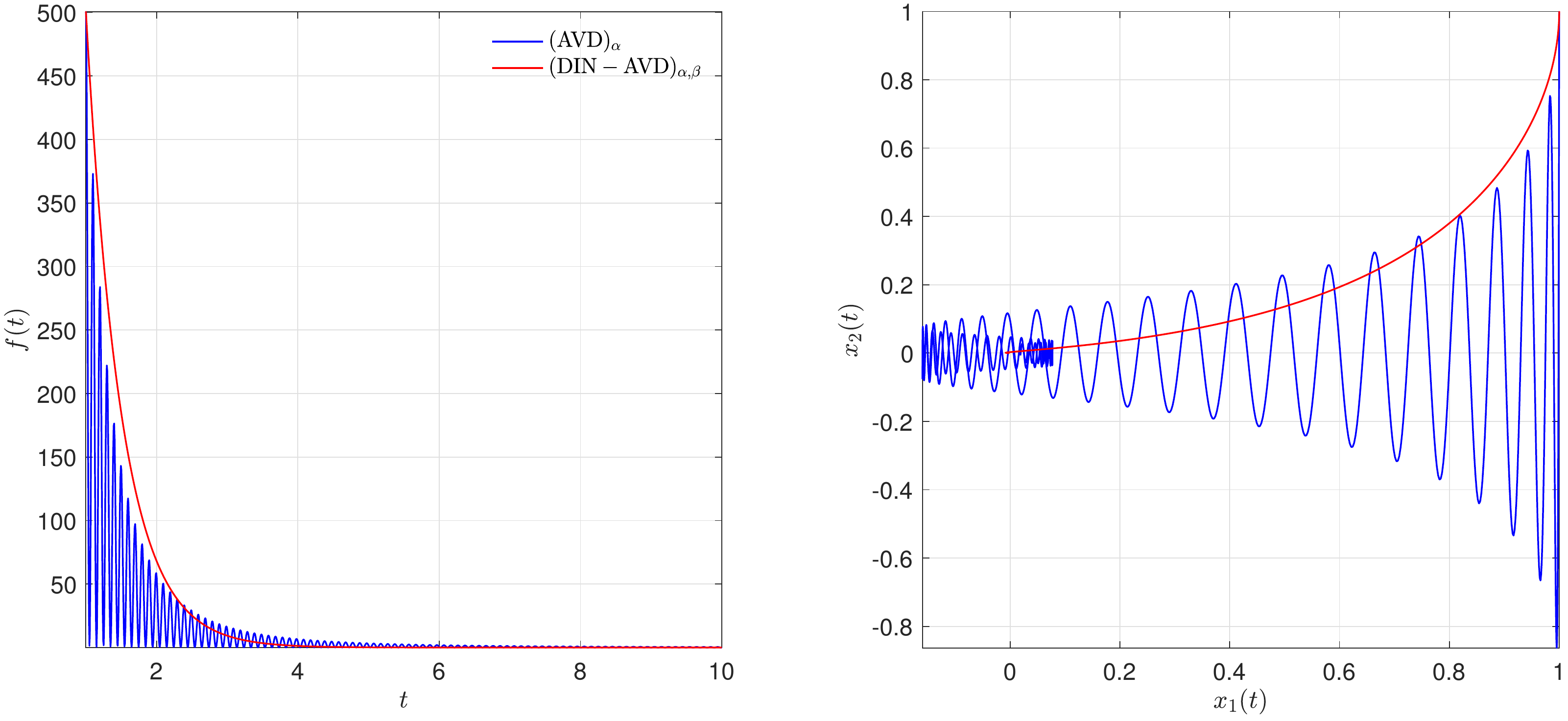}
\end{center}
\caption{Evolution of the objective (left) and trajectories (right) for ${\rm (AVD)}_{\alpha}$ ($\alpha=3.1)$ and  ${\rm (DIN-AVD)}_{\alpha, \beta}$ ($\alpha=3.1,\beta=1$) on an ill-conditioned quadratic problem in $\R^2$.}
\label{figH}
\end{figure}
\end{small}

To illustrate the remarkable effect of the Hessian-driven damping, let's compare the two dynamics ${\rm (AVD)}_{\alpha}$  and  ${\rm (DIN-AVD)}_{\alpha, \beta}$ on a simple ill-conditioned quadratic minimization problem. In the following example of \cite{ACFR},  the trajectories can be computed in closed form. Take $\cH= \R^2$ and $ f(x_1,x_2)=\frac{1}{2}(x_1^2+1000x_2^2)$. We take parameters $\alpha=3.1$, $\beta=1$,  to obey the condition $\alpha >3$. Starting with initial conditions: $(x_1(1),x_2(1))=(1,1)$, $(\dot x_1(1),\dot x_2(1))=(0,0)$, we have the trajectories displayed in Figure~\ref{figH}. We observe that the wild oscillations of ${\rm (AVD)}_{\alpha}$ are neutralized by the presence of the Hessian-driven damping in  ${\rm (DIN-AVD)}_{\alpha, \beta}$.

At first glance, the presence of the Hessian may seem to cause numerical difficulties. However, this is not the case because the Hessian intervenes in the above ODE in the form $\nabla^2  f (x(t)) \dot{x} (t)$, which is nothing other than the derivative with respect to time of $\nabla  f (x(t))$. Thus, the temporal discretization of this dynamic provides first-order algorithms
which, by comparison with the accelerated gradient method of Nesterov, contain a correction term which is equal to the difference of the gradients at two consecutive steps.

\noindent The following closely related inertial system was recently introduced by
Alecsa--L\'aszl\'o--Pinta \cite{ALP}
\begin{equation*}
\ddot{x}(t) + \frac{\alpha}{t} \dot{x}(t) + \nabla f \Big (x(t) + \beta \dot{x}(t) \Big) = 0.
\end{equation*}
The link with ${\rm (DIN-AVD)}_{\alpha, \beta}$ results from  Taylor expansion: as $t \to +\infty$  we have $\dot{x}(t) \to 0$, and so
$ \nabla f \Big (x(t) + \beta \dot{x}(t) \Big) \approx   \nabla  f (x(t)) +  \beta  \nabla^2  f (x(t)) \dot{x} (t) $.

\subsubsection{Hessian-driven damping and unilateral mechanics}

Another motivation for the study of ${\rm \mbox{(DIN)}}_{\gamma,\beta}$ comes from mechanics, and the modeling of damped shocks.
In \cite{AMR}, Attouch-Maing\'e-Redont consider the inertial system with Hessian-driven damping
\begin{equation}
\label{shoks}
\ddot{x}(t) + \gamma \dot{x}(t) + \beta \nabla^2 f (x(t))\dot{x}(t) + \nabla f (x(t)) + \nabla g(x(t))\; =0,
\end{equation}
where $g:\cH\to\R$ is a smooth real-valued function.
 An interesting property of this system
is that, after the introduction of an auxiliary variable $y$, it can be equivalently written as a first-order
system involving only the time derivatives $\dot{x}(t)$, $\dot{y}(t)$ and the gradient terms $\nabla f (x(t))$, $\nabla g (x(t))$. More precisely, the system \eqref{shoks} is equivalent to the following first-order differential equation
\begin{equation}
\label{shoks2}
\left\{
\begin{array}{l}
\dot{x}(t)+\beta \nabla f(x(t))+ax(t)+by(t)=0,\vspace{1mm}\\
\dot{y}(t)-\beta\nabla g(x(t))+ax(t)+by(t)=0,
\end{array}
\right.
\end{equation}
where $a$ and $b$ are real numbers such that: $a+b=\gamma$ and $\beta b=1$. Note that \eqref{shoks2} is different from the classical Hamiltonian formulation, which would still involve the Hessian of $f$. In contrast,  the formulation \eqref{shoks2}  uses  only first-order information from the function $f$ (no occurence of the Hessian of $f$).
 Replacing $\nabla f$ by $\partial f$ in \eqref{shoks2}
 allows us to extend the analysis to the case of a convex lower semicontinuous function $f :\cH \to \Rb $, and so to introduce constraints in the model. When $f = \delta_K$ is the indicator
function of a closed convex set $K \subset \cH$ , the subdifferential operator $\partial f$ takes account of the
contact forces, while $\nabla g$ takes account of the driving forces. In this setting, by playing with
the geometric damping parameter $\beta$, one can describe nonelastic shock laws with restitution
coefficient (for more details we refer to \cite{AMR} and references therein).
 Combination of dry friction ($\phi(u)= r\|u\|$) with Hessian damping
has been considered by Adly--Attouch  \cite{AA-preprint-jca}, \cite{AA}.

\subsection{Inertial dynamics with dry friction}
Although dry friction (also called Coulomb friction) plays a fundamental role in mechanics, its use in optimization has only recently been developed. Due to the nonsmooth character of the associated damping function $ \phi (u) = r \| u \| $,
the dynamics is a differential inclusion, which, when the speed is not equal to zero, is given by
$$
\quad \ddot x(t) + r \frac{\dot x(t)}{\|\dot x(t)\|}+ \nabla f(x(t))=0.
$$
In this case,  the energy estimate gives
$\int_{0}^{+\infty} \| \dot{x}(t) \| dt <+\infty.$
Therefore, the trajectory has finite length, and it converges strongly.
The limit $x_{\infty}$  of the trajectory  $x(\cdot)$ satisfies
$$
\| \nabla f(x_{\infty} )\| \leq r.
$$
Thus, $x_{\infty}$ is an ``approximate''  critical point
 of $f$. In practice, for optimization purpose, we  choose  a small $r>0$.
This amounts to solving the optimization problem $\min_{\cH} f$ with the  variational principle of Ekeland, instead of the Fermat rule.
The importance of this case in optimization comes from the finite time stabilization property of the trajectories, which is satisfied generically with respect to the initial data.
The rigourous mathematical treatment of this case has been considered by Adly--Attouch--Cabot \cite{AAC}, see  Adly--Attouch \cite{AA-preprint-jca, AA0, AA} for recent developements.
Corresponding PDE's results have been obtained by Amann--Diaz \cite{AmaDia} for the nonlinear wave equation, and by
Carles--Gallo \cite{CarlesGallo} for the nonlinear Schr\"odinger equation.

%The importance of this case in optimization comes from the finite time stabilization property of the trajectories, which is satisfied generically with respect to the initial data
%As a specific property of dry friction, this system exhibits trajectories which converge in finite time to approximate equilibria, see Adly-Attouch-Cabot \cite{AAC}, Amann-Diaz \cite{AmaDia},  and recently  \cite{AA-preprint-jca}, \cite{AA0}, \cite{AA}.

\subsection{Closed-loop versus open-loop damping}
In the strongly convex case, the autonomous system (HBF) provides an exponential rate of convergence.
On the other hand, the $\mbox{\rm (AVD)}_{\alpha}$  system provides a convergence rate of order $1/t^{\alpha}$.
Thus, in this case, the closed-loop damping behaves better that the open-loop damping.
For general convex functions (\ie in the worst case), we have the opposite  situation.
 $\mbox{\rm (AVD)}_{\alpha}$ provides a convergence rate  $1/t^2$, while (HBF)  gives only $1/t$.
In this paper, we will study  the impact of the choice of the damping potential on the rate of convergence.
A related question is:  using  autonomous systems, can we obtain for general convex functions, a convergence rate of order $1/t^2$,
\ie as good as the Nesterov accelerated gradient method?
As we will see, to answer these questions, we will have to study different types of closed-loop damping, and rely on the geometric properties of the data.
These questions fall within the framework of an active research current, to quote some recent works, Apidopoulos--Aujol--Dossal--Rondepierre \cite{AADR} (geometrical properties of the data), Iutzeler--Hendricx \cite{IH}  (online acceleration), Lin--Jordan \cite{LJ} (control perspective on high-order optimization), Poon--Liang \cite{PL}  (geometry of first-order methods and adaptive acceleration).

\vspace{3mm}

\section{Damping via closed-loop velocity control, existence and uniqueness}\label{sec: basic_1}
In this section, we will successively introduce the notion of damping potential, then prove the existence and  uniqueness of the solution of the corresponding Cauchy problem.

\subsection{Damping potential}

We consider  the  differential inclusion
\begin{equation*}
\mbox{\rm (ADIGE-V)} \qquad  0\in \ddot x(t) +\partial \phi(\dot x(t))+ \nabla f(x(t)),
\end{equation*}
where $\phi$ is a convex damping potential, which is defined below.

\begin{definition}\label{def1}
A function $\phi : \cH \to \R_+$ is  a damping potential if it satisfies (i), (ii), (iii):

$(i)$ $\phi$ is a nonnegative convex continuous function;

\smallskip

$(ii)$  $\phi (0) = 0  = \min_{\cH} \phi $;

\smallskip

$(iii)$ the minimal section of $\partial \phi$ is bounded on the bounded sets, that is,  for any $R>0$
$$
\sup_{\|u\|\leq R} \|(\partial \phi)^0(u)  \| <+\infty.
$$
\end{definition}
In the above,  $(\partial \phi)^0(u)$ is the element of minimal norm of the closed convex non empty set $\partial \phi(u)$, see \cite[Proposition 2.6]{Brezis}.
Note that, when $\cH$ is finite dimensional, then property $(iii)$ is automatically satisfied. Indeed, in this case, $\partial \phi$ is bounded on the bounded sets, see \cite[Proposition 16.17]{BC}.

\noindent The concept of damping potential is flexible, and allows to cover various situations. For example,
$$ \phi_1(u)=\frac{\gamma}{2}\|u\|^{2} + r\|u\|, \quad \phi_2(u)=\max\{\frac{\gamma}{2}\|u\|^{2}; r\|u\| \}$$
are damping potentials
which combine  dry friction with viscous damping, see \cite{AA}.

\subsection{Existence and uniqueness results}

In this section, we  study the existence and the uniqueness of the solution of the Cauchy problem associated with \mbox{\rm (ADIGE-V)},  where $\phi$ is a convex damping potential.
No convexity assumption is made on the function $f$, which is supposed to be differentiable.
Since we work with  autonomous (dissipative) systems, we can take an arbitrary initial time $t_0$.
As is usual, we take $t_0 =0$, and hence work on the time interval $[0, +\infty[$.\\
Let us precise the notion of strong solution.

\begin{definition}
The trajectory  $x: [0, +\infty[ \to \cH$ is said to be a strong global solution of
\mbox{\rm (ADIGE-V)} if it satisfies the following properties:

\smallskip

$(i)$ $x\in \mathcal C^1 ([0, +\infty[; \cH)$,

\smallskip

$(ii)$  $\dot{x} \in {\rm Lip} (0, T; \cH)$, $\ddot x \in L^{\infty} (0, T; \cH)$ for all $T>0$,

\smallskip

$(iii)$ For almost all $t>0$, \,
$
0\in \ddot x(t) +\partial \phi(\dot x(t))+ \nabla f(x(t)).
$
\end{definition}
Note that, since $\dot{x} \in {\rm Lip} (0, T; \cH)$, it is absolutely continuous on the bounded time intervals, its distribution derivative coincide with its derivative almost everywhere (which exists). Thus, the acceleration $\ddot x$ belongs to $L^{\infty} (0, T; \cH)$ for all
$T>0$, but it is not necessarily continuous, see \cite[Appendix]{Brezis} for further details on  vector-valued  Lebesgue and Sobolev spaces.

Let's prove the following existence and uniqueness result for the associated Cauchy problem.
\begin{theorem}\label{basic_exist_thm} Let $f: \cH \to \R$ be a differentiable function whose gradient is  Lipschitz continuous on the bounded subsets of $\cH$, and such that $\inf_{\cH} f >-\infty$. Let $\phi : \cH \to \R_+$ be a damping potential (see Definition \ref{def1}).  Then, for any $x_0,x_1\in \mathcal H$, there exists a unique strong global solution $x: [0, +\infty[ \to \cH$ of \mbox{\rm (ADIGE-V)} such that $x(0)=x_0$ and $\dot x(0)=x_1$, that is
$$\left\{
\begin{array}{l}
 0\in \ddot x(t) +\partial \phi(\dot x(t))+ \nabla f(x(t))  \vspace{3mm}\\
 x(0)=x_0, \, \dot x(0)=x_1.
\end{array}\right.
$$
\end{theorem}

\begin{proof} We successively consider the case where $ \nabla f $ is Lipschitz continuous over the whole space, then the case where it is supposed to be Lipschitz continuous only on the bounded sets.
In both cases, the idea is to mix the existence results for ODEs which are respectively based on the Cauchy--Lipschitz theorem, and on the theory of maximally monotone operators. We treat the two cases independently because the proof is much simpler in the first case.

\medskip

 \textbf{Case a)  $\nabla f$ is Lipschitz continuous on the whole space}. The Hamiltonian formulation of
(ADIGE-V)  gives the equivalent first-order differential
inclusion in the product space $\mathcal H\times \mathcal H$:
\begin{equation}\label{first_order_cl_loop_0}
0\in\dot z(t)+\partial \Phi(z(t))+F(z(t)),
\end{equation}
where $z(t)=(x(t), \dot x(t)) \in \mathcal H\times \mathcal H $, and
\begin{itemize}
\item
$\Phi:\mathcal H\times \mathcal H\to \R$ is the convex function defined by $\Phi(x,u)=\phi(u)$

\item   $F: \mathcal H\times \mathcal H\to \mathcal H\times \mathcal H$
is defined by \,  $F(x,u)=(-u,\nabla f(x))$.
\end{itemize}
Since $\nabla f$ is Lipschitz continuous on the whole space  $\cH$, we immediately get that    $F$ is a Lipschitz continuous mapping on $\mathcal H\times \mathcal H$. So, we can apply a result related to evolution equations governed by  Lipschitz perturbations of
convex subdifferentials \cite[Proposition 3.12]{Brezis} in order to conclude that \eqref{first_order_cl_loop_0} has
a unique strong global solution with initial data $z(0)=(x_0,x_1)$.

\medskip

\textbf{Case b)  $\nabla f$ is Lipschitz continuous on the bounded sets}.
The major difficulty in (ADIGE-V) is the presence of the term
$\partial \phi (\dot{x}(t))$, which involves a possibly nonsmooth operator $\partial \phi $.
A natural idea is to regularize this operator, and thus obtain a classical evolution equation.
To this end, we  use the Moreau-Yosida regularization.
Let us recall some basic facts concerning this regularization procedure.
 For any $\lambda >0$,  the Moreau envelope of $\phi$  of index $\lambda$  is the function $\phi_{\lambda}: \cH \to \mathbb R $  defined by: for all $u\in \cH$,
$$
\phi_{\lambda} (u) = \min_{\xi \in \cH} \left\lbrace \phi (\xi) + \frac{1}{2 \lambda} \| u - \xi \|^2   \right\rbrace.
$$
The function  $\phi_{\lambda} $ is  convex, of class $ {\mathcal C}^{1,1}$, \, and satisfies  $\inf_{\cH} \phi_{\lambda} = \inf_{\cH} \phi $, $\argmin_{\cH} \phi_{\lambda} = \argmin_{\cH} \phi$.
One can consult \cite[section 17.2.1]{ABM}, \cite{AE}, \cite{BC}, \cite{Brezis} for an in-depth study of the properties of the Moreau envelope in a Hilbert framework.
In our context, since $\phi: \cH \to \R$ is a damping potential, we can easily verify that
$\phi_{\lambda} $ is still  a damping potential. In particular $\phi_{\lambda}(0) = \inf_{\cH} \phi_{\lambda}=0  $. According to the subdifferential inequality for convex functions, this implies that, for all $u\in \cH$
\begin{equation}\label{ineq_phi}
\left\langle   \nabla \phi_{\lambda}(u), u \right\rangle \geq 0.
\end{equation}
% This inequality will be  useful later.
 We will also use the following inequality, see \cite[Proposition 2.6]{Brezis}: for any $\lambda >0$, for any $u\in \cH$
\begin{equation}\label{ineq_phi_b}
\|  \nabla \phi_{\lambda}(u)\|  \leq \|(\partial \phi)^0(u)  \|.
\end{equation}
So, for each $\lambda >0$, we consider  the approximate evolution equation
\begin{equation}\label{existence_approx}
 \ddot{x}_{\lambda}(t)  +  \nabla \phi_{\lambda} (\dot{x}_{\lambda}(t)) + \nabla f (x_{\lambda}(t)) = 0,\; t\in [0,+\infty[.
\end{equation}
We will first prove the existence and uniqueness of a global classical solution $x_{\lambda}$ of \eqref{existence_approx} satisfying
$x_{\lambda}(0)=x_0$ and $\dot {x}_{\lambda}(0)=x_1$.
Then, we will prove that the filtered sequence $(x_{\lambda})$
converges uniformly as $\lambda \to 0$ over the bounded time intervals towards a solution of   (ADIGE-V).
According to the Hamiltonian formulation of \eqref{existence_approx}, it is equivalent
to consider  the first-order (in time) system
\begin{equation}\label{Hamilton_Yosida}
 \left\{
\begin{array}{l}
\dot x_{\lambda}(t)   -u_{\lambda}(t) =0;  	 \\
\rule{0pt}{18pt}
 \dot{u}_{\lambda}(t) +\nabla \phi_{\lambda}(u_{\lambda}(t))+ \nabla f(x_{\lambda}(t))  = 0 ,
 \hspace{2.3cm}
\end{array}\right.
\end{equation}
 with the Cauchy data
$x_{\lambda}(0) =x_0$, \, $u_{\lambda}(0)= x_1       $.
Set
$Z_{\lambda}(t) = (x_{\lambda}(t), u_{\lambda}(t)) \in \cH \times \cH .$\\
The  system (\ref{Hamilton_Yosida})  can be written equivalently as
$$
\dot{Z}_{\lambda}(t) + F_{\lambda}( Z_{\lambda}(t)) =  0, \quad Z_{\lambda}(0) = (x_0, x_1),
$$
where  $F_{\lambda}: \cH \times \cH\rightarrow \cH \times \cH,\;\;(x,u)\mapsto F_{\lambda}(x,u)$ is defined by
$$
F_{\lambda}(x,u)= \Big( 0,  \nabla \phi_{\lambda}(u) \Big) +
 \Big( -u, \nabla f(x)   \Big).
$$
Hence $F_{\lambda}$ splits as follows
$
F_{\lambda}(x,u) = \nabla \Phi_{\lambda} (x,u) + G (x,u),
$
where
$$
\Phi (x,u) =   \phi(u), \quad  \Phi_{\lambda}(x,u)= \phi_{\lambda}(u),
\quad
G(x,u) = \Big(  -u, \, \nabla f(x)  \Big).
$$
Therefore, it is equivalent to consider the first-order differential inclusion with Cauchy data
\begin{equation}
\label{1odd_existence}
\dot{Z}_{\lambda}(t) +\nabla \Phi_{\lambda}(Z_{\lambda}(t)) + G( Z_{\lambda}(t))= 0, \quad Z_{\lambda}(0) = (x_0, x_1).
\end{equation}
According to the Lipschitz continuity of $\nabla \Phi_{\lambda}$,  and the fact that $G$ is Lipschitz continuous on the bounded sets, we have that the sum operator $ \nabla \Phi_{\lambda} + G$ which governs \eqref{1odd_existence} is  Lipschitz continuous on the bounded sets.
As a consequence, the existence of a local solution to \eqref{1odd_existence} follows from the classical Cauchy--Lipschitz theorem.
To pass from a local solution to a global solution, we use a standard energy argument, and the following a priori estimate on the solutions of \eqref{existence_approx}. After taking the scalar product of
\eqref{existence_approx} with $\dot{x}_{\lambda} $, and using
\eqref{ineq_phi}, we get that the global energy
\begin{equation}\label{energy_approx}
\mathcal E_{\lambda} (t):= f(x_{\lambda}(t)) -\inf\nolimits_{\cH}f + \demi \| \dot{x}_{\lambda}(t) \|^2 ,\end{equation}
is a decreasing function of $t$. According to the Cauchy data, and $f$ minorized,
this implies that, on any bounded time interval, the filtered sequences of functions
$(x_{\lambda})$ and $(\dot{x}_{\lambda}) $ are bounded.
According to the property \eqref{ineq_phi_b} of the Yosida approximation, and the property $(iii)$ of the
 damping potential $\phi$, this implies that
$$
\| \nabla \phi_{\lambda} (x_{\lambda}(t))\| \leq \| (\partial \phi )^{0} (x_{\lambda}(t))\|
$$
is also bounded uniformly with respect to $\lambda >0$  and $t$ bounded. According to the constitutive equation \eqref{existence_approx}, this in turn implies that the filtered sequence $(\ddot{x}_{\lambda} )$ is also bounded.
This implies that if a maximal solution is defined on a finite time interval $[0, T[$, then the limits of $x_{\lambda}(t)$ and $\dot{x}_{\lambda} (t)$
exist, as $t \to T$. Then,  we can apply the local existence result, which gives a solution defined on a larger interval, thus contradicting the maximality of $T$.

To prove the uniform  convergence of the filtered sequence $(Z_{\lambda})$ on the bounded time intervals, we proceed in a similar way as in the proof of Br\'ezis  \cite[Theorem 3.1]{Brezis}, see also
Adly-Attouch \cite{AA-preprint-jca} in the context of damped inertial dynamics.
Take  $T >0$, and $ \lambda, \mu >0$.
Consider the corresponding solutions of \eqref{1odd_existence}  on $[0, T]$
\begin{eqnarray*}
&&\dot{Z}_{\lambda}(t) +\nabla \Phi_{\lambda}(Z_{\lambda}(t)) + G( Z_{\lambda}(t))= 0, \quad Z_{\lambda}(0) = (x_0, x_1) \smallskip \\
&&\dot{Z}_{\mu}(t) +\nabla \Phi_{\mu}(Z_{\mu}(t)) + G( Z_{\mu}(t))= 0, \quad Z_{\mu}(0) = (x_0, x_1).
\end{eqnarray*}
Let's make the difference between the two above equations, and take the scalar product with $Z_{\lambda}(t) - Z_{\mu}(t)$. We obtain
\begin{eqnarray}
\demi \frac{d}{dt}\| Z_{\lambda}(t) - Z_{\mu}(t) \|^2 &+ &
\left\langle  \nabla \Phi_{\lambda}(Z_{\lambda}(t)) - \nabla \Phi_{\mu}(Z_{\mu}(t)) ,  Z_{\lambda}(t) - Z_{\mu}(t) \right\rangle \nonumber\\
&+& \left\langle  G( Z_{\lambda}(t)) - G( Z_{\mu}(t)) ,  Z_{\lambda}(t) - Z_{\mu}(t) \right\rangle =0 . \label{basic_ex_Y}
\end{eqnarray}
We now use the following ingredients:

\medskip

\noindent a) According to the  properties of the Yosida approximation (see \cite[Theorem 3.1]{Brezis}), we have
$$
\left\langle  \nabla \Phi_{\lambda}(Z_{\lambda}(t)) - \nabla \Phi_{\mu}(Z_{\mu}(t)) ,  Z_{\lambda}(t) - Z_{\mu}(t) \right\rangle
\geq -\frac{\lambda}{4} \|\nabla \Phi_{\mu}(Z_{\mu}(t))  \|^2 -
\frac{\mu}{4} \|\nabla \Phi_{\lambda}(Z_{\lambda}(t))  \|^2.
$$
Since the filtered sequences $(x_{\lambda})$ and $(\dot{x}_{\lambda}) $ are uniformly bounded  on $[0, T]$, there exists a constant $C_T$ such that, for all $0\leq t \leq T$
$$\| Z_{\lambda}(t)\|\leq C_T .$$
According to \eqref{ineq_phi_b}, and the fact that $\phi$ is a damping potential (property $(iii)$ of Definition \eqref{def1}), we deduce that
$$
 \|\nabla \Phi_{\lambda}(Z_{\lambda}(t))  \|  \leq \sup_{\|\xi\|\leq C_T} \|(\partial \phi)^0(\xi)  \|= M_T <+\infty.
$$
Therefore
$$
\left\langle  \nabla \Phi_{\lambda}(Z_{\lambda}(t)) - \nabla \Phi_{\mu}(Z_{\mu}(t)) ,  Z_{\lambda}(t) - Z_{\mu}(t) \right\rangle
\geq -\frac{1}{4} M_T (\lambda +\mu).
$$
b)  According to the local Lipschitz assumption  on $\nabla f$, the mapping $G : \cH \times  \cH  \to \cH \times  \cH$ is Lipschitz continuous on the bounded sets.
Using again that the sequence $(Z_{\lambda})$ is uniformly bounded on $[0, T]$, we deduce that there exists a constant $L_T$ such that, for all $t\in [0,T]$, for all $\lambda, \mu >0$
$$
\|  G( Z_{\lambda}(t)) - G( Z_{\mu}(t)) \| \leq L_T \|   Z_{\lambda}(t) -  Z_{\mu}(t) \|.
$$
Combining the above results, and using Cauchy--Schwarz inequality, we deduce from
\eqref{basic_ex_Y} that
$$
\demi \frac{d}{dt}\| Z_{\lambda}(t) - Z_{\mu}(t) \|^2
\leq  \frac{1}{4} M_T (\lambda +\mu) + L_T \|   Z_{\lambda}(t) -  Z_{\mu}(t) \|^2 .
$$
We now proceed with the integration of this differential inequality.
Using that $ Z_{\lambda}(0) - Z_{\mu}(0) =0$, elementary calculus gives
$$
\| Z_{\lambda}(t) - Z_{\mu}(t) \|^2 \leq \frac{M_T}{4L_T}(\lambda +\mu) \Big( e^{2L_T t}  -1  \Big).
$$
Therefore, the filtered sequence $(Z_{\lambda})$  is a Cauchy sequence for the uniform convergence on $[0, T]$, and hence it converges uniformly.
This means the uniform convergence of $x_{\lambda}$ and $\dot{x}_{\lambda}$ to $x$ and $\dot{x}$ respectively.
To go to the limit on \eqref{existence_approx}, let us write it as follows
\begin{equation}\label{hbdf_lambda_bbb}
  \nabla \phi_{\lambda} (\dot{x}_{\lambda}(t)) =  \xi_{\lambda} (t)
\end{equation}
where
$
\xi_{\lambda} (t) :=  -\ddot{x}_{\lambda}(t)  - \nabla f (x_{\lambda}(t)) .
$
We now rely on the variational convergence properties of the Yosida approximation.
Since $(\phi_{\lambda})$ converges increasingly to $\phi$ as
$\lambda \downarrow 0$, the sequence of integral functionals
$$
\Psi^{\lambda}(\xi) := \int_{0}^T \phi_{\lambda} (\xi(t))dt
$$
converges increasingly to
$
\Psi (\xi)= \int_{0}^T \phi (\xi(t))dt.
$
Therefore $(\Psi^{\lambda})$   Mosco-converges to $\Psi$ in
$\mathbb L^2 (0, T; \cH)$.
According to the theorem which makes the link between the Mosco convergence of a sequence of convex lower semicontinuous functions and the graph convergence of their subdifferentials, see  Attouch \cite[Theorem 3.66]{Att00}),  we have that
$$
\partial \Psi^{\lambda} \to \partial \Psi
$$
with respect to the topology ${\rm strong}-\mathbb L^2 (0, T; \cH) \times {\rm weak}-\mathbb L^2 (0, T; \cH) $. According to \eqref{hbdf_lambda_bbb} we have
$$
\xi_{\lambda} =\nabla  \Psi^{\lambda} (\dot{x}_{\lambda}).
$$
Since
$
\dot{x}_{\lambda} \to \dot{x} \, \mbox{ strongly in} \, \mathbb L^2 (0, T; \cH)
$
and $\xi_{\lambda}$ converges weakly in $\mathbb L^2 (0, T; \cH)$ to $\xi$ given by
\begin{equation}\label{def:xi}
\xi (t) =  -\ddot{x}(t) - \nabla f (x(t)) ,
\end{equation}
we deduce that $\xi \in \partial \Psi (\dot{x})$, that is
$$
\xi (t) \in \partial \phi (\dot{x}(t)).
$$
According to the formulation (\ref{def:xi}) of $\xi$, we finally obtain  that $x$ is a solution of  (ADIGE-V).\\
The uniqueness of the solution of the Cauchy problem is obtained exactly in the same way as in the case of the global Lipschitz assumption.
\end{proof}

\begin{remark} {\rm
The above existence and uniqueness result uses as  essential ingredient that the potential function $f$ to be minimized is a  differentiable function, whose  gradient is locally Lipschitz continuous.
The introduction of constraints into $f$  via the indicator function would lead to solutions  involving shocks when reaching the boundary
of the constraint. In this case, existence can be still obtained in finite dimension, but uniqueness may not be satisfied, see
Attouch--Cabot--Redont \cite{ACR}.}
\end{remark}

\section{Closed-loop velocity control, preliminary convergence results }\label{sec: basic_2}

Let  $x: [0, +\infty[ \to \cH$ be a solution trajectory of (ADIGE-V).
\subsection{Energy estimates}
Define the global energy at time $t$  as follows:
\begin{equation}\label{energy}
\mathcal E (t):= f(x(t)) -\inf\nolimits_{\cH}f + \demi \| \dot{x}(t) \|^2 .
\end{equation}
Take the scalar product of (ADIGE-V) with $ \dot {x} (t) $. According to the derivation chain rule, we get
\begin{equation}\label{closed_loop_2}
 \frac{d}{dt} \mathcal E (t) + \left\langle \partial \phi(\dot x(t)),   \dot x(t) \right\rangle  = 0.
\end{equation}
The convex subdifferential inequality, and $\phi (0)=0$, gives, for all $u \in \cH$
\begin{equation}\label{energy_2}
\left\langle \partial \phi(u),   u \right\rangle \geq \phi (u).
\end{equation}
Combining the two above inequalites, we get
\begin{equation}\label{closed_loop_2_b}
 \frac{d}{dt} \mathcal E (t) +  \phi(\dot x(t))   \leq 0.
\end{equation}
Since $\phi$ is nonnegative, this implies that the global energy is non-increasing. Since $f$ is minorized, this implies that the velocity $\dot{x}(t)$ is bounded over $[0, +\infty[$. Precisely,
\begin{equation}\label{closed_loop_2_c}
\sup_{t \geq 0}  \| \dot{x}(t) \|  \leq R_1:= \sqrt{2 \mathcal E (0)}.
\end{equation}
To go further, suppose that the trajectory $x(\cdot)$ is bounded (this is verified for example if $f$ is coercive), and set
\begin{equation}\label{closed_loop_2_d}
\sup_{t \geq 0}  \| x(t) \|  \leq R_2.
\end{equation}
Let us now establish a bound on the acceleration.
For this, we rely on the approximate dynamics
\begin{equation}\label{existence_approx_b}
 \ddot{x}_{\lambda}(t)  +  \nabla \phi_{\lambda} (\dot{x}_{\lambda}(t)) + \nabla f (x_{\lambda}(t)) = 0,\; t\in [0,+\infty[.
\end{equation}
A similar estimate as above gives $\sup_{t \geq 0}  \| \dot{x}_{\lambda}(t) \|  \leq R_1:= \sqrt{2 \mathcal E (0)}$.
According to property $(iii)$ of the damping potential, we obtain
$$
 \|\nabla \phi_{\lambda} (\dot{x}_{\lambda}(t))   \|  \leq \sup_{\|u\|\leq R_1} \|(\partial \phi)^0(u)  \|= M_1 <+\infty.
$$
According to the local Lipschitz continuity property of $\nabla f$
$$
 \|\nabla f (x_{\lambda}(t))   \|  \leq \sup_{\|x\|\leq R_2} \|\nabla f (x)  \|= M_2 <+\infty.
$$
Combining the two above inequalities with \eqref{existence_approx_b}, we get that for all $\lambda >0$, and all $t\geq 0$
$$
 \| \ddot{x}_{\lambda}(t) \| \leq M_1 + M_2.
$$
Since $\ddot{x}_{\lambda}(t)$ converges weakly to $\ddot{x}(t)$
as $\lambda \to 0$ (see the proof of Theorem \ref{basic_exist_thm}), we obtain
\begin{equation}\label{closed_loop_2_e}
\sup_{t \geq 0}  \| \ddot{x}(t) \|  < +\infty.
\end{equation}
Moreover,
by integrating \eqref{closed_loop_2_b}, we immediately obtain
$
\int_{0}^{+\infty}  \phi(\dot x(t))  dt <+\infty.
$
Let us summarize  the above results, and complete them, in the following proposition.
\begin{proposition}\label{preliminary_est}
Let  $x: [0, +\infty[ \to \cH$ be a solution trajectory of {\rm (ADIGE-V)}.
Then, the global energy $\mathcal E (t)= f(x(t)) -\inf_{\cH}f + \demi \| \dot{x}(t) \|^2 $ is non-increasing, and
$$
\sup_{t \geq 0}  \| \dot{x}(t) \|  < +\infty,  \quad \int_{0}^{+\infty}  \phi(\dot x(t))  dt <+\infty.
$$
Suppose moreover that $x$ is  bounded. Then
\begin{equation}\label{closed_loop_2_f}
 \sup_{t \geq 0}  \| \ddot{x}(t) \|  < +\infty.
\end{equation}
Suppose moreover that there exists $p\geq 1$, and $r>0$ such that, for all $u\in \cH$,    $\phi (u) \geq r \|u\|^p$. Then
\begin{equation}\label{dx-conv-0}
 \lim_{t\to +\infty}  \| \dot{x}(t) \|=0.
\end{equation}
\end{proposition}
\begin{proof}
We just need to prove the last point.
From  $\int_{0}^{+\infty}  \phi(\dot x(t))  dt <+\infty$ and
$\phi (u) \geq r \|u\|^p$, we get $\int_{0}^{+\infty} \| \dot{x}(t) \|^{p} dt <+\infty$.
This estimate, combined with $\sup_{t \geq 0}  \| \ddot{x}(t) \|  < +\infty$ classically implies that
$
 \lim_{t\to +\infty}  \| \dot{x}(t) \|=0.
$
\end{proof}

\medskip

Let us complete the above result by examining the convergence of the acceleration towards zero. To get this result, we need additional assumptions on the data $ f $ and $ \phi $.

\begin{proposition}\label{est_acceleration}
Let  $x: [0, +\infty[ \to \cH$ be a bounded solution trajectory of {\rm (ADIGE-V)}.
Suppose that  $f$ is a $\cC^2$ function, and that $\phi$ is a $\cC^2$ function which satisfies (i) and (ii):

\medskip

\noindent (i) (local strong convexity)
%there exists some $r>0$ and $\gamma >0$ such that for all $u\in \cH$ with $\|u\|\leq r$ the following inequality holds
there exists positive constants $\gamma >0$,   and  $\rho >0$ such that for all $u$ in
$\cH$ with $\|u\| \leq \rho$ the following inequality holds
$$
\left\langle \nabla^2 \phi (u) \xi, \xi \right\rangle \geq \gamma \|\xi\|^2    \quad   \mbox{for all } \xi \in \cH;
$$
(ii) (global growth)  there exist $p\geq 1$ and $r>0$ such that   $\phi (u) \geq r \|u\|^p$ for all $u\in \cH$.

\smallskip

\noindent Then, the following convergence property is satisfied:
\begin{equation}\label{acc}
 \lim_{t\to +\infty}  \| \ddot{x}(t) \|=0.
\end{equation}
\end{proposition}
\begin{proof}
Let us derivate (ADIGE-V), and set $w(t):= \ddot x(t)$. We obtain
$$
\dot w(t) + \nabla^2 \phi (\dot{x}(t))w(t) = - \nabla^2 f(x(t))\dot{x}(t).
$$
Take the scalar product of the above equation with $w(t)$. We get
$$
\demi \frac{d}{dt} \|w(t)\|^2  + \left\langle \nabla^2 \phi (\dot{x}(t))w(t), w(t) \right\rangle = - \left\langle\nabla^2 f(x(t))\dot{x}(t),w(t)\right\rangle.
$$
According to Proposition \ref{preliminary_est}, we have
$
 \lim_{t\to +\infty}  \| \dot{x}(t) \|=0.
$
From the local strong convexity assumption $(i)$, and the
Cauchy--Schwarz inequality, we deduce that for $t$ sufficiently large, say $t\geq t_1$
$$
\demi \frac{d}{dt} \|w(t)\|^2  + \gamma \| w(t)\|^2 \leq \|\nabla^2 f(x(t))\dot{x}(t)\| \|w(t)\|.
$$
Since $x(\cdot)$ is bounded and $\nabla f$ is Lipschitz continuous on bounded sets, we get that for some $C>0$
$$
\demi \frac{d}{dt} \|w(t)\|^2  + \gamma \| w(t)\|^2 \leq C\|\dot{x}(t)\| \|w(t)\|  \  \mbox{for all } t \geq t_1.
$$
After multiplication by $e^{2\gamma t}$,
and integration from $t_1$ to $t$, we get
$$
 \demi \left(e^{\gamma t}  \|w(t)\| \right)^2  \leq   \demi \left(e^{\gamma t_1}  \|w(t_1)\| \right)^2   +    C  \int_{t_1}^t
e^{\gamma \tau}\|\dot{x}(\tau)\|\left( e^{\gamma \tau}\|w(\tau)\| \right)d \tau.
$$
According to the Gronwall Lemma (see \cite[Lemma A.5]{Brezis}) we obtain
$$
e^{\gamma t}  \|w(t)\| \leq   e^{\gamma t_1}\|w(t_1)\|   +    C \int_{t_1}^t
e^{\gamma \tau}\|\dot{x}(\tau)\|d \tau.
$$
Therefore
$$
 \|\ddot x(t)\| \leq   \|\ddot x(t_1)\| e^{-\gamma (t-t_1)}  +    C e^{-\gamma t} \int_{t_1}^t
e^{\gamma \tau}\|\dot{x}(\tau)\|d \tau.
$$
Since $\lim_{t\to +\infty}  \| \dot{x}(t) \|=0,$ we have
$
\lim_{t\to +\infty} e^{-\gamma t} \int_{t_1}^t
e^{\gamma \tau}\|\dot{x}(\tau)\|d \tau=0.
$
Therefore, by passing to the limit on the above inequality we get
$\lim_{t\to +\infty}  \| \ddot{x}(t) \|=0.$
\end{proof}

\begin{corollary}\label{attractor}
Let us make the assumption of Proposition \ref{est_acceleration} and suppose that the trajectory $x(\cdot)$ is relatively compact.
Then for any sequence $x(t_n) \to x_{\infty}$ with $t_n \to +\infty$ we have  $\nabla f ( x_{\infty})=0$.
Set $S= \{x\in \cH: \, \nabla f ( x)=0 \}$. Therefore,
$$
\lim_{t \to + \infty}  d(x(t), S)=0.
$$
\end{corollary}

\begin{remark}{\rm
\textit{a)} Without geometric assumption on the function $f$, the trajectories of (ADIGE-v) may fail to converge.
In \cite{AGR}   Attouch--Goudou--Redont exhibit a function $f: \R^2 \to \R$ which
is $\mathcal C^1$, coercive, whose gradient is Lipschitz continuous on the bounded sets,
and such that the (HBF) system
admits an orbit $t \mapsto x(t)$ which does not converge as $t \to +\infty$.
The above result expresses that in such situation, the  attractor is the set  $S= \{x\in \cH: \, \nabla f ( x)=0 \}$.

\textit{b)} It is necessary to assume that $\phi$ is a smooth function in order to get the conclusion of Corollary \ref{attractor}. In fact in the case of the dry friction, that is $\phi (u)= r \|u\|$, there is convergence of the orbits to points satisfying
$ \|\nabla f ( x_{\infty})\| \leq r$, and which are not in general critical point of $f$.
}
\end{remark}

\subsection{Model example}  Consider the case
 $\phi (u)= \frac{r}{p}\|u\|^p$ with $p >1$, in which case the dynamic
 (ADIGE-V) writes
\begin{equation}\label{closed_loop_p}
\ddot x(t) +  r\|   \dot x(t)\|^{p-2} \dot x(t)+ \nabla f(x(t))=0.
\end{equation}
Therefore, for $p>2$, the viscous damping coefficient $\gamma(\cdot)$ that enters equation \eqref{closed_loop_p}, and which is equal to
\begin{equation}\label{g-speed}
\gamma (t) :=r \| \dot{x}(t) \|^{p-2}
\end{equation}
tends to zero as $t\to +\infty$. So,  we are in the setting of the inertial dynamics with vanishing damping coefficient. Consequently, in the associated inertial gradient algorithms, the extrapolation coefficient tends to $1$, and we can expect fast asymptotic convergence results.
To summarize, in the case of \eqref{closed_loop_p}, and for $f$ coercive, we have obtained that, for $p>2$
\begin{equation}\label{closed_loop_33}
 \lim_{t\to +\infty}  \gamma (t)=0, \quad \gamma (\cdot) \in L^{\frac{p}{p-2}}(0, +\infty).
\end{equation}
Are these informations sufficient to derive the convergence rate of the values, and obtain similar convergence properties as for the $\mbox{\rm (AVD)}_{\alpha}$  system?\\
To give a first answer to this question, we rely on the results of Cabot--Engler--Gaddat \cite{CEG}, Attouch--Cabot \cite{AC10} and Attouch--Chbani--Riahi \cite{ACR-Pafa} which concern the asymptotic stabilization of inertial gradient dynamics with general time-dependent viscosity coefficient $\gamma(t)$.
In the case of a vanishing damping coefficient, the key property which insures the asymptotic minimization property is that
$$
\int_{0}^{+\infty} \gamma (t)\, dt = +\infty.
$$
This means that the coefficient $\gamma(t)$ can go to zero as $t\to +\infty$, but not too fast in order to  dissipate the energy enough.
On the positive side, $\gamma (t) = \frac{\alpha}{t}$ does satisfy
the conditions \eqref{closed_loop_33} for any $p>0$, which does not exclude the Nesterov case.
On the negative side, we can easily find $\gamma(t)$ such that
\begin{equation}\label{closed_loop_333}
 \lim_{t\to +\infty}  \gamma (t)=0, \quad \gamma (\cdot) \in L^{\frac{p}{p-2}}(0, +\infty) \, \mbox{ and } \int_{0}^{+\infty} \gamma (t)\, dt < +\infty.
\end{equation}
So, without any other hypothesis, we cannot conclude from this information alone.
At this point, the idea is to introduce additional information,  assuming a geometric property on the function $f$ to minimize.
In the next two sections, we successively consider the case where $f$ is a strongly convex function, then the case of the functions $ f $ satisfying the Kurdyka--Lojasiewicz property.

%Except this situation, we are  able to give convergence results only in particular cases: the cases $p=1$ and $p=2$, and the linear case.

\section{The strongly convex case: exponential convergence rate}
\label{rate-f-str-conv}

We will study the asymptotic behavior of the system (ADIGE-V)
when  $f$ is a strongly convex function. Recall that $f: \cH \to \R$ is said to be  $\mu$-strongly convex (with $\mu >0$) if $f - \frac{\mu}{2}\| \cdot \|^2$ is convex.
Then, we will consider the particular case where $f$ is strongly convex and quadratic. Finally, we will give numerical illustrations in  dimension one.

\subsection{General  strongly convex function \textit{f}}\label{strong-convex}
\begin{theorem}\label{strong_convex_thm} Let $f: \cH \to \R$ be a differentiable function  which is $\mu$-strongly convex for some $\mu>0$, and whose gradient is  Lipschitz continuous on the bounded sets.  Let $\overline{x}$ be the unique minimizer of $f$.\\
Let $\phi : \cH \to \R_+$ be a damping potential (see Definition \ref{def1}) which is differentiable, and whose gradient is  Lipschitz continuous on the bounded subsets of $\cH$. Suppose  that $\phi$ satisfies the following growth conditions (i) and (ii):

\medskip

$(i)$ (local) there exists positive constants $\alpha$,   and  $\rho >0$ such that for all $u$ in
$\cH$ with $\|u\| \leq \rho$
$$     \left\langle \nabla \phi (u), u \right\rangle \geq \alpha \|u\|^2 .$$

$(ii)$ (global) there exists  $p\geq 1$, $c>0$, such that for all $u$ in $\cH$, $\phi (u) \geq c\|u\|^p$.

\medskip
 Then, for any   solution trajectory $x: [0, +\infty[ \to \cH$ of \mbox{\rm (ADIGE-V)}, we have  exponential
convergence rate to zero  as $t \to +\infty$  for $f(x(t))-f(\overline{x}) $, $\| x(t)-\overline{x}\|$ and the velocity $\|\dot x (t)\|$.
\end{theorem}
\begin{proof}
We will use the following inequalities which are attached to the strong convexity  of $f$:
\begin{eqnarray}
&& f(\overline{x})-f(x(t)) \geq \langle \nabla f(x(t)),\overline{x}-x(t)\rangle+\frac{\mu}{2}\|x(t)-\overline{x}\|^2  \label{from-str-conv1}  \\
&& f(x(t)) - f(\overline{x}) \geq \frac{\mu}{2}\|x(t)-\overline{x}\|^2. \label{from-str-conv2}
\end{eqnarray}
Let us consider the global energy (introduced in \eqref{energy}, in the preliminary estimates)
 $$
 \cE (t):=\frac{1}{2}\|\dot x(t)\|^2 + f(x(t)) - f(\overline{x}).
 $$
By Proposition \ref{preliminary_est},
$\dot x(t)$ is bounded on $\R_+$.
Moreover,  $\cE (\cdot)$ is non-increasing, and hence bounded from above. By definition of  $\cE (t)$, this implies that $f(x(t))$ is bounded from above. Since $f$ is strongly convex, it is coercive, which implies that $x(\cdot)$ is  bounded.
Since $x(\cdot)$ and $\dot{x}(\cdot)$ are  bounded, and  the vector fields $\nabla f $ and $\nabla \phi$ are locally Lipschitz continuous, we deduce from the constitutive equation
$
\ddot x(t) = -\nabla \phi(\dot x(t))- \nabla f(x(t))
$
that $\ddot x (\cdot)$
is also bounded.
According to the preliminary estimates established in Proposition \ref{preliminary_est}, we have
$
\int_0^{+\infty} \phi (\dot{x}(t)) dt < +\infty.
$
Combining this property with the global growth assumption $(ii)$ on $\phi$, we deduce that there exists $p\geq 1$ such that
$$
\int_0^{+\infty} \| \dot{x}(t)\|^p dt < +\infty.
$$
Since $\ddot x (\cdot)$
is  bounded,  this implies that
$\dot{x}(t) \to 0$ as $t \to + \infty$.
So, for $t$ sufficiently large, say $t\geq t_1$
$$
\| \dot{x}(t)\| \leq \rho.
$$
Time derivation of $\cE (\cdot)$, together with the constitutive equation (ADIGE-V),  gives for $t\geq t_1$
\begin{eqnarray} \dot \cE(t) & = & \langle \dot x(t), -\nabla\phi(\dot x(t))-\nabla f(x(t))\rangle + \langle \dot x(t),\nabla f(x(t))\rangle
\nonumber \\
&=& -\langle \dot x(t),\nabla \phi(\dot x(t))\rangle \nonumber \\
&\leq & -\alpha \|\dot x(t)\|^2, \label{strong_conv_1}
\end{eqnarray}
where the last inequality comes from the growth condition $(i)$ on $\phi$, and $\| \dot{x}(t)\| \leq \rho$ for $t\geq t_1$.\\
Since  $\dot{x}(\cdot)$ is  bounded,
let $L$ be the Lipschitz constant of $\nabla\phi$ on a ball that contains the velocity vector $\dot x(t)$ for all $t\geq 0$. Since $\nabla \phi (0) =0$ we have, for all $t\geq 0$
\begin{equation}\label{local_Lip}
\| \nabla \phi(\dot x(t))\| \leq L \| \dot x(t)\|.
\end{equation}
Using successively  (ADIGE-V), \eqref{local_Lip}
and \eqref{from-str-conv1}, we obtain
\begin{eqnarray}\frac{d}{dt}\Big(\langle x(t)-\overline{x},\dot x(t)\rangle \Big) &=&
\|\dot x(t)\|^2  + \langle x(t)-\overline{x},-\nabla \phi(\dot x(t))-\nabla f(x(t))\rangle \nonumber\\
&\leq & \|\dot x(t)\|^2 + L\|x(t)-\overline{x}\| \|\dot x(t)\|  - \langle x(t)-\overline{x},\nabla f(x(t))\rangle \nonumber\\
&\leq & \|\dot x(t)\|^2 + \frac{L^2}{2\mu}\|\dot x(t)\|^2 +\frac{\mu}{2} \|x(t)-\overline{x}\|^2 + \langle \overline{x} -x(t),\nabla f(x(t))\rangle \nonumber\\
&\leq & \left(1+\frac{L^2}{2\mu}\right)\|\dot x(t)\|^2 + f(\overline{x}) - f(x(t)). \label{strong_conv_2}
\end{eqnarray}
Take now $\epsilon >0$ (we will specify below how it should be chosen), and define
$$h_{\epsilon}(t) := \cE(t) + \epsilon \langle x(t)-\overline{x},\dot x(t)\rangle.$$
Time derivation of $h_{\epsilon}$, together with \eqref{strong_conv_1} and \eqref{strong_conv_2}, gives for $t\geq t_1$
$$\dot h_{\epsilon}(t) \leq -\left(\alpha-\epsilon\left(1+\frac{L^2}{2\mu}\right)\right)\|\dot x(t)\|^2
-\epsilon(f(x(t))-f(\overline{x})).$$
Choose $\epsilon>0$ such that $\alpha-\epsilon\left(1+\frac{L^2}{2\mu}\right)>0$. Set $C_1:= \min \{\alpha-\epsilon\left(1+\frac{L^2}{2\mu}\right);  \epsilon  \}$.
We deduce that
\begin{equation}\label{dot_h-str-conv}\dot h_{\epsilon}(t) \leq -C_1\Big(\|\dot x(t)\|^2+f(x(t))-f(\overline{x})\Big).\end{equation}
Further, from \eqref{from-str-conv2} and the Cauchy--Schwarz inequality we easily obtain
\begin{eqnarray*} h_{\epsilon}(t) &\leq & \frac{1}{2}\|\dot x(t)\|^2 + f(x(t)) - f(\overline{x}) +
\frac{\epsilon}{2}\|x(t)-\overline{x}\|^2+\frac{\epsilon}{2}\|\dot x(t)\|^2\\
&\leq & \left(\frac{1}{2}+\frac{\epsilon}{2}\right)\|\dot x(t)\|^2 + \left(1+\frac{\epsilon}{\mu}\right)(f(x(t))-f(\overline{x}))\\
&\leq & \left(1+\epsilon\left(\frac{1}{2}+\frac{1}{\mu}\right)\right)\Big(\|\dot x(t)\|^2 + f(x(t))-f(\overline{x})\Big).
\end{eqnarray*}
Combining this inequality with \eqref{dot_h-str-conv}, we obtain
$\dot {h}_{\epsilon}(t) + C_2 h_{\epsilon}(t)\leq 0,$
with $C_2:= \frac{C_1}{1+\epsilon\left(\frac{1}{2}+\frac{1}{\mu}\right)}    >0$. Then, the Gronwall inequality classically implies
\begin{equation}\label{rate_h-str-conv}h_{\epsilon}(t) \leq h_{\epsilon}(0)e^{-C_2t}.\end{equation}
Finally, from \eqref{from-str-conv2} and the Cauchy--Schwarz inequality we have
\begin{eqnarray*}h_{\epsilon}(t) &\geq & \frac{1}{2}\|\dot x(t)\|^2 + f(x(t)) - f(\overline{x})
-\frac{\epsilon}{2}\|x(t)-\overline{x}\|^2 - \frac{\epsilon}{2}\|\dot x(t)\|^2\\
&\geq & \left(\frac{1}{2}-\frac{\epsilon}{2}\right)\|\dot x(t)\|^2 + \left(1-\frac{\epsilon}{\mu}\right)(f(x(t))-f(\overline{x})).\end{eqnarray*}
Therefore, by taking $\epsilon$ small enough, we obtain the existence of  $C_3>0$ such that
$$h_{\epsilon}(t)\geq C_3 \Big(\|\dot x(t)\|^2 + f(x(t))-f(\overline{x})\Big).$$
Combining this inequality  with  \eqref{rate_h-str-conv} and  \eqref{from-str-conv2}, we obtain an exponential
convergence rate to zero for $f(x(t))-f(\overline{x}) $, $\|x(t)-\overline{x}\|$ and the velocity $\|\dot x (t)\|$.
\end{proof}

%\begin{remark} Since in Proposition \ref{preliminary_est} it has been proved that $\dot x$ is bounded,
%it seems that instead of asking global Lipschitz for $\nabla \phi$, it is enough to assume that
%$\nabla\phi$ is Lipschitz continuous on bounded sets.
%\end{remark}

\medskip

\begin{remark}
In subsection \ref{Sec:KL_poly_growth}, as a consequence of the Kurdyka-Lojasiewicz theory, we will extend the above results to the case where we  only assume a quadratic growth assumption
$$
f(x) -\inf\nolimits_{\cH} f \geq c\dist (x, \argmin f)^2.
$$
\end{remark}
\begin{remark} In section \ref{sec:weakdamping}, we will give indications concerning the case of a general convex function $f$, whose solution set $\argmin f$ is nonempty.
Let us recall that, in the case of (HBF), which corresponds to $\phi(u)= r \|u\|^2$, each trajectory converges weakly and its limit belongs to $\argmin f$.
Apart from this important case, the convergence of the trajectories depends both on the geometric properties of the function to be minimized $f$ and on those of the damping potential $\phi$. Precisely in  section \ref{sec:weakdamping} we will give an example in dimension one, with trajectories which do not converge.
\end{remark}

\subsection{Case \textit{f} convex quadratic positive definite}\label{f_convex_quadratic}

Let us make precise the previous results in the case
$
f(x) = \demi \left\langle Ax, x \right\rangle,
$
where $A :\cH \to \cH$ is a linear continuous positive definite self-adjoint operator.
%In this case, we can specify the crucial role played by the growth conditions on the damping potential $\phi$,  which were made in Theorem \ref{strong_convex_thm}.
Then
$
\nabla f(x)= Ax,
$
and  (ADIGE-V) is written
\begin{equation}\label{closed_loop_1_bbb}
 \ddot{x}(t) + \partial \phi ( \dot{x}(t))  +  A(x(t)) \ni 0.
\end{equation}
Let us  prove the following  ergodic convergence result, valid for a general damping potential $\phi$.

\begin{theorem}\label{edp} Let
$x:[0,+\infty[\to\mathcal H$ be a solution trajectory of \eqref{closed_loop_1_bbb},
where $\phi$ is a damping potential, and  $A :\cH \to \cH$ is a linear continuous positive definite self adjoint operator.
Then, we have the following ergodic convergence result for the weak topology: as $t \to  +\infty$,
$$
\frac{1}{t} \int_0^t   x(\tau) d \tau \rightharpoonup x_{\infty},
$$
where the limit $x_{\infty}$ satisfies
%\begin{equation}\label{first_order_cl_loop_0_linear_limit}
$0\in \partial \phi(0)+ Ax_{\infty}.$\\
%\end{equation}
When $\phi$ is differentiable at the origin, we have $Ax_{\infty}=0$, that is $x_{\infty}=0$.\\
When $\phi (x)= r \|x\|$, we have $\| Ax_{\infty}\|\leq r.$
\end{theorem}
\begin{proof}
The Hamiltonian formulation of \eqref{closed_loop_1_bbb} gives
 the equivalent first-order differential
inclusion in the product space $\mathcal H\times \mathcal H$:
\begin{equation}\label{first_order_cl_loop_0_linear}
0\in\dot z(t)+\partial \Phi(z(t))+F(z(t)),
\end{equation}
where $z(t)=(x(t), \dot x(t)) \in \mathcal H\times \mathcal H $, and
\begin{itemize}
\item
$\Phi:\mathcal H\times \mathcal H\to \R$ is the convex continuous function defined by $\Phi(x,u)=\phi(u)$

\item   $F: \mathcal H\times \mathcal H\to \mathcal H\times \mathcal H$
is defined by \,  $F(x,u)=(-u , Ax)$.
\end{itemize}
The trick is to renorm the product space $\mathcal H\times \mathcal H$ as follows:

\noindent  The mapping $(x,y) \mapsto  \left\langle Ax, y \right\rangle$
defines a scalar product on $\mathcal H$, which is equivalent to the initial one.
Accordingly, let us equip the  product space $\mathcal H\times \mathcal H$ with the scalar product
$$
 \left\langle \left\langle (x_1 u_1), (x_2 ,u_2) \right\rangle \right\rangle := \left\langle A x_1, x_2 \right\rangle  + \left\langle u_1, u_2 \right\rangle
$$
With respect to this new scalar product, let us observe that:
\begin{itemize}
\item $F: \mathcal H\times \mathcal H\to \mathcal H\times \mathcal H$ is a linear continuous skew symmetric operator. Since $A$ is self-adjoint
\begin{eqnarray*}
\left\langle \left\langle F(x,u), (x,u) \right\rangle \right\rangle &=& \left\langle \left\langle (-u,Ax), (x,u) \right\rangle \right\rangle \\
&=& -\left\langle Au, x \right\rangle + \left\langle Ax, u \right\rangle =0.
\end{eqnarray*}
\item The subdifferential of $\Phi$ is unchanged that is $\partial\Phi (x,u) = (0, \partial \phi (u)).$
\end{itemize}
\noindent Therefore, the differential inclusion \eqref{first_order_cl_loop_0_linear} is governed by the sum of two maximally monotone operators, one of them is the subdifferential of a convex continuous function, the other one is a monotone skew symmetric operator.
By the classical Rockafellar theorem (see \cite[Corollary 24.4]{BC}), their sum is still maximally monotone.
Consequently, we can apply the  theory concerning the semi groups generated by general maximally monotone operators, and conclude that $z(t)$ converges weakly and in an ergodic way towards
 a zero
$z_{\infty}= (x_{\infty}, u_{\infty}) $ of  $\partial \Phi+F$. This means
$$
(0, \partial \phi (u_{\infty})) + (-u_{\infty}, Ax_{\infty}) =(0,0).
$$
Equivalently $u_{\infty} =0$ and $\partial \phi (0) + Ax_{\infty}\ni 0$.
\end{proof}

\medskip

In the case of the wave equation, this type of argument was developed by Haraux in \cite[Lecture 12,  Theorem 45]{Haraux_PDE}.
A recent account on these questions can be found in
Haraux-Jendoubi \cite{HJ2} and
Alabau Boussouira-Privat-Tr\'elat \cite{APT}.

\subsection{Numerical illustrations}\label{sec:num}
Finding explicit solutions in a closed form of nonlinear oscillators   has  direct applications in various fields.
In the one-dimensional case, the corresponding second-order differential equation
 $
 \ddot x(t) +  d(x(t), \dot x(t))\dot x(t)+  g(x(t))=0
 $
 is known as the Levinson-Smith equation. It is reduced to the  Li\'enard equation when $d$ depends only on $x$.
One can consult \cite{GCG, GGH} for recent reports on the subject and the description of some of the different techniques  developed to resolve these questions.
  In our setting, we will provide some insight on this question by combining energetic and topological arguments.

\subsubsection{A numerical one-dimensional example}
Consider the case $\cH = \R$, $f(x) = \demi |x|^2$, and $\phi (u) = \frac{1}{p} |u|^p$ with $p>1$.
Then, (ADIGE-V) writes
\begin{equation}\label{adige_v_one_dim}
\ddot x(t) +  |\dot x(t)|^{p-2} \dot x(t)+ x(t)=0.
\end{equation}
It is a linear oscillator with nonlinear damping. According to the previous results,

\medskip

$\bullet$ For $p=2$, according to the strong convexity of the potential function $f(x)=\demi |x|^2$ and  Theorem \ref{strong_convex_thm}, we have  convergence at an exponential rate of $x(t)$ and $\dot{x}(t)$ toward $0$.
Indeed, $p=2$ is the only value of $p$ for which the hypothesis of
Theorem \ref{strong_convex_thm} are satisfied. For $p>2$ the  local hypothesis $(i)$ is not satisfied, and for $p<2$ the gradient of $\phi$ fails to be Lipschitz continuous on the bounded   sets containing the origin.

\medskip

$\bullet$  For $p>1$,  let us first show that
$
\lim_{t \to +\infty} \dot{x}(t) =0.
$
This results from Proposition \ref{preliminary_est}, and the fact that the trajectory is bounded. This last property  results from
 the fact that
the global energy
$
\mathcal E (t)= \demi |\dot{x}(t) |^2 +  \demi |x(t) |^2
$
is non-increasing, and hence convergent (and bounded from above).

Let us  show that $x(t)$ tends to zero.
Since $\lim_{t \to +\infty} \dot{x}(t) =0$, and $\lim_{t \to +\infty}  \mathcal E (t)$ exists, we have
\begin{equation}\label{conv _dim_1}
\lim_{t \to +\infty} |x(t) |^2 = \lim_{t \to +\infty}  \mathcal E (t) \quad \mbox{\rm exists}.
\end{equation}
Since the identity operator clearly satisfies the assumptions of Theorem \ref{edp}, we have the following ergodic convergence result
$
 \lim_{t \to +\infty} \frac{1}{t} \int_0^t   x(\tau) d \tau =0.
$
 There are two possibilities:

\textit{a)} For $t$ sufficiently large, $x(t)$ has a fixed sign.
According to  \eqref{conv _dim_1}, we get that   $\lim_{t \to +\infty} x(t):= x_{\infty} \, \mbox{ exists}.$
The convergence implies the ergodic convergence. Therefore,
$\lim_{t \to +\infty} \frac{1}{t} \int_0^t   x(\tau) d \tau = x_{\infty}. $
 But we know that the ergodic limit is zero, hence $x_{\infty}=0. $

\medskip

\textit{b}) The trajectory changes sign an infinite number of time as $t\to +\infty$. This means that there exists sequences $s_n$ and $t_n$ which tend to infinity such that  $x(t_n) x(s_n) <0$.
Since the trajectory is continuous, by the mean value theorem, this implies the existence of $\tau_n \in[s_n, t_n]$ such that $x(\tau_n)=0$.
Hence $x(\tau_n)^2=0$ for all $n\in \N$, with $\tau_n \to +\infty$. Since $\lim_{t \to +\infty} |x(t) |^2$ exists, this implies that
$\lim_{t \to +\infty} |x(t) |^2 =0$.
Clearly, this implies that $\lim_{t \to +\infty} x(t) =0$.

So, for $p>1$, for any solution trajectory of
\eqref{adige_v_one_dim}, we have:
\begin{equation}\label{conv _dim_2}
\lim_{t \to +\infty} x(t) =0 \, \, \mbox{ and }\, \lim_{t \to +\infty} \dot{x}(t) =0.
\end{equation}
%To obtain this result, we used the mean value theorem which relies in an essential way on the fact that we work in the one dimensional case.
Now let's analyze how the trajectories and their speeds go to zero. %In fact, the situation differs notably according to $ p> 2 $ or $ p <2 $.
 As we shall see, the case $p>2$ corresponds to a weak damping, while the case $p<2$ corresponds to a strong damping.

\smallskip

\textbf{Case $p>2$.}
Since the speed $| \dot x(t)|$ tends to zero, we have
$\gamma (t):= | \dot x(t)|^{p-2} \to 0 \, \mbox{ as } \, t\to +\infty.$
\begin{figure}[h]
	\centering
	%\captionsetup[subfigure]{position=top}
	%\subfloat[...]
	{\includegraphics*[viewport=78 200  540 600,width=0.325\textwidth]{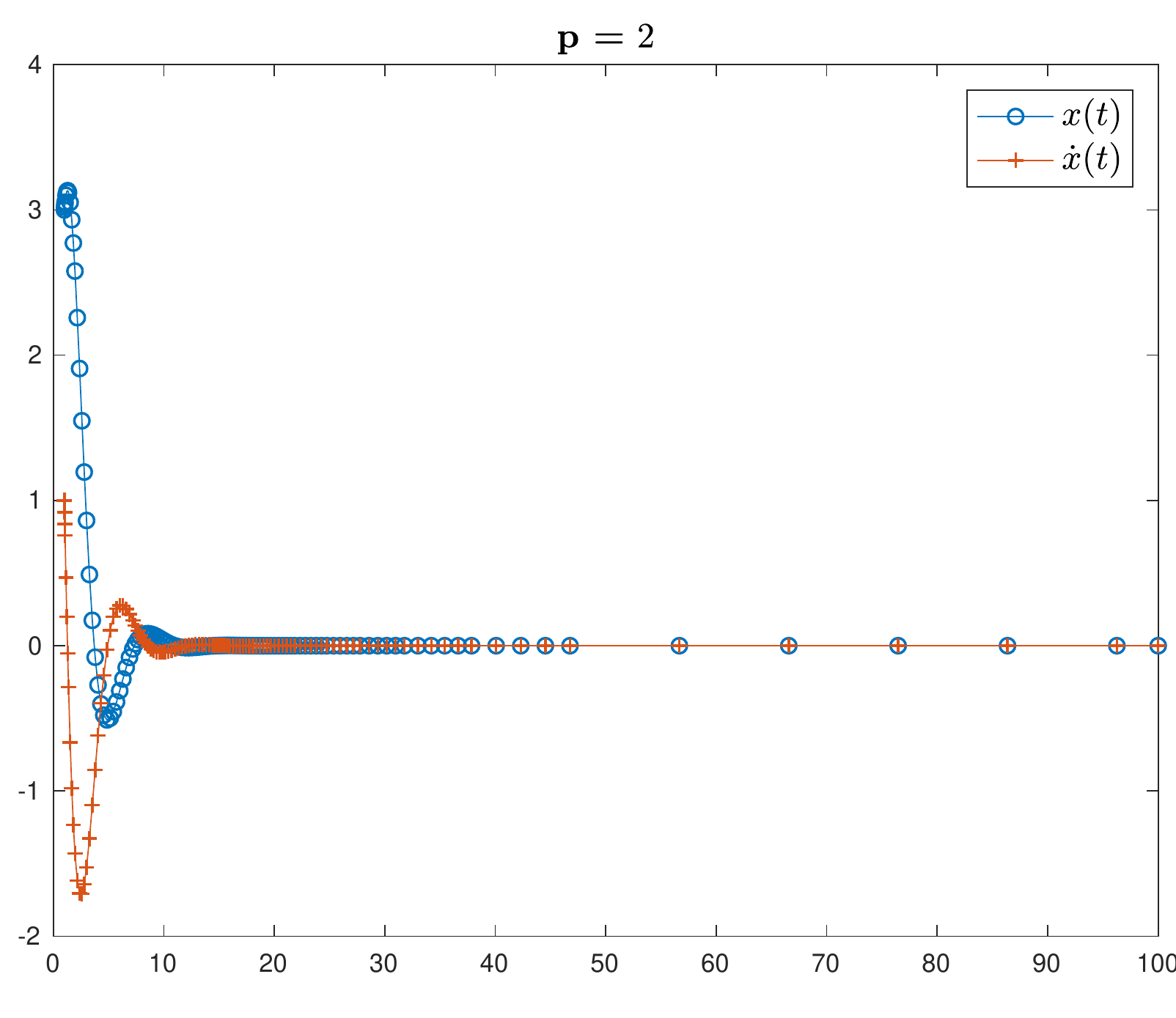}}\hspace{0.03cm}
	%\subfloat[...]
	{\includegraphics*[viewport=78 200  540 600,width=0.325\textwidth]{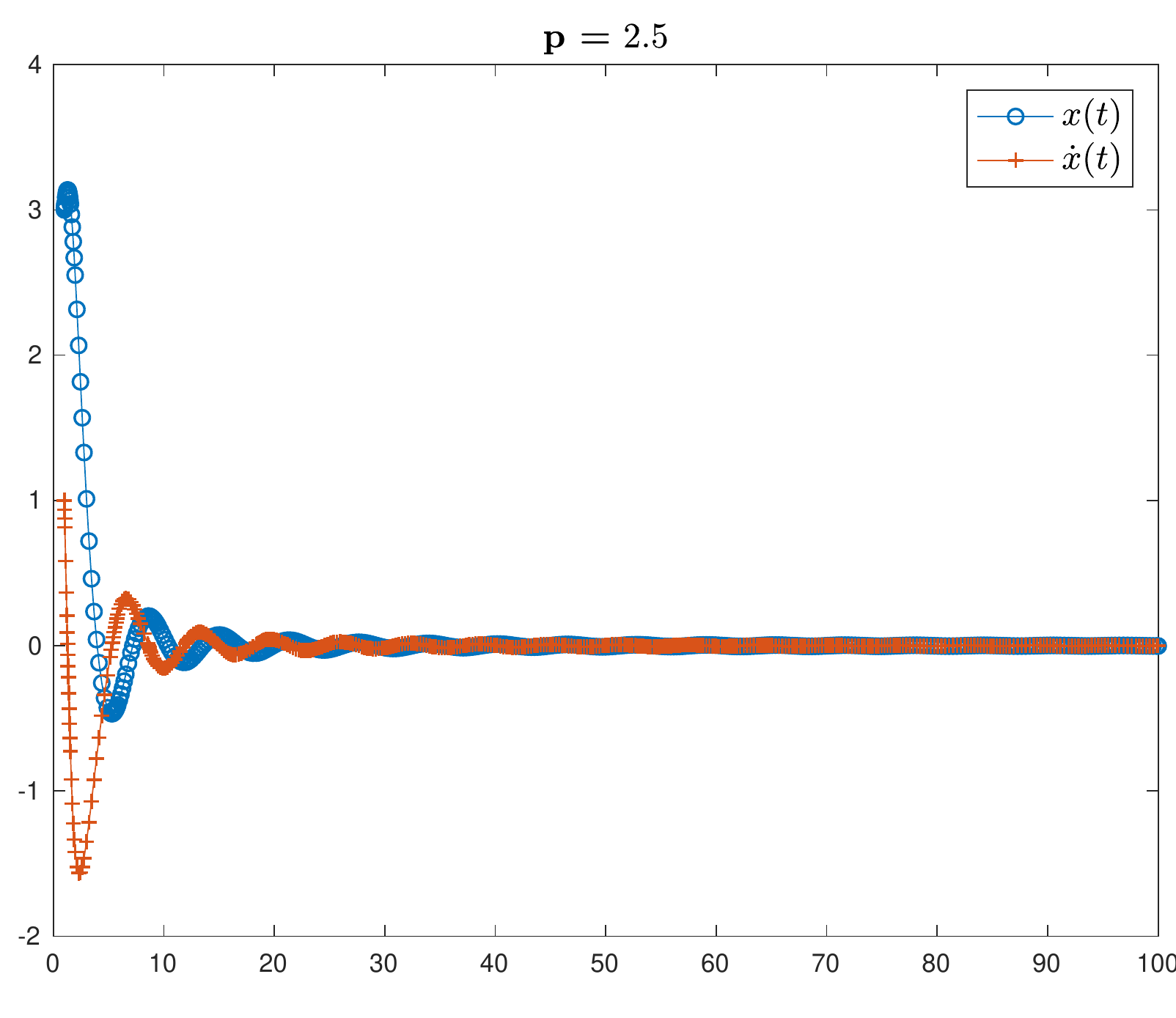}}\hspace{0.03cm}
	%\subfloat[...]
	{\includegraphics*[viewport=78 200  540 600,width=0.325\textwidth]{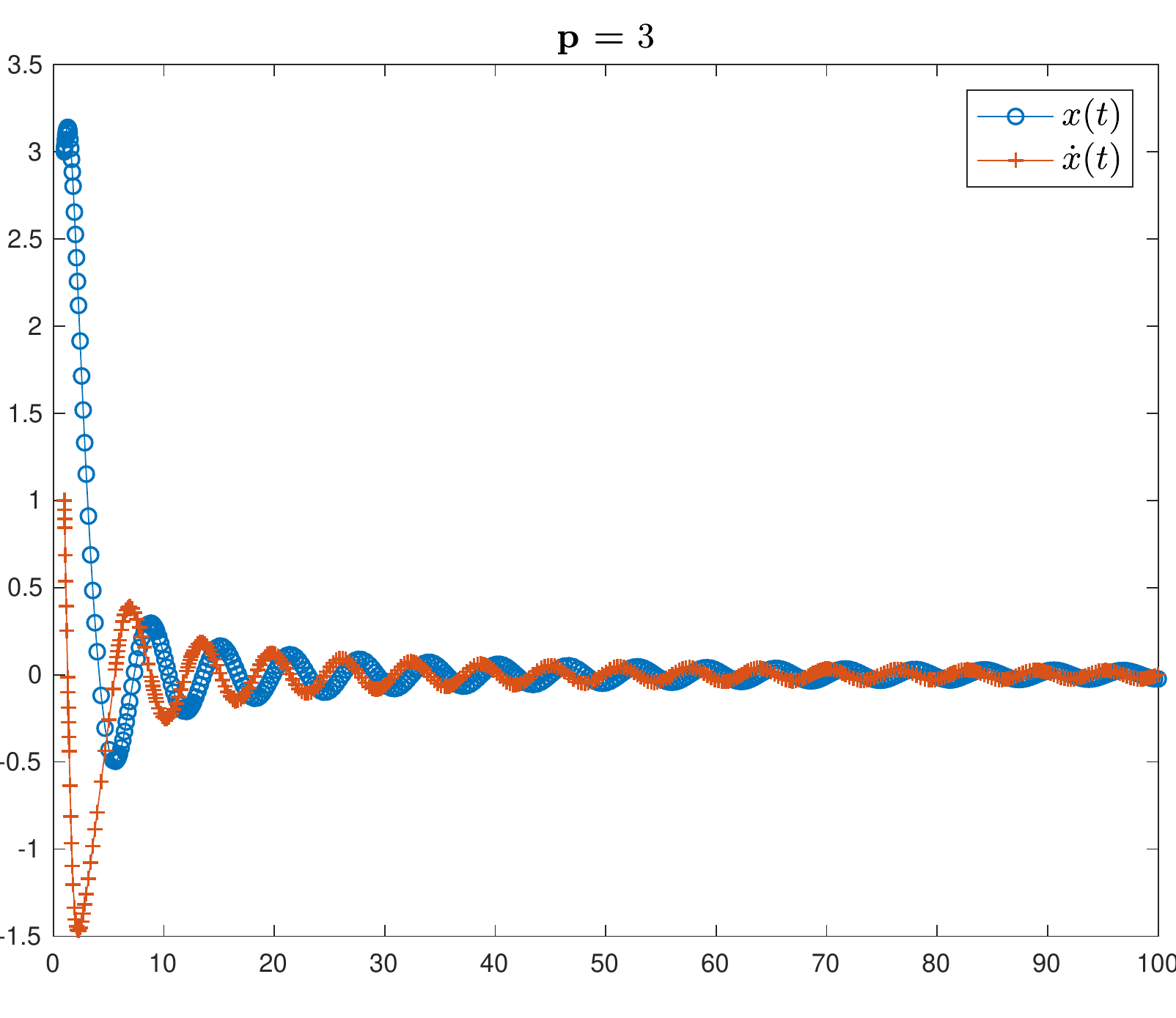}}\\
	{\includegraphics*[viewport=78 200  540 600,width=0.325\textwidth]{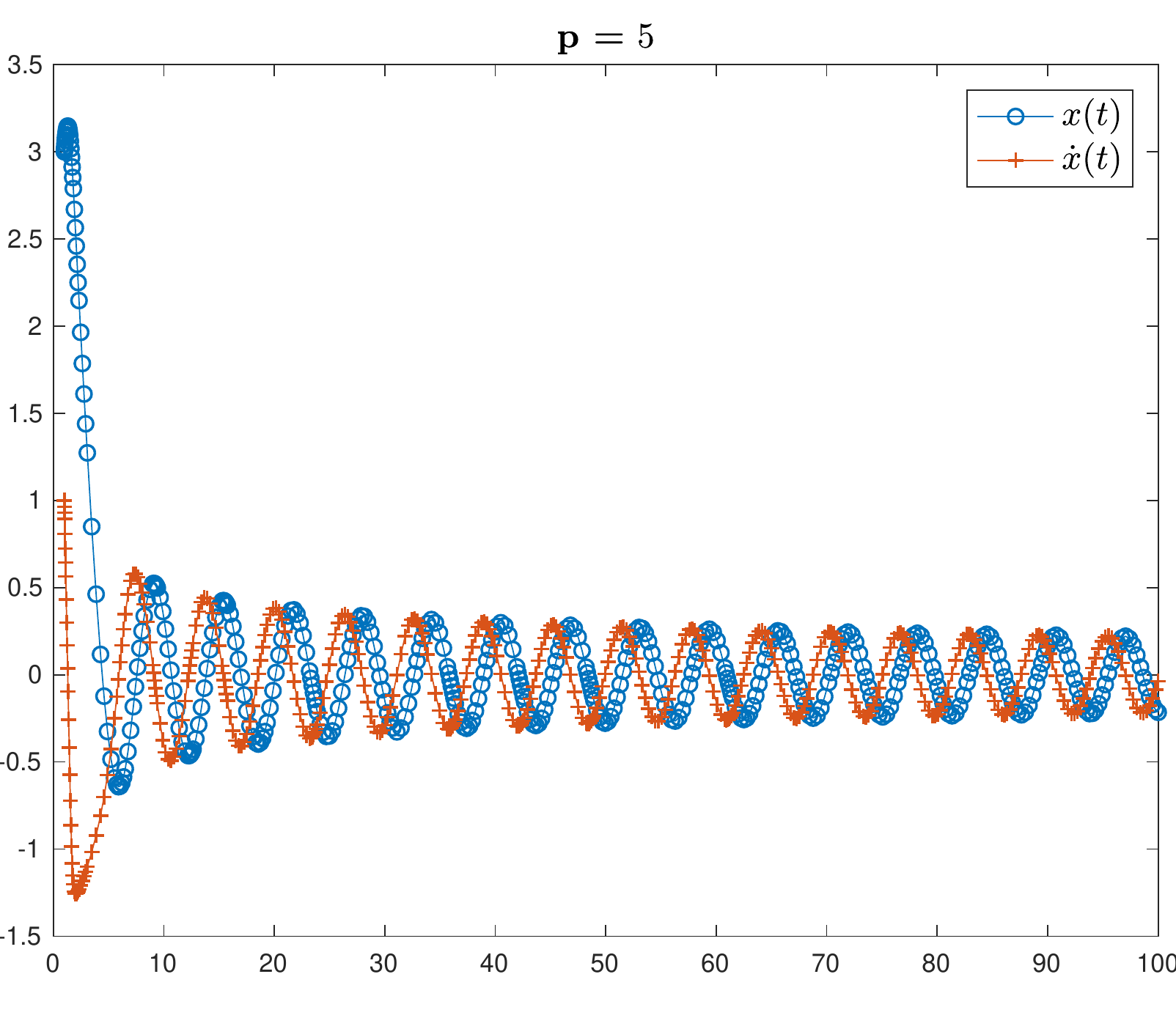}}\hspace{0.03cm}
	%\subfloat[...]
	{\includegraphics*[viewport=78 200  540 600,width=0.325\textwidth]{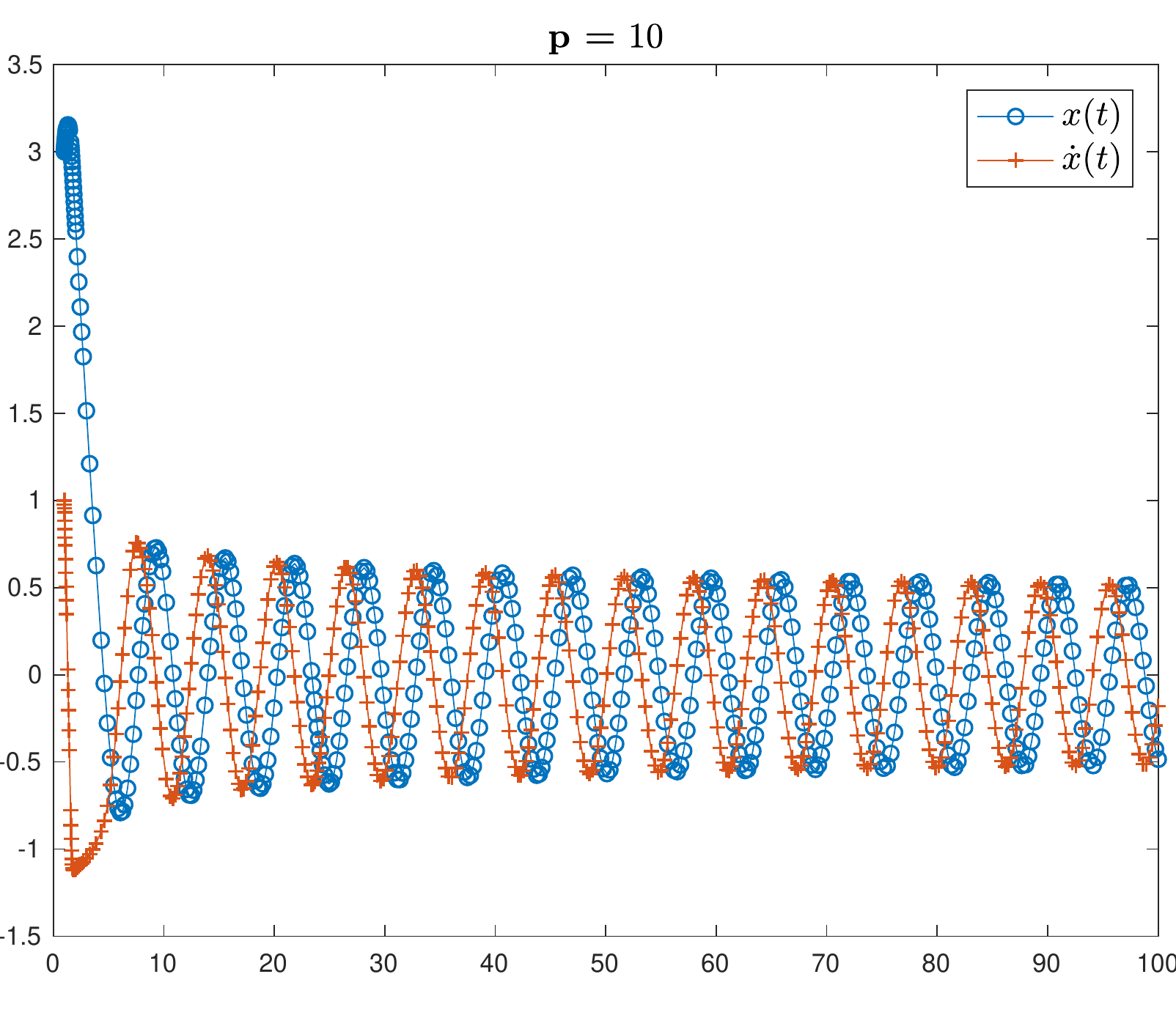}}\hspace{0.03cm}
	%\subfloat[...]
	{\includegraphics*[viewport=78 200  540 600,width=0.325\textwidth]{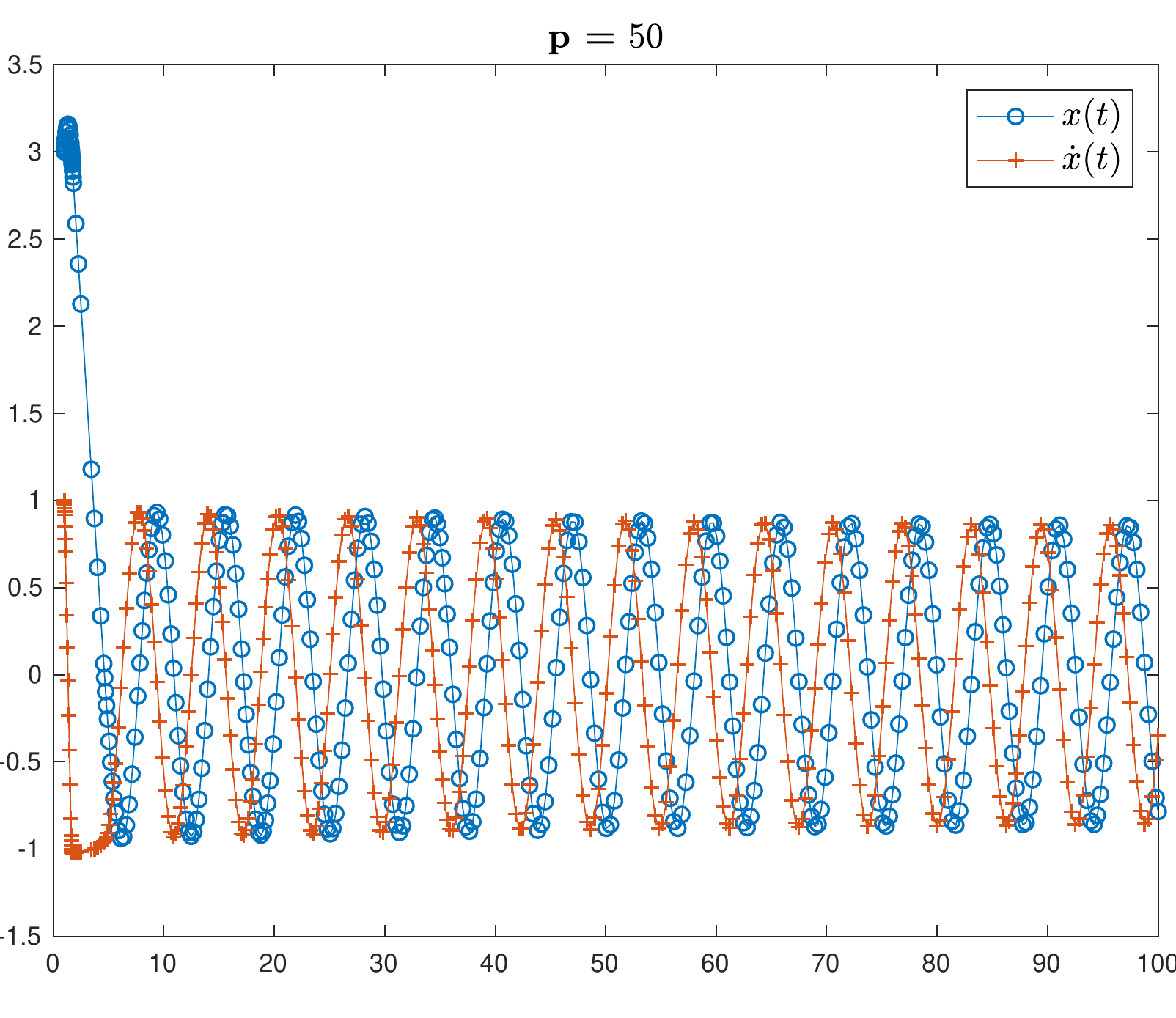}}
	\caption{\small The evolution of the trajectories $x(t)$ (blue line) and $\dot x(t)$ (red line) of the dynamical system \eqref{adige_v_one_dim} for different values of $p \geq 2$.
}
	\label{fig:ex0}	
\end{figure}
The viscous damping coefficient $\gamma(t)$ becomes asymptotically small.  As a result, the damping effect also becomes  weak (that's what we call weak damping). As $p$ increases, the damping effect tends to decrease,  the trajectory tends to oscillate more and more, and the rate of convergence  deteriorates.\\
This is illustrated  in Figure \ref{fig:ex0}, where we can see the evolution of the trajectory $x(t)$ (blue line) and  its derivative $\dot x(t)$ (red line) of the dynamical system \eqref{adige_v_one_dim} with starting point $(x(0), \dot x(0)) = (3,1)$ and for different values of $p$ greater than or equal to $2$. The trajectory and the velocity tend towards zero,  however, the oscillations become stronger as $p$ increases, and the convergence towards zero becomes very slow.
For large $p$, the oscillatory aspect  conforms to  the ergodic convergence of the trajectory to $0$ (indeed in dimension one the trajectory tends towards zero, but we can expect that in higher dimensions there is only ergodic convergence towards zero).
Note also that for $p>2$, and $p$ close to $2$, the trajectory is close to that corresponding to $p=2$, and therefore enjoys excellent convergence  properties. It would be  interesting  to  study  this situation, because it is a natural candidate to obtain convergence rates similar to the  accelerated gradient method of Nesterov.

%\vspace{8mm}

\textbf{Case $1<p<2$}.
According to \eqref{conv _dim_2} we have
$\lim_{t \to +\infty} \dot{x}(t) =0,$
and
$\lim_{t \to +\infty} x(t) =0.$\\
 Since $\lim_{t \to +\infty} \dot{x}(t) =0$ and $2-p >0$,  we have that the viscous damping coefficient
 $$\gamma (t):= \frac{1}{|\dot x(t)|^{2-p}} \to +\infty \, \mbox{ as } \, t\to +\infty .$$
We are in the setting of a strong damping effect.
This situation was analyzed in the following result of \cite{AC1}, which we reproduce here. It concerns the asymptotic behaviour of

$$
{\rm(IGS)_{\gamma}} \quad \ddot x(t) + \gamma (t) \dot x(t) + \nabla  f(x(t))=0.
$$

\begin{proposition}\label{large_gamma}
Let $f:\cH\to \R$ be a function of class $\cC^1$ such that $\nabla f$ is bounded on the bounded subsets of $\cH$. Given $r>0$ and $\theta>1$, assume that $\gamma(t)=r\, t^\theta$ for every $t\geq t_0\geq 0$.\\ Then each bounded solution trajectory $x(.)$ of $\rm{(IGS)_{\gamma} }$ satisfies $\int_{t_0}^{+\infty}\|\dot x(t)\|\, dt <+\infty$, and hence converges strongly toward some $x^*\in \cH$.
\end{proposition}
According to this result, we can obtain some information about the convergence rate of the velocity to zero.
We have two cases: either
$\int_{t_0}^{+\infty}\|\dot x(t)\|  \, dt <+\infty$,
or $\int_{t_0}^{+\infty}\|\dot x(t)\|  \, dt = +\infty$.
In this last case, according to Proposition \ref{large_gamma}, we cannot have $\gamma(t)= \frac{1}{|\dot x(t)|^{2-p}}  $  of order $r\, t^\theta$ with $\theta>1$.
This excludes the possibility to have $|\dot x(t)|$ or order
$\frac{1}{t^\frac{\theta}{2-p}}$ with $\theta>1$.
So, the best that we can expect is,
$
|\dot x(t)| \sim \frac{1}{t^\frac{1}{2-p}} \, \, \mbox{ as} \; \;  t \to +\infty.
$
This estimate is in accordance with the exponential decay when $p=2$, and the finite length property when $p=1$.
We emphasize the fact that the above argument is not a rigorous proof, it just gives an indication of the type of convergence rate that we can expect.

\smallskip

\begin{figure}[h]
	\centering
	%\captionsetup[subfigure]{position=top}
	%\subfloat[...]
	{\includegraphics*[viewport=78 200  540 600,width=0.325\textwidth]{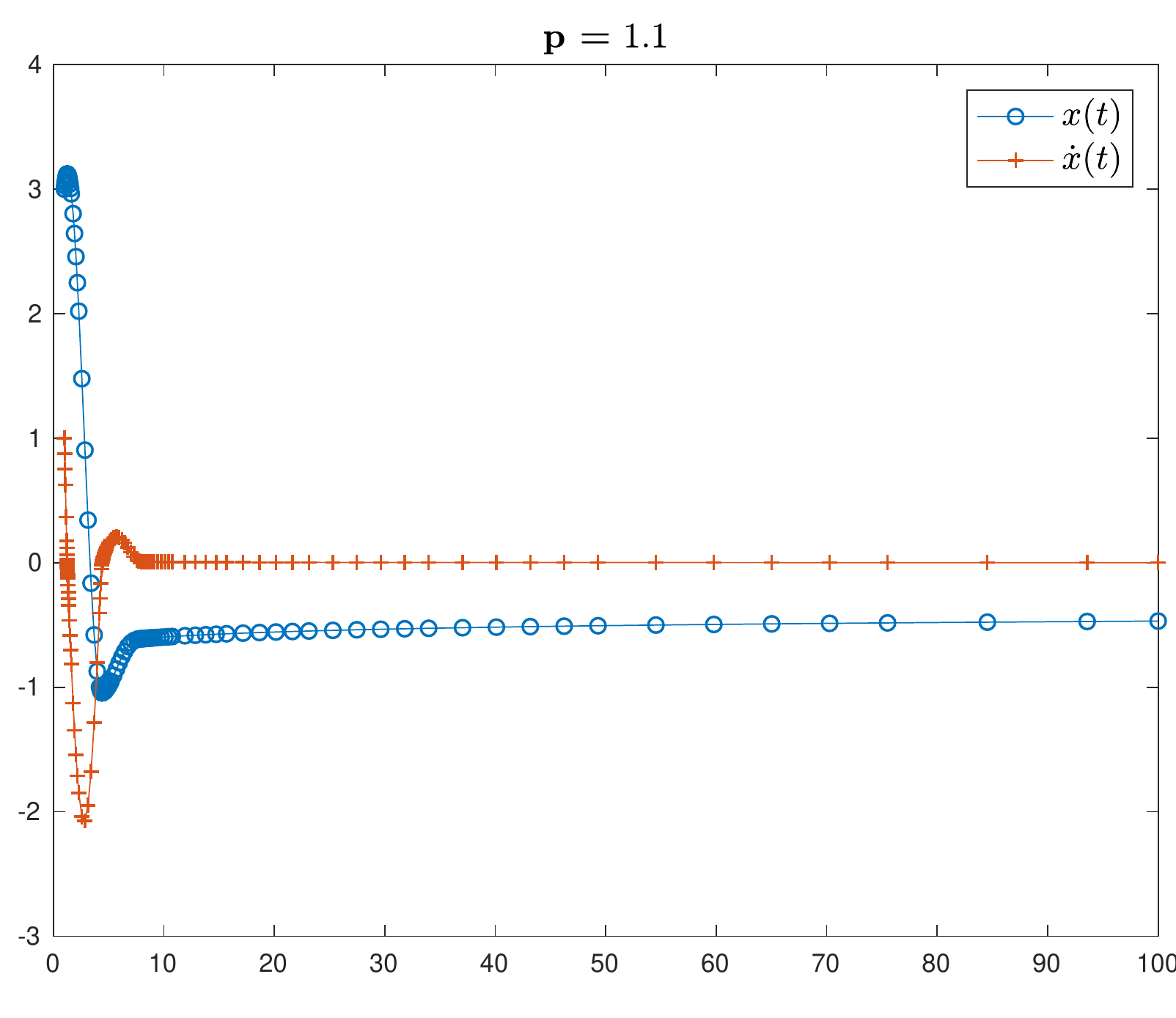}}\hspace{0.03cm}
	%\subfloat[...]
	{\includegraphics*[viewport=78 200  540 600,width=0.325\textwidth]{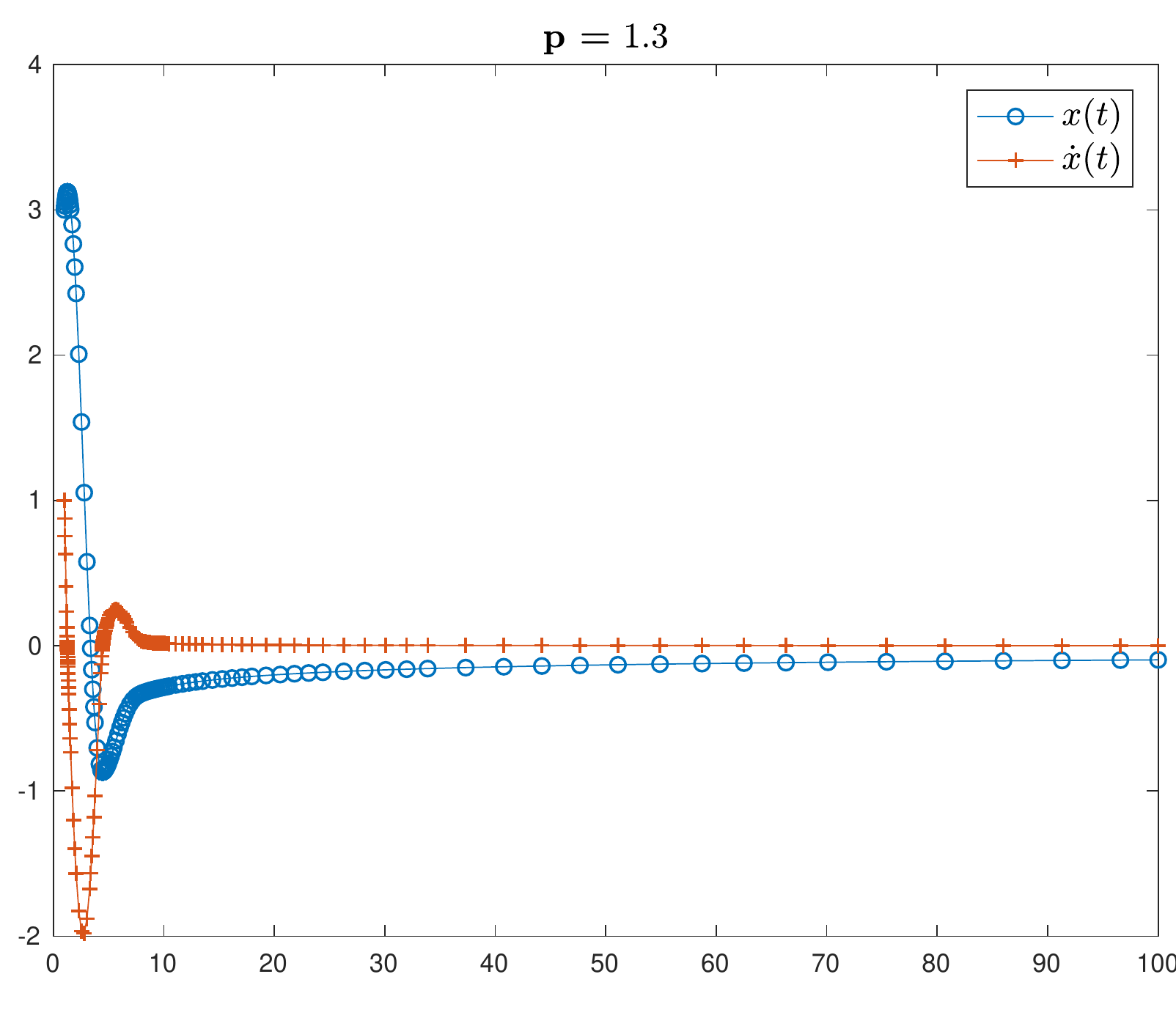}}\hspace{0.03cm}
	%\subfloat[...]
	{\includegraphics*[viewport=78 200  540 600,width=0.325\textwidth]{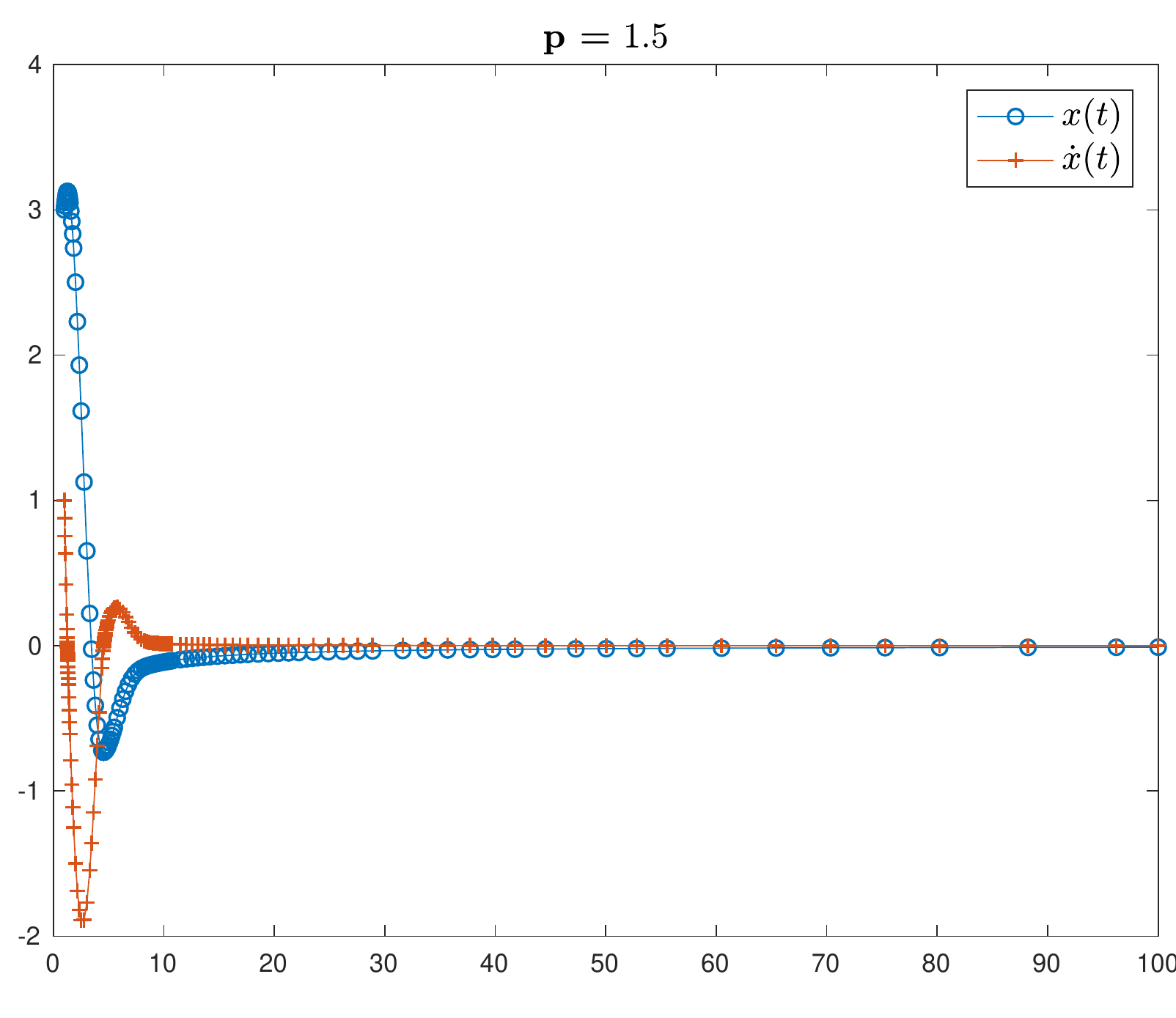}}
	\caption{\small Evolution of  $x(t)$ (blue) and $\dot x(t)$ (red) for different values of $1 < p < 2$.
}
	\label{fig:ex2}	
\end{figure}
In Figure \ref{fig:ex2} we can see the evolution of the trajectory $t \mapsto x(t)$ (blue line) and of its derivative $t \mapsto \dot x(t)$ (red line) of the dynamical system \eqref{adige_v_one_dim} with starting point $(x(0), \dot x(0)) = (3,1)$ for different values of $1  < p < 2$.
 Because of the strong damping, the trajectories exhibit small oscillations, and the velocity converges fastly to zero.
By contrast, the convergence of $x(t)$ to zero highly depends on the parameter $p$.
When $p$ is close to $1$, the convergence of the trajectory to zero is poor, however, already a slight increase of $p$ concisely improves the convergence of the trajectory.
Indeed, when $p$ becomes large the convergence of the trajectory improves.

\vspace{3mm}

\section{Weak damping: from slow convergence to attractor  effect}
\label{sec:weakdamping}
As we  already noticed, even in the case of a strongly convex function $ f $, when the damping effect becomes too weak, then the convergence property is degraded. In the case of the damping $ \|  \dot x (t) \| ^ {p-2} \dot x (t) $), this corresponds to situations where $ p >2 $.
In this section, we  give examples showing that in the case of a general convex function, the situation is even worse, and the trajectory may not converge in the case of weak damping.
In this case, one has to replace the convergence notions by the concept of attractor, a central subject for the theory of dynamic systems, and PDE's, see Hale \cite{Hale}, Haraux \cite{Haraux} for seminal contributions to the subject in the case of gradient systems (\ie systems for which there exists a Lyapunov function). For optimization purposes, this is a promising research topic, largely to be explored in the case of a general damping function.
In the next section, we take a convex function $ f $ with a continuum of minimizers, and examine the lack of convergence when the damping becomes too weak.
In fact, as we have already underlined, convergence depends both on the geometric properties of the damping potential and on the potential function $ f $ to be minimized. The corresponding geometric aspects concerning $ f $ will be examined a little later.

\subsection{An example where convergence fails to be satisfied}
The following example is based on Haraux \cite[section 5.1]{Haraux}.
Take $\cH =\R$, and $f: \R \to \R$ a convex function of class $\cC^1$ which achieves its minimal value on the line segment $[a,b]$, with $a<b$.
We suppose that $f$ is coercive, \ie $\lim_{|x|\to+\infty} f(x)=+\infty$. Its graph looks like a bowl with a flat bottom.

\noindent Consider the evolution equation with closed-loop damping
\begin{equation}\label{closed_loop_1_counter_example}
 \ddot{x}(t) + | \dot{x}(t) |^{p-2} \dot{x}(t)  +  \nabla f (x(t)) = 0.
\end{equation}
Let us discuss, according to the value of $ p $, the convergence properties of the trajectories of this system. We will need the following elementary lemma, see \cite[Lemma 5.1.3]{Haraux}.

\begin{lemma}\label{lem_tech} Let $v \in \cC^2 (\R_+)$ which satisfies, for some $c>0$
$$
\dot{v}(0) >0; \quad  \ddot{v}(t) \geq -c \dot{v}(t)^2 \, \mbox{ for all } \, t \geq 0.
$$
Then, $v$ is an increasing function, and \,
$
\lim_{t\to +\infty} v(t) = + \infty.
$
\end{lemma}
\begin{proof} As long as $\dot{v}(t) >0$,  by integration of the  differential inequality $ \ddot{v}(t) +c \dot{v}(t)^2 \geq 0$, we obtain
$$
\dot{v}(t) \geq \frac{1}{ct + \frac{1}{\dot{v}(0)}}.
$$
This immediately implies that $\dot{v}(t) >0$ for all $t\geq 0$.
By integrating the above inequality, we obtain
$$
v(t) \geq v(0) + \int_0^t \frac{1}{c\tau + \frac{1}{\dot{v}(0)}} d\tau
$$
which implies
$\lim_{t\to +\infty} v(t) = + \infty.$
\end{proof}

\begin{proposition}\label{prop_counterex}
Suppose that $p \geq 3$.
Then, any solution trajectory of \eqref{closed_loop_1_counter_example}
which is not constant, passes an infinitely of times through the points $a$ and $b$.
\end{proposition}

\begin{proof}
According to Proposition \ref{preliminary_est},  the trajectory $x(\cdot)$ is bounded and satisfies
\begin{equation}\label{dx-conv-00}
 \lim_{t\to +\infty}  \| \dot{x}(t) \|=0 \,  \mbox{ and }\, \sup_{t \geq 0}  \| \ddot{x}(t) \|  < +\infty.
\end{equation}
%From the constitutive equation \eqref{closed_loop_1_counter_example}, we deduce immediately that any limit point of $x(t)$ as $t\to +\infty$
Let us argue by contradiction, and assume that there exists some $t_1 >0$ such that $x(t) \geq a$ for all $t\geq t_1$.
We can distinguish two cases:

\smallskip

\noindent $\bullet$ \textbf{First case}: $\dot{x}(t) \geq 0$ for all $t \geq t_1$. Then, $t\mapsto x(t)$ is increasing and bounded, hence converges to some $x_{\infty}\in \R$. From the constitutive equation \eqref{closed_loop_1_counter_example}, $\lim_{t\to +\infty}  \| \dot{x}(t) \|=0$,  and the continuity of $\nabla f$, we deduce that
$\lim_{t\to +\infty}   \ddot{x}(t) = - \nabla f (x_{\infty})$.
Using again that $\lim_{t\to +\infty}  \| \dot{x}(t) \|=0$, we deduce that $\nabla f (x_{\infty})=0$.
Since $x \mapsto \nabla f( x)$ is an increasing function, and
 $\nabla f( x(t_1)) \geq 0$, we obtain that
$$
\nabla f( x (t))=0  \mbox{ for all } \, t\geq t_1.
$$
Returning to the constitutive equation \eqref{closed_loop_1_counter_example}, we get
$$
\ddot{x}(t) +   | \dot{x}(t) |^{p-2} \dot{x}(t)=0 \mbox{ for all } \, t\geq t_1.
$$
Since  $\lim_{t\to +\infty}  \| \dot{x}(t) \|=0$, and $p\geq 3$, we have for $t$ sufficiently large, say $t\geq t_2 \geq t_1$\\
 $
  | \dot{x}(t) |^{p-1} \leq | \dot{x}(t) |^2.
 $
Therefore, for  all $t\geq t_2$
$$
\ddot{x}(t) +  | \dot{x}(t) |^{2} \geq 0 .
$$
Since $x(\cdot)$ is not constant, there exists some $t_3 \geq t_2$ such that $\dot{x}(t_3) >0$.
According to Lemma \ref{lem_tech}, we have
$\lim_{t\to +\infty} x(t) = + \infty$,
a clear contradiction with the convergence of $x(t)$.

\smallskip

\noindent $\bullet$ \textbf{Second case}: there exists  $t_2 \geq t_1$ such that $\dot{x}(t_2) < 0$. From the constitutive equation \eqref{closed_loop_1_counter_example} and $x(t)\geq a$
we get, for  all $t\geq t_2$
$$
\ddot{x}(t) +   | \dot{x}(t) |^{p-2} \dot{x}(t)= - \nabla f(x(t)) \leq 0
$$
This implies, for $t$ large enough
$$
\ddot{x}(t) \leq    | \dot{x}(t) |^{2}.
$$
Let's apply Lemma \ref{lem_tech} to $-x(\cdot)$. Since  $-\dot{x}(t_2) > 0$, we obtain $\lim_{t\to +\infty} x(t) = - \infty$, a  contradiction.

A similar argument gives the same kind of result for $b$, namely, for every $t_1>0$ there exists $t>t_1$ such that $x(t)>b$.  Therefore, there is an infinite number of times such that trajectory takes the values $a$ and $b$, which means that it oscillates indefinitely between $a$ and $b$.
\end{proof}

By contrast, if the damping effect is sufficiently important, there is convergence. In our situation, this corresponds to the case $2\leq p<3$, as shown in the following proposition.

 \begin{proposition}\label{prop_conv}
Suppose that $2 \leq p < 3$.
Then, any solution trajectory of \eqref{closed_loop_1_counter_example} converges, and its limit belongs to $[a,b]$.
 \end{proposition}

\begin{proof} When $ p = 2 $ the convergence follows from Alvarez's theorem for (HBF). So suppose $ 2 <p <3 $. We  sketch the main lines of the proof, whose details can be found in Haraux-Jendoubi \cite[Theorem 9.2.1]{HJ2}, which deals with a slightly more general situation.
Take $x$ a solution trajectory of \eqref{closed_loop_1_counter_example}, and denote by $\omega(x)$ its  limit set, that is the set of all its limit points as $t \to +\infty$ (limits of sequences $x(t_n)$ for $t_n \to +\infty$). By classical argument, this set is a connected subset of  $\{\nabla f =0\}$, that is  $\omega(x) \subset [a,b]$.
If $\omega(x)$ is reduced to a singleton, the proof is finished.
Let us therefore examine the complementary case
$$
\omega(x) = [c,d]\subset [a,b], \, \mbox{ with } \, c<d,
$$ and show that this leads to a contradiction.
Set $l:= \frac{1}{2}(c+d)$. Let us prove that
$\lim_{t\to +\infty} x(t)=l$, which gives $\omega(x) = \{l\}$,  a clear contradiction with $\omega(x) = [c,d]$, $c\neq d$.
First, since $l$ belongs to the interior of  $\omega(x) $, according to the intermediate value property, there exists a sequence $(t_n)$ with $t_n \to +\infty$ such that $ x(t_n)=l$. By continuity of $x$, and since $l$ belongs to the interior of $[c,d]$, for each $n\in \N$ there exists $\delta_n >0$ such that
$$
x(t) \in [c,d] \mbox{ for all } \,  t\in [t_n, t_n + \delta_n].
$$
Let's prove that for $n$ large enough we can take $\delta_n =+\infty$. Set
$$
\theta_n = \inf \left\lbrace t >t_n: \, x(t) \notin [c,d]   \right\rbrace,
$$
and assume $\theta_n <+\infty$. So for all $t \in [t_n, \theta_n]$ we have $\nabla f(x(t))=0$, and  \eqref{closed_loop_1_counter_example} reduces to
$$
\ddot{x}(t) +   | \dot{x}(t) |^{p-2} \dot{x}(t)=0 .
$$
After multiplying by $\dot{x}(t)$, we get for all $t \in [t_n, \theta_n]$
$$
\frac{d}{dt} | \dot{x}(t) |^2 + 2 | \dot{x}(t) |^{p}=0.
$$
After integration from $t_n$ to $t \in [t_n, \theta_n]$, we get for all $t \in [t_n, \theta_n]$
$$
| \dot{x}(t) | = \Big(  | \dot{x}(t_n) |^{-p+2} + (p-2) (t-t_n)\Big)^\frac{-1}{p-2}.
$$
After further integration, we get for all $t \in [t_n, \theta_n]$
\begin{align*}
|x(t) - l| & = |x(t) - x(t_n)| \leq \int_{t_n}^t | \dot{x}(s) | ds  \\
& = \frac{1}{p-3} \Big(  | \dot{x}(t_n) |^{-p+2} + (p-2) (t-t_n)\Big)^\frac{p-3}{p-2} +  \frac{1}{3-p}  | \dot{x}(t_n)|^{3-p}\\
& \leq \frac{1}{3-p}  | \dot{x}(t_n)|^{3-p},
\end{align*}
where, to obtain the last inequality, we use the hypothesis $2<p<3$.
Since $\dot{x}(t_n)$ converges to zero as $t_n \rightarrow + \infty$, for $n$ large enough we have that $\theta_n = +\infty$. This means that for $n$ large enough
$$x(t) \in [c,d] \, \,  \mbox{and} \, \,  |x(t) - l| \leq \frac{1}{3-p}  | \dot{x}(t_n)|^{3-p} \quad \forall t \in [t_n,+\infty),$$
which implies that $x(t)$ converges to $l$ as $t\to +\infty$.
\end{proof}

\vspace{2mm}

\noindent Figure \ref{fig:ex4} illustrates the attractor effect when the damping becomes too weak. Take $f:\R\to\R$
  $f(x)=\frac{1}{2}(x+1)^2 \, \mbox{ for } x\leq -1, \ f(x)=0 \, \mbox{ for }\, |x|<1, \, \mbox{ and } f(x)=\frac{1}{2}(x-1)^2\, \mbox{ for } x\geq 1.$

\begin{figure}[h]
	\centering
	%\captionsetup[subfigure]{position=top}
	%\subfloat[...]
	{\includegraphics*[viewport=78 200  540 600,width=0.325\textwidth]{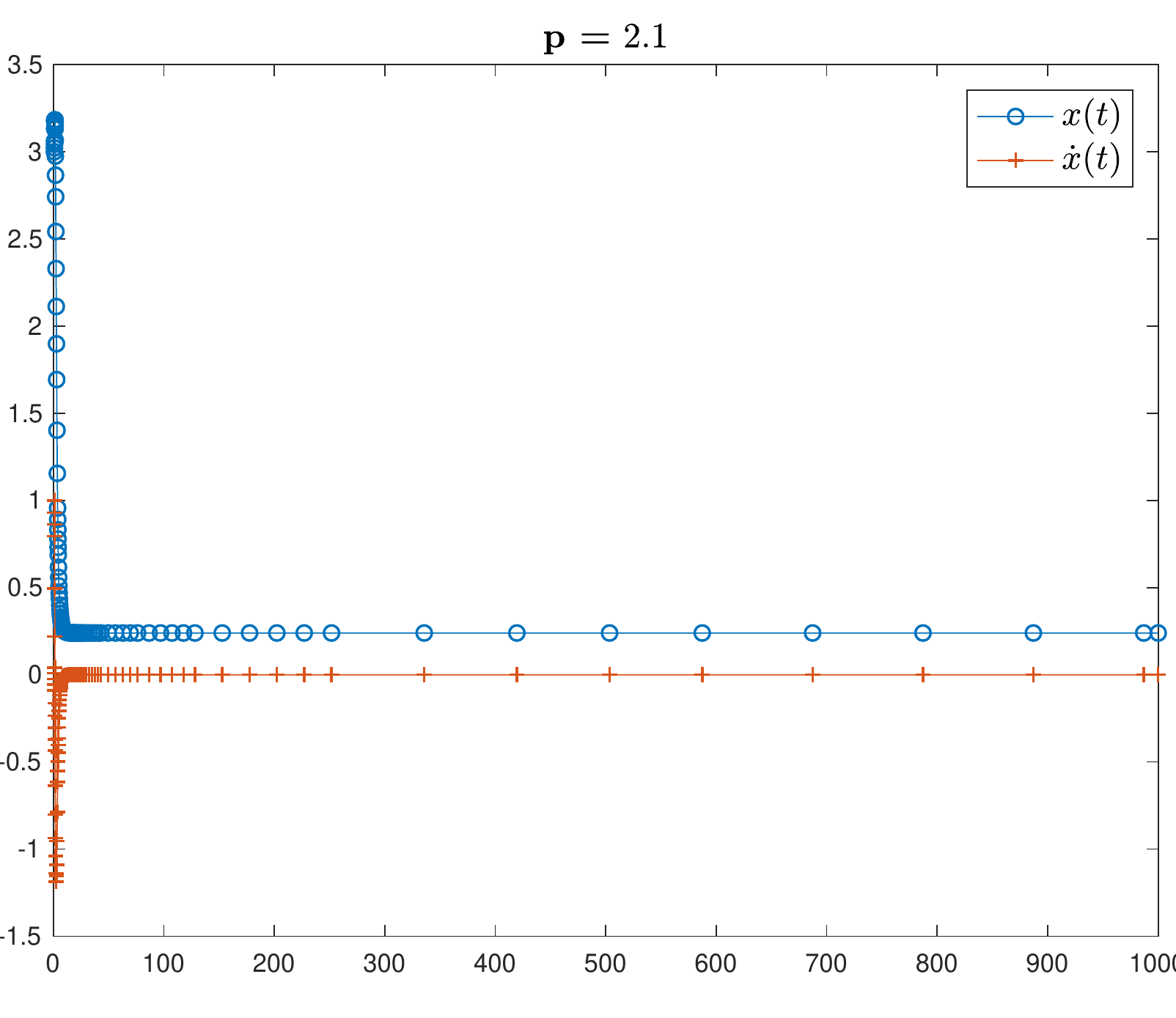}}\hspace{0.03cm}
	%\subfloat[...]
	{\includegraphics*[viewport=78 200  540 600,width=0.325\textwidth]{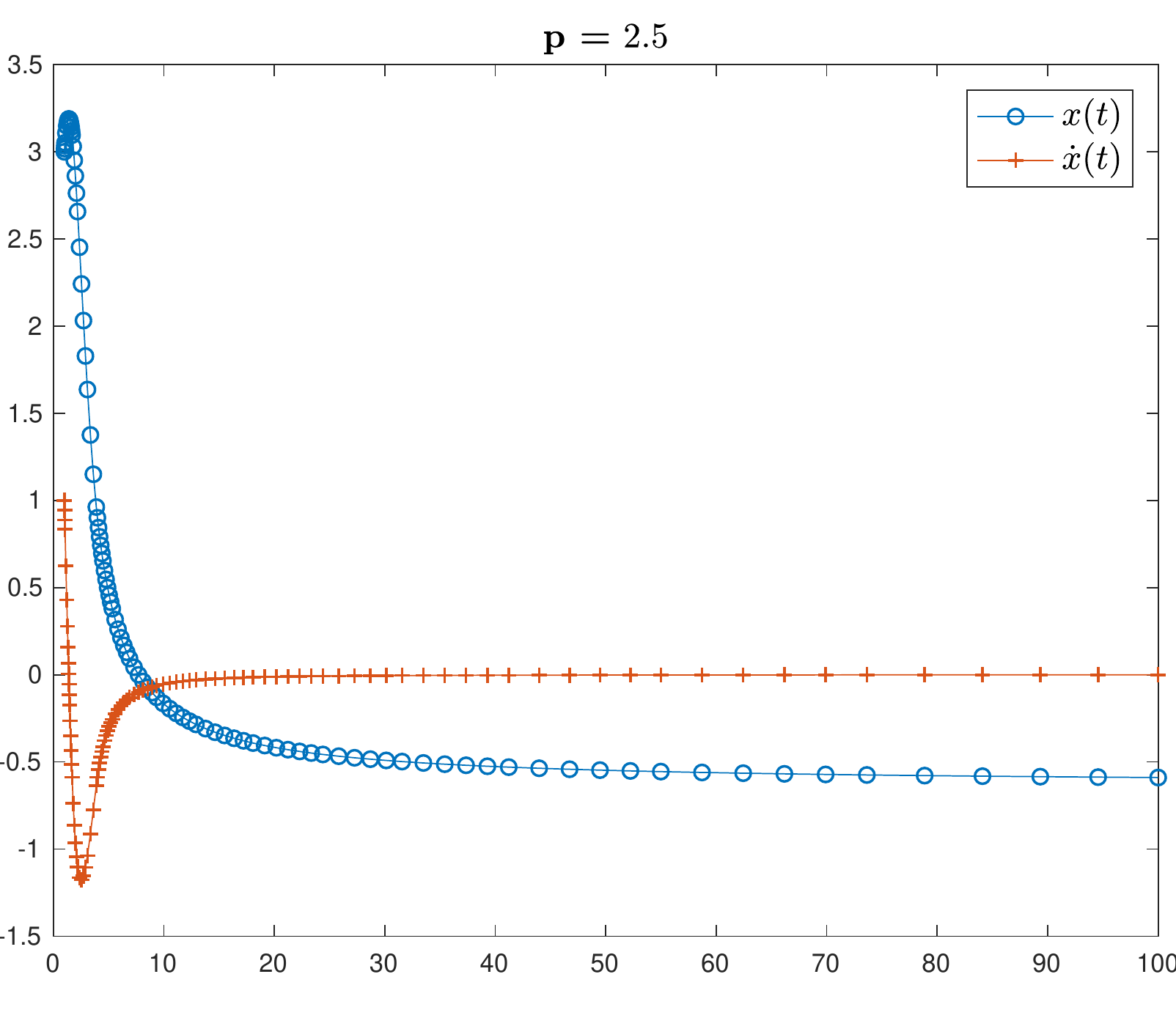}}\hspace{0.03cm}
	%\subfloat[...]
	{\includegraphics*[viewport=78 200  540 600,width=0.325\textwidth]{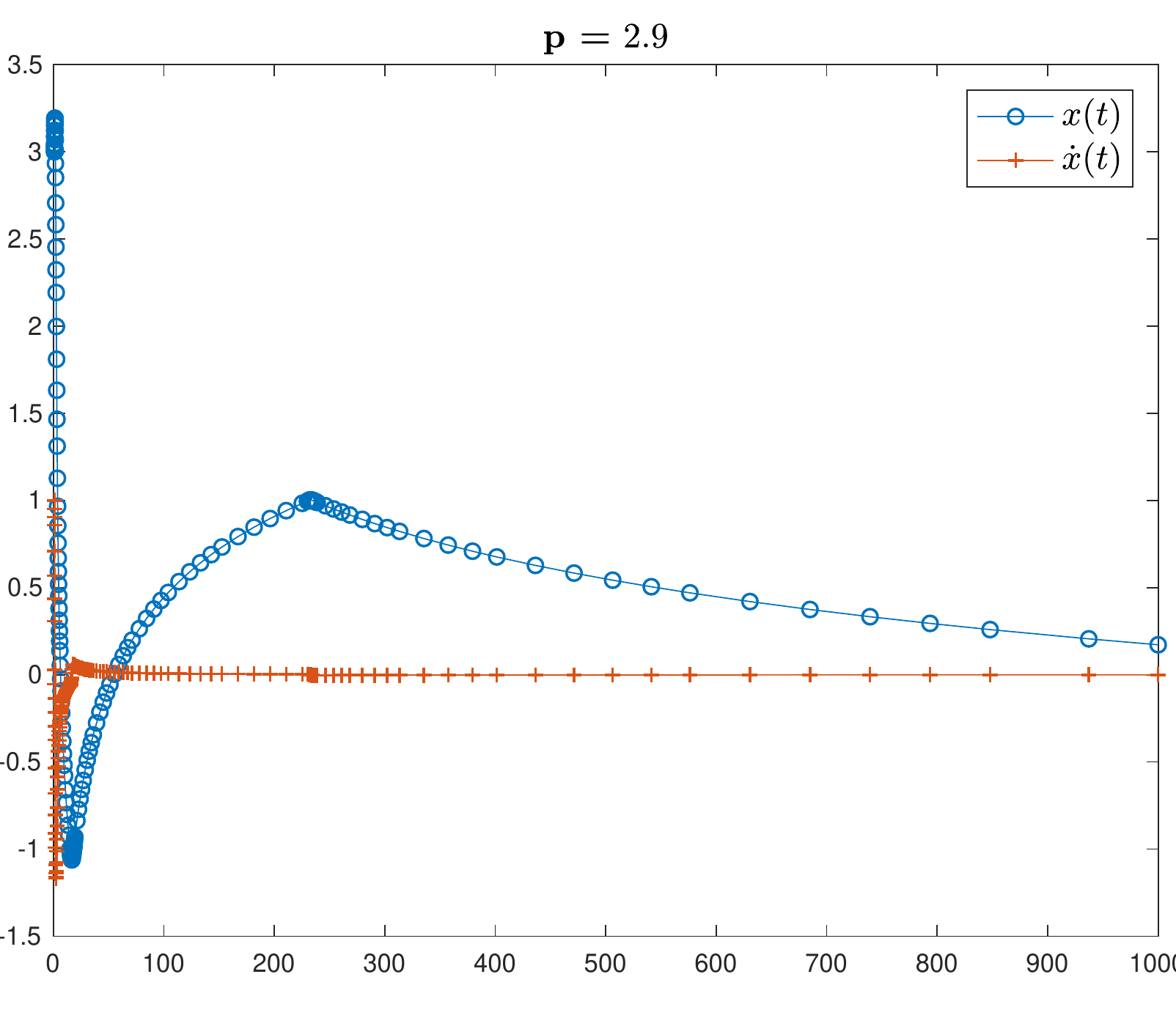}}\\
	\vspace{3mm}
	{\includegraphics*[viewport=78 200  540 600,width=0.325\textwidth]{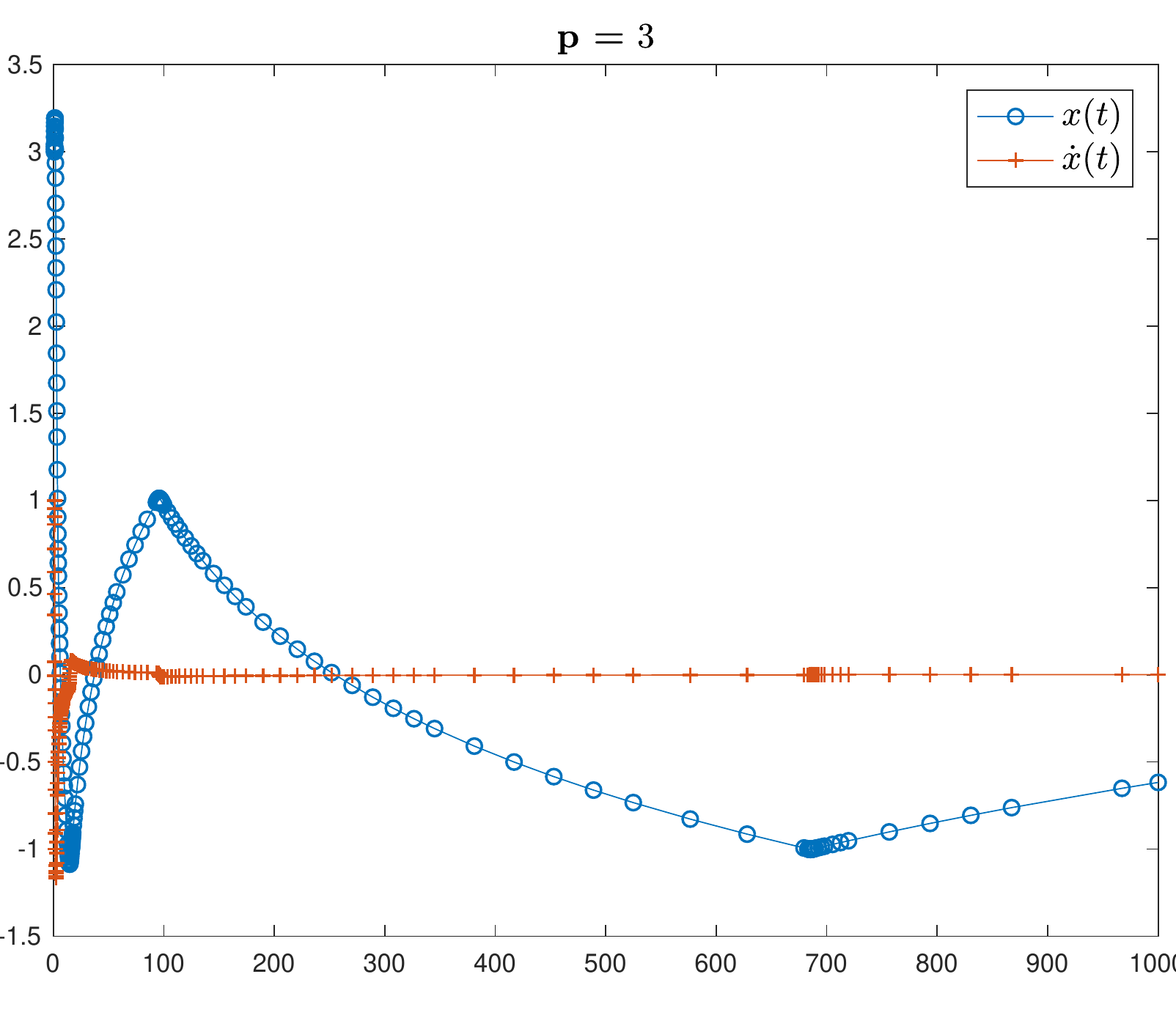}}\hspace{0.03cm}
	%\subfloat[...]
	{\includegraphics*[viewport=78 200  540 600,width=0.325\textwidth]{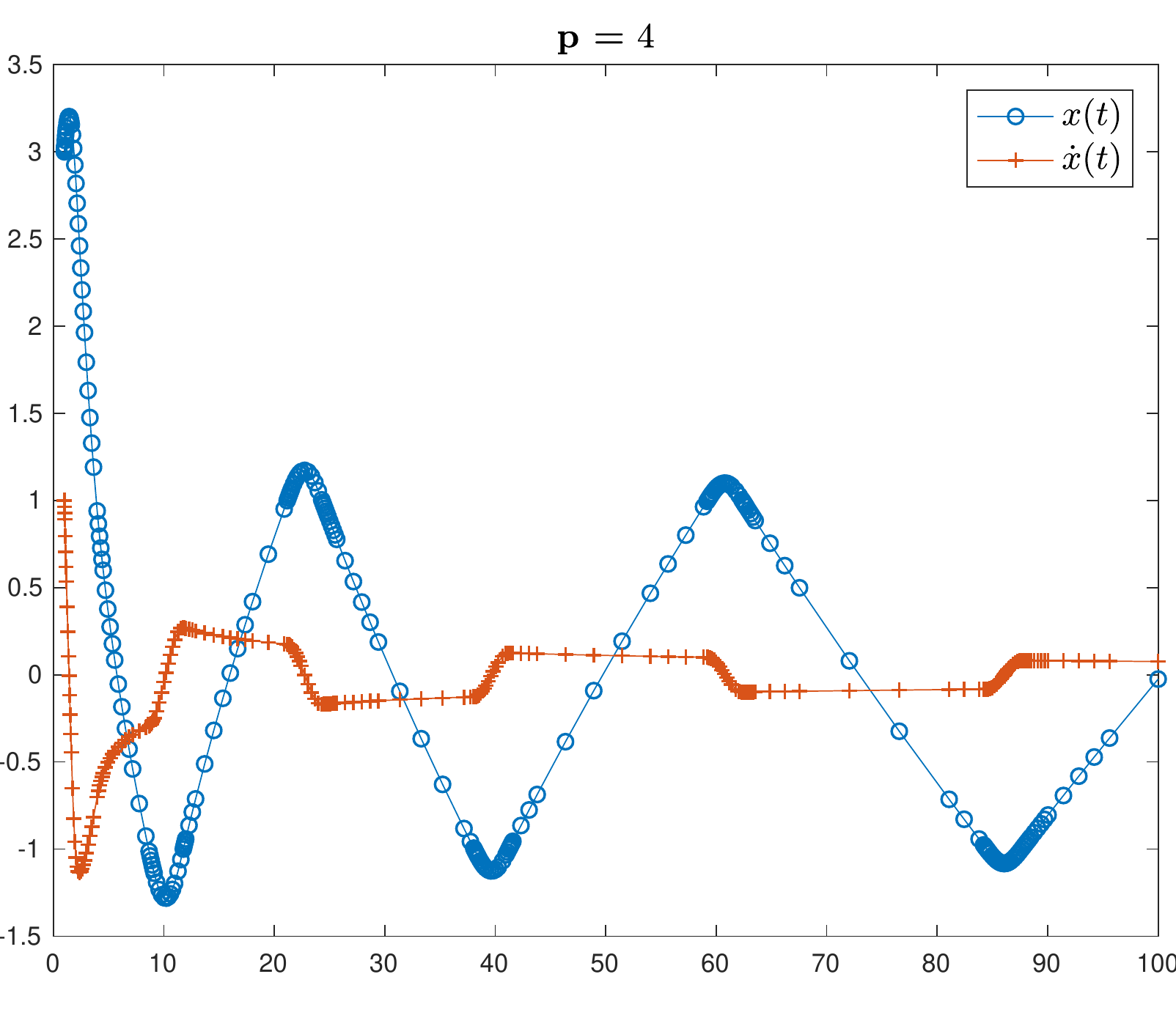}}\hspace{0.03cm}
	%\subfloat[...]
	{\includegraphics*[viewport=78 200  540 600,width=0.325\textwidth]{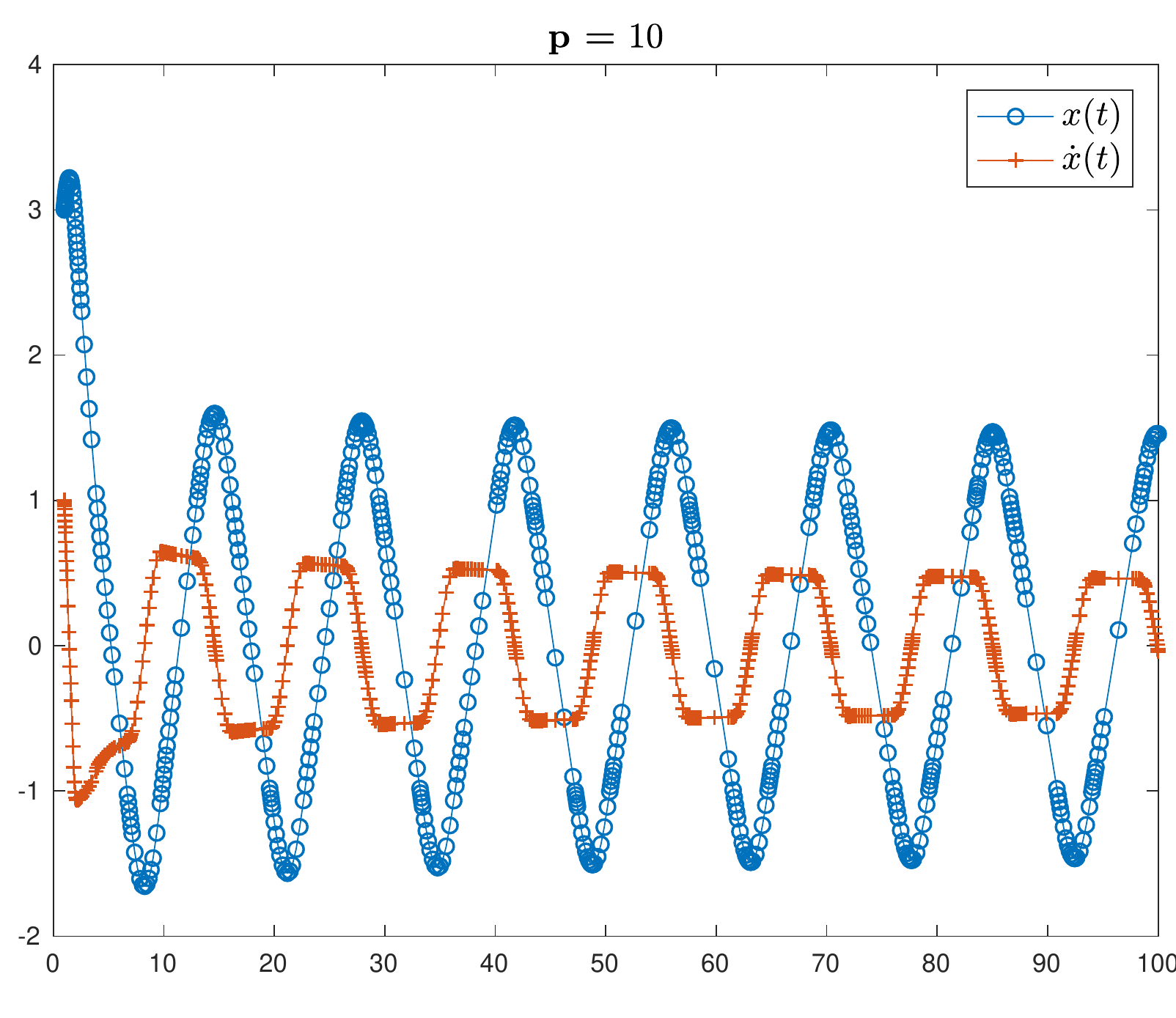}}
	\caption{\small Evolution of the trajectories $x(t)$ (blue) and $\dot x(t)$ (red) of the dynamical system \eqref{adige_v_one_dim} for  $p \geq 2$.
}
	\label{fig:ex4}	
\end{figure}

   For $p= 3$ there is a radical change in trajectory behavior. For $p\geq 3$, they do not converge asymptotically, and exhibit very oscillating behavior by steadily passing through the points $-1$ and $+1$. For $2 < p <3$, there is numerical evidence that the trajectories converge, which confirms the conclusion of Proposition \ref{prop_conv}.

%\vspace{1cm}

\subsection{An explicit one-dimensional example}
Take $\mathcal H = \mathbb R$ and $f(x) = c|x|^{\gamma}$, where $c$ and $\gamma$ are positive parameters. Let us look for  solutions of
\begin{equation}\label{closed_loop_1_example}
 \ddot{x}(t) + r| \dot{x}(t) |^{p-2} \dot{x}(t)  +  \nabla f (x(t)) = 0,
\end{equation}
when $p>1$.
Precisely, we look for nonnegative solutions
of the form $x(t)= \displaystyle{\frac{1}{t^{\theta}}}$, with $\theta >0$. This means that the trajectory is not oscillating, it is a completely damped trajectory. We proceed by identification,  and determine the values of the parameters $c$, $\gamma$, $r, p$ and $\theta$ which provide such solutions. On the one hand,
$$
\ddot{x}(t) + r| \dot{x}(t) |^{p-2} \dot{x}(t) = \frac{\theta (\theta +1)}{t^{\theta +2}} - \frac{ \theta^{p-1} r}{t^{(\theta +1)(p-1)}}.
$$
On the other hand, $\nabla f (x)= c \gamma |x|^{\gamma -2}x $, which gives
$$
\nabla f (x(t))=   \frac{c \gamma}{t^{\theta (\gamma -1)}}.
$$
Thus, $x(t)= \frac{1}{t^{\theta}}$ is solution of \eqref{closed_loop_1_example} if, and only if,
$$
\frac{\theta (\theta +1)}{t^{\theta +2}} - \frac{ \theta^{p-1} r}{t^{(\theta +1)(p-1)}} + \frac{c \gamma }{t^{\theta (\gamma -1)}}=0.
$$
This is equivalent to solve the following system:
\begin{itemize}
	\item [$i)$] $\theta+ 2 = \theta (\gamma -1)$;
	\item [$ii)$] $\theta+ 2 = (\theta +1)(p-1)  $;
	\item [$iii)$] $c \gamma = \theta^{p-1} r -\theta (\theta +1)$;
	\item [$iv)$] $\theta >0$, $c>0$
\end{itemize}

\noindent Solving $i)$ and $ii)$ with respect to $\gamma$  gives $2<p<3, \,  \gamma >2$, and the following values of $\theta$ and $p$
$$
  \theta = \frac{2}{\gamma -2}, \quad  p = \frac{3\gamma -2}{\gamma}.
$$
Condition $iii)$ gives
$
c= \frac{\theta}{\gamma}\left(  \theta^{p-2} r - (\theta +1)  \right)
$
and the nonnegativity condition $iv)$ gives
$
r >  \frac{\theta + 1}{\theta^{p-2}}.
$
We have $\min f = 0$ and
$$
f (x(t)) =c \frac{1}{t^{\frac{2 \gamma}{\gamma -2 }}}=  \frac{c}{t^{\frac{2 }{p-2}}}.
$$
To summarize, we have shown that when taking $2<p<3$, and $f(x)= c|x|^{\frac{2}{3-p}} $ there exists a solution of
$$
\qquad \ddot{x}(t) +r \| \dot{x}(t) \|^p \dot{x}(t)  + \nabla  f (x(t)) = 0,
$$
of the form $x(t)= \frac{1}{t^{\frac{3-p}{p-2}}}$, for which
$$
f (x(t)) - \min f =  \frac{c}{t^{\frac{2 }{p -2}}}.
$$

As expected,
the speed of convergence of $f (x(t))$ to $0$ depends on the parameter $p$.  Therefore, without other geometric assumptions on $f$, for $2 <p<3$, we cannot expect a convergence rate better than
$\mathcal O \Big(\displaystyle{\frac{1}{t^{\frac{2 }{p -2}}}}\Big)$.
When $p \to 3$ from below, the function $f(x)= c|x|^{\frac{2}{3-p}} $ becomes very flat around its minimum (the origin) and the convergence rate of $x(t) =\frac{1}{t^{\frac{3-p}{p-2}}}$ to the origin becomes very slow.

\section{Damping via closed-loop velocity control, quasi-gradient and (KL)}\label{sec: basic_3}

In this section,  $\cH=\mathbb R^N$ is the finite dimensional Euclidean space. This will allow us to use the Kurdyka--Lojasiewicz property, which we briefly designate by (KL). Unless otherwise indicated, no convexity assumption is made  on the  function $f$ to minimize, which will be assumed to satisfy (KL). To obtain the convergence of the orbits, the need for a geometric assumption on the function $f$ to be minimized has long been recognized.
As for the steepest descent, without additional geometric assumptions on the potential function $f$, the bounded orbits of the heavy ball with friction dynamic (HBF) may not converge. Let's recall the result from
 \cite{AGR}   where it is shown  a function $f: \R^2 \to \R$ which
is $\mathcal C^1$, coercive, whose gradient is Lipschitz continuous on the bounded sets,
and such that the (HBF) system
admits an orbit $t \mapsto x(t)$ which does not converge as $t \to +\infty$. This example is an inertial version of the famous Palis--De Melo counterexample for the continuous steepest descent \cite{PDM}.
In this section, we examine an important situation where the convergence property is satisfied, namely when $f$ is assumed to satisfy the property (KL), a  geometric notion which is presented below.

\subsection{Some basic facts concerning (KL)}
A  function $G: \R^N \to \R$ satisfies the (KL) property if its values can be reparametrized in the neighborhood of each of its
critical point, so that the resulting function become sharp. This means the existence of a continuous, concave, increasing function $\theta$   such that for all $u$ in a slice of $G$
$$
\| \nabla (\theta \circ G) (u)\| \geq 1.
$$
The function $\theta$ captures the geometry of $G$ around its critical point, it is called a desingularizing function;  see \cite{BDLM}, \cite{ABS}, \cite{abrs1} for further results.
The tame functions satisfy the property (KL). %which makes it possible to cover most of the practical situations in finite dimension, see \cite{CBFP}.
 Tameness refers to an ubiquitous geometric property of functions and sets encountered in most finite
dimensional optimization problems.
Sets or functions are called tame when they can be described by a finite number of basic formulas/
inequalities/Boolean operations involving standard functions such as polynomial, exponential,
or max functions. Classical examples of tame objects are piecewise
linear objects (with finitely many pieces), or semi-algebraic objects. The  general notion covering these situations is the concept of $o$-minimal structure; see van den Dries \cite{vdDries}.
Tameness models nonsmoothness via the so-called stratification property of
tame sets/functions. It was this property which motivated the vocable of tame topology, “la topologie
modérée” according to Grothendieck.
All these aspects have been well documented in a series of recent papers devoted to nonconvex nonsmooth optimization; see  Ioffe \cite{Ioffe} ("An invitation to tame optimization"), Castera--Bolte--F\'evotte--Pauwels \cite{CBFP} for an application to deep learning, and references therein.  We refer to \cite{abrs1} for illustrations,  and examples within a general
optimization setting.
%One is referred to (Van denD98; Cos00; Shi12) for foundational material.

This property is particularly interesting in our context, because we work with an \textit{autonomous dynamical system}, in which case the (KL) theory applies to the quasi-gradient systems. This contrasts with the   accelerated gradient method of Nesterov which is based on a non-autonomous dynamic system, and for which we do not have a convergence theory based on the (KL) property.
Under this property, we will obtain convergence results with convergence rates linked  to the geometry of the data functions $f$ and $\phi$, via the desingularizing function.

\subsection{Quasi-gradient systems}
Let us first recall the main lines of the quasi-gradient approach to the inertial gradient systems as developed by B\'egout--Bolte--Jendoubi in \cite{BBJ}.
The geometric interpretation is simple: a vector field  $F$ is called quasi-gradient for a function $E$
if it has the same singular point as $E$  and if the angle  between the field $F$ and the gradient $\nabla E$ remains acute
and bounded away from $\pi/2$.  A precise  definition is given below. Of course, such systems have a behavior
which is very similar to those of gradient systems. We refer to Barta--Chill--Fa\v{s}angov\'a \cite{BCF},  \cite{BF}, \cite{CF}, Chergui \cite{Chergui},  Huang \cite{Huang} and the references therein for
further geometrical insights on the topic.

\begin{definition}\label{quasi_grad_def_0}
Let $\Gamma$ be a nonempty closed subset of $\mathbb R^N$, and let
$F : \mathbb R^N  \to \mathbb R^N$ be a locally Lipschitz continuous
mapping.  We say that the first-order system
\begin{equation}\label{quasi_grad_def_1}
\dot{z}(t) + F (z(t))=0,
\end{equation}
has a quasi-gradient structure for $E$ on $\Gamma$,  if there exist a differentiable function $E : \mathbb R^N \to \R$  and $\alpha >0$ such that the two following conditions are satisfied:

\medskip

(angle condition) \quad \qquad $\left\langle  \nabla E (z), F(z) \right\rangle
\geq \alpha \|\nabla E (z) \|  \| F(z) \| $  \, for all $z\in \Gamma$;

\medskip

(rest point equivalence) \quad $\mbox{\rm crit}E \cap \Gamma = F^{-1} (0) \cap \Gamma$.
\end{definition}

%\noindent The geometrical interpretation is clear: at each $z \in \Gamma$,
% the angle  between the vector $F(z)$  and the gradient $\nabla E (z)$  remains acute
%and bounded away from $\pi /2$.
Based on this notion, we have the following convergence properties for the bounded trajectories of \eqref{quasi_grad_def_1}.
The following result is a localized version and straight adaptation of  \cite[Theorem 3.2]{BBJ}.

\begin{theorem} {\rm} \label{quasi_grad_thm_1}
Let $F: \mathbb R^N  \to \mathbb R^N$ be a locally Lipschitz continuous mapping.
 Let $z: [0, +\infty[ \to \mathbb R^N $ be a bounded solution trajectory of \eqref{quasi_grad_def_1}.
Take $R \geq \sup_{t\geq 0}     \|z(t)\|$.
Assume that     $F$  defines  a quasi-gradient vector field for $E_R$ on $\bar{B}(0,R)$, where  $E_R: \mathbb R^N \to \R$ is a differentiable function.
 Assume
further that the function  $E_R$ is {\rm(KL)}.
Then, the following properties are satisfied:

\medskip

$(i)$ \, $z(t) \to z_{\infty}$ as $t \to + \infty$, where $ z_{\infty}\in  F^{-1} (0)$;

\medskip

$(ii)$ \,  $\dot{z} \in L^1 (0, +\infty; \R^N )$, $\dot{z}(t) \to 0$ as $t \to + \infty$;

\medskip

$(iii)$
$ \| z(t) -  z_{\infty}\| \leq \frac{1}{\alpha_R} \theta \Big( E_R(z(t)) - E_R (z_{\infty})\Big)
$

\medskip

\noindent where $\theta$ is the desingularizing function for $E_R$ at  $z_{\infty}$, and $\alpha_R$ enters the angle condition of Definition  \ref{quasi_grad_def_0}.
\end{theorem}

\subsection{Convergence of systems with closed-loop velocity control under (KL)}\label{Sec-conv-KL}

Let us apply the above approach to  the inertial system with
closed-loop damping
\begin{equation}\label{first_order_cl_loop_quasi_0}
 \ddot x(t) +\nabla \phi(\dot x(t))+ \nabla f(x(t))=0,
\end{equation}
by writing it as a first-order system, via its Hamiltonian formulation.
We will assume that  $\nabla \phi$ is locally Lipschitz continuous. Indeed, we can reduce to this situation by using a regularization procedure based on the Moreau envelope.

\begin{theorem} \label{quasi_grad_thm_2}
Let $f: \R^N \to \R$ be a $\mathcal C^2$ function whose gradient is  Lipschitz continuous on the bounded sets, and such that $\inf_{\R^N} f >-\infty$.
Let  $ E_{\lambda}: \R^N \times \R^N\to \R $ be  defined by: for all $(x,u)\in  \R^N \times \R^N$
 $$
 E_{\lambda}(x,u):= \demi \|u\|^2 + f(x) +\lambda \left\langle \nabla f (x), u\right\rangle.
$$
Suppose that  the function  $ E_{\lambda}$ satisfies the {\rm (KL)} property.\\
Let $\phi : \R^N \to \R_+$ be a damping potential (see Definition \ref{def1}) which is differentiable, and  which  satisfies the following growth conditions:

\medskip

$(i)$ (local) there exists positive constants $\gamma$,  $\delta$, and  $\epsilon >0$ such that, for all $u$ in
$\R^N$ with $\|u\| \leq \epsilon$
$$\phi (u) \geq \gamma \|u\|^2   \mbox{ and }    \|\nabla \phi (u) \| \leq \delta \|u\|.$$

$(ii)$ (global) there exists  $p\geq 1$, $c>0$, such that for all $u$ in $\R^N$, $\phi (u) \geq c\|u\|^p$.

\medskip

\noindent  Let $x: [0, +\infty[ \to \mathbb R^N $ be a bounded solution trajectory of
$$
 \ddot x(t) +\nabla \phi(\dot x(t))+ \nabla f(x(t)) =0 .
$$
Then, the following properties are satisfied:

\medskip

$(i)$ \, $x(t) \to x_{\infty}$ as $t \to + \infty$, where $ x_{\infty}\in  \crit f$;

\medskip

$(ii)$ \,  $\dot{x} \in L^1 (0, +\infty; \R^N )$, $\dot{x}(t) \to 0$ as $t \to + \infty$;

\medskip

$(iii)$ \,  For $\lambda$ sufficently small, and $t$ sufficiently large
$$ \| x(t) -  x_{\infty}\| \leq \frac{1}{\alpha} \theta \Big(  E_{\lambda}(x(t),u(t)) - E_{\lambda} (x_{\infty},0)\Big)
$$

\smallskip

\noindent where $\theta$ is the desingularizing function for $ E_{\lambda}$ at  $(x_{\infty},0)$, and $\alpha$ enters the corresponding angle condition.
\end{theorem}
\begin{proof}
According to the preliminary estimates established in Proposition \ref{preliminary_est}, we have
$$
\int_0^{+\infty} \phi (\dot{x}(t)) dt < +\infty  \, \, \mbox{ and }\, \, \sup_{t\geq 0} \| \dot{x}(t)\| <+\infty.
$$
Combining the first above property with the global growth assumption on $\phi$, we deduce that there exists $p\geq 1$ such that
$$
\int_0^{+\infty} \| \dot{x}(t)\|^p dt < +\infty.
$$
According to the constitutive equation, we have
$$
\ddot x(t) = -\nabla \phi(\dot x(t))- \nabla f(x(t)).
$$
Since $x(\cdot)$ and $\dot{x}(\cdot)$ are  bounded, and  $\nabla f $ is locally Lipschitz continuous, we deduce that $\ddot x (\cdot)$
is also bounded. Classically, these properties imply that
$\dot{x}(t) \to 0$ as $t \to + \infty$.
Take $R \geq \sup_{t\geq 0}     \|x(t)\|$.
Therefore, for $t$ sufficiently large, we have that the trajectory
$t \mapsto (x(t),\dot{x}(t))$ in the phase space $\R^N \times  \R^N$
belongs to the closed set $\Gamma = \bar{B}(0,R) \times \bar{B}(0,\epsilon) $.

\noindent The  Hamiltonian formulation of  \eqref{first_order_cl_loop_quasi_0} gives the first-order differential system
\begin{equation}\label{first_order_cl_loop_quasi_1}
\dot z(t) + F(z(t)) =0,
\end{equation}
where $z(t)=(x(t), \dot x(t)) \in \R^N \times \R^N $, and
   $F: \R^N \times  \R^N \to \R^N \times \R^N$
is defined by
\begin{equation}\label{ham}
F(x,u)=(-u, \nabla \phi(u)+ \nabla f(x)).
\end{equation}
Following \cite{BBJ}, take $E_{\lambda} : \mathbb R^N \times \mathbb R^N\to \R$ defined by
\begin{equation}\label{e_l_quasi_grad_th2}
 E_{\lambda}(x,u):= \demi \|u\|^2 + f(x) +\lambda \left\langle \nabla f (x), u\right\rangle,
\end{equation}
where the parameter $\lambda >0$  will be adjusted  to verify the quasi-gradient property. We have
$$
\nabla  E_{\lambda}(x,u) = \Big( \nabla f (x)+ \lambda \nabla^2 f (x)u, \, u + \lambda \nabla f (x) \Big).
$$
Let us analyze the 	angle condition with $\Gamma = \bar{B}(0,R) \times \bar{B}(0,\epsilon) $. According to the above formulation of $F$ and $\nabla  E_{\lambda}$, we have
\begin{eqnarray*}
\left\langle  \nabla  E_{\lambda}(x,u), F(x,u) \right\rangle
&=& \left\langle  \Big( \nabla f (x)+ \lambda \nabla^2 f (x)u, \, u + \lambda \nabla f (x) \Big), \Big(-u, \nabla \phi(u)+ \nabla f(x)\Big) \right\rangle \\
&=&  -\left\langle  \nabla f (x)+ \lambda \nabla^2 f (x)u, \, u  \right\rangle + \left\langle  u+ \lambda \nabla f (x), \, \nabla \phi(u)+ \nabla f(x)  \right\rangle.
\end{eqnarray*}
After development and simplification, we get
\begin{eqnarray*}
\left\langle  \nabla  E_{\lambda}(x,u), F(x,u) \right\rangle
&=&  - \lambda \left\langle  \nabla^2 f (x)u, \, u  \right\rangle + \left\langle  u, \, \nabla \phi(u) \right\rangle
+ \lambda \left\langle  \nabla f (x) , \, \nabla \phi(u)  \right\rangle
+ \lambda  \|  \nabla f (x) \|^2.
\end{eqnarray*}
According to the local Lipschitz assumption on $\nabla f$,  let
$$
M:= \sup_{\|x\| \leq R} \|  \nabla^2 f (x)\| <+\infty.
$$
Since $\phi$ is a damping potential, we have
$$
\left\langle  u, \, \nabla \phi(u) \right\rangle \geq \phi (u).
$$
Combining the above results, we obtain
\begin{eqnarray}
\left\langle  \nabla  E_{\lambda}(x,u), F(x,u) \right\rangle
&\geq &  - \lambda M \|u\|^2 +  \phi(u)
+ \lambda \left\langle  \nabla f (x) , \, \nabla \phi(u)  \right\rangle
+ \lambda  \|  \nabla f (x) \|^2 \nonumber \\
&\geq &  - \lambda M \|u\|^2 +  \phi(u)
-\frac{\lambda}{2} \|  \nabla f (x) \|^2 - \frac{\lambda}{2} \| \nabla \phi(u)\|^2  + \lambda  \|  \nabla f (x) \|^2  \nonumber \\
&\geq &  - \lambda M \|u\|^2 +  \phi(u)
 - \frac{\lambda}{2} \| \nabla \phi(u)\|^2  + \frac{\lambda}{2}  \|  \nabla f (x) \|^2 .  \label{quasi_gradient_1}
\end{eqnarray}
At this point, we use the local growth assumption on $\phi$: for all $u$ in
$\R^N$ with $\|u\| \leq \epsilon$
\begin{equation} \label{quasi_gradient_11}
\phi (u) \geq \gamma \|u\|^2   \mbox{ and }    \|\nabla \phi (u) \| \leq \delta \|u\|.
\end{equation}
By
combining  \eqref{quasi_gradient_1} with \eqref{quasi_gradient_11},  we obtain
\begin{eqnarray}
\left\langle  \nabla  E_{\lambda}(x,u), F(x,u) \right\rangle
&\geq &   \Big( \gamma -\lambda M   - \frac{\lambda}{2} \delta^2  \Big)  \|u\|^2
  + \frac{\lambda}{2}  \|  \nabla f (x) \|^2  . \label{quasi_gradient_2}
\end{eqnarray}
Take $ \lambda $ small enough to satisfy
$$
\gamma > \lambda \left(M +  \frac{\delta^2}{2}  \right).
$$
Then,
\begin{eqnarray}
\left\langle  \nabla  E_{\lambda}(x,u), F(x,u) \right\rangle
&\geq &   \alpha_0 ( \|u\|^2
  +   \|  \nabla f (x) \|^2)   \label{quasi_gradient_3}
\end{eqnarray}
with $\alpha_0:= \min\{\gamma - \lambda \left(M +  \frac{\delta^2}{2}      \right)  , \frac{\lambda}{2}  \}$.\\
On the other hand,
\begin{eqnarray*}
\| \nabla  E_{\lambda}(x,u)\| &\leq& \sqrt{2} \left( 1+  \lambda \max\{1, M\}   \right)  ( \|u\|^2
  +   \|  \nabla f (x) \|^2)^{\demi} \\
\| F(x,u)\| &\leq& \sqrt{2}(1+ \delta )  ( \|u\|^2
  +   \|  \nabla f (x) \|^2)^{\demi} .
\end{eqnarray*}
Therefore
\begin{eqnarray}
\| \nabla  E_{\lambda}(x,u)\|  \| F(x,u) \|
&\leq &   2 \left( 1+  \lambda \max\{1, M\}   \right)  (1+\delta) ( \|u\|^2
  +   \|  \nabla f (x) \|^2).   \label{quasi_gradient_4}
\end{eqnarray}
As a consequence, the angle condition
$$
\left\langle  \nabla E (z), F(z) \right\rangle
\geq \alpha \|\nabla E (z) \|  \| F(z) \|
$$
is satisfied on $\Gamma$, by taking
$$
\alpha = \frac{\min\{\gamma - \lambda \left(M +  \frac{\delta^2}{2}      \right)  , \frac{\lambda}{2}  \}}{2 \left( 1+  \lambda \max\{1, M\}   \right)  (1+\delta) }.
$$
Finally, the rest point equivalence is a consequence of the inequality
\eqref{quasi_gradient_3}.
Then, apply the abstract theorem \ref{quasi_grad_thm_1} to obtain the claims $(i),(ii),(iii)$.
\end{proof}

\begin{remark}\label{rem-quasi-grad-kl}
$(i)$ The above result allows to consider nonlinear damping.
The main restrictive assumption is that the damping potential is assumed to be nearly quadratic close to the origin. It is not necessarily quadratic close to the origin but it has to satisfy:
for all $u$ in
$\R^N$ with $\|u\| \leq \epsilon$
$$\phi (u) \geq \gamma \|u\|^2   \mbox{ and }    \|\nabla \phi (u) \| \leq \delta \|u\|.$$
$(ii)$ According to \cite[Proposition 3.11]{BBJ}, a desingularizing function of $f$ (see \cite[Definition  2.1]{BBJ}) is desingularizing of $E_{\lambda }$ too, for all $\lambda\in[0,\lambda_1]$.

\smallskip

\noindent $(iii)$ In Section \ref{rem-quasi-hessian} and Section \ref{rem-quasi-hessian_both-quasi} we will develop similar analysis for related dynamical systems which involve Hessian-driven damping.

\smallskip

\noindent $(iv)$  Following \cite[Theorem 4.1 and Theorem 3.7]{BBJ}, a key condition which induces convergence rates for the
trajectories of a quasi-gradient system in the framework of the (KL) property is
\begin{equation}\label{bolte-cond-rate-quasi}\| \nabla  E_{\lambda}(x,u)\|\leq b\| F(x,u)\| \mbox{ for } (x,u)\in\Gamma
\end{equation}
with $b>0$. Let us check this condition in the setting of Theorem \ref{quasi_grad_thm_2}.
We have seen there that
$$\| \nabla  E_{\lambda}(x,u)\|^2\leq C_1 (\|u\|^2+\|\nabla f(x)\|^2),$$
with $C_1 > 0$.
Further, from the Cauchy--Schwarz inequality and the properties of $\phi$ we derive for $\sigma > 1$:
\begin{eqnarray*}\| F(x,u)\|^2 &=& \|u\|^2 + \|\nabla\phi(u)+\nabla f(x)\|^2\\
&\geq& \|u\|^2 + \|\nabla\phi(u)\|^2+\|\nabla f(x)\|^2 -2\|\nabla \phi(x)\|\|\nabla f(x)\|\\
&\geq& \|u\|^2 + \|\nabla\phi(u)\|^2+\|\nabla f(x)\|^2 -\sigma \|\nabla \phi(u)\|^2-\frac{1}{\sigma}\|\nabla f(x)\|^2\\
&\geq& \left(1-\frac{1}{\sigma}\right)\|\nabla f(x)\|^2+\left(1-(\sigma-1)\delta^2\right)\|u\|^2.
\end{eqnarray*}
From there, we can choose $\sigma > 1$ such that
$$\| F(x,u)\|^2 \geq C_3(\|u\|^2+\|\nabla f(x)\|^2),$$
with $C_3>0$. Condition \eqref{bolte-cond-rate-quasi} is now met  with $b=\sqrt{C_1/C_3}$.

\noindent As in \cite[Section 5]{BBJ}, explicit convergence rates can be derived from \cite[Theorem 4.1 and Theorem 3.7]{BBJ}, based on (3.19) and Remark 3.4(c) in \cite{BBJ}.
\end{remark}

\subsection{Application:  \textit{f} with polynomial growth}\label{Sec:KL_poly_growth}

This concerns the question raised at the end of Section \ref{strong-convex}. Additionally to the
hypotheses of Theorem \ref{quasi_grad_thm_2}, assume that $f$ is convex, $\argmin f\neq\emptyset$ and for each $x^*\in\argmin f$, there exists $\eta > 0$ such that
$$f(x) -\inf\nolimits_{\cH} f \geq c\dist (x, \argmin f)^r \quad \forall x\in B(x^*,\eta),$$
with $r\geq 1$ and $c>0$.

According to the proof of \cite[Corollary 5.5]{BBJ}, $f$ satisfies the \L{}ojasiewicz inequality with
desingularizing function (see \cite[Definition 2.1]{BBJ}) of the form $\varphi(s)=c's^{1/r}$, with
$c'>0$. According to \cite[Proposition 3.11]{BBJ}, this is a desingularizing function of $E_{\lambda }$ too, for all $\lambda\in[0,\lambda_1]$ (with $E_{\lambda}$ defined in Theorem \ref{quasi_grad_thm_2}).
In Remark \ref{rem-quasi-grad-kl}(iv) we have shown that \eqref{bolte-cond-rate-quasi} holds.
Relying now on \cite[Theorem 3.7]{BBJ} and \cite[Remark 3.4(c)]{BBJ}, we derive sublinear rates for $\|x(t)-x_{\infty}\|$ in case $r<2$ and exponential rate in case $r=2$.

\subsection{Application: fixed damping matrix}

We will recover and improve the results of Alvarez \cite[Theorem 2.6]{Alvarez}, which concerns the case $f$ convex, and the damped inertial equation
$$
\ddot{x}(t) + A (\dot{x}(t))   +  \nabla f (x(t)) = 0,
$$
where $A: \cH \to \cH$ is a  positive definite self-adjoint linear operator, which is possibly anisotropic (see also \cite{BotCse}).
While the proof of Theorem 2.6 in \cite{Alvarez} works in general Hilbert spaces, we have to restrict ourselves to finite-dimensional spaces, however,  we can drop the convexity assumption on $f$.
The following result is a direct consequence of Theorem \ref{quasi_grad_thm_2} applied to  $\phi : \R^N \rightarrow \R+, \, \phi(x) = \frac{1}{2}\langle Ax, x\rangle,$ where $A : \R^N \rightarrow \R^N$ is a positive definite self-adjoint linear operator.  We note that in this setting, $\phi$ is a damping potential, it is convex continuous and  attains its minimum at the origin. Moreover, the local and  global growth conditions are met.
Indeed, for all $u\in \R^N$,    we have $\phi (u)\geq \demi\lambda_{min} \|u\|^2  $ and $\|\nabla \phi (u)\| \leq \lambda_{max}\|u\|$, where $\lambda_{min}$ and $\lambda_{max}$ are  the smallest and  largest positive eigenvalues of $A$ respectively.

\begin{theorem} \label{quasi_grad_thm_3}
Let $f: \R^N \to \R$ be a $\mathcal C^2$ function whose gradient is  Lipschitz continuous on the bounded sets, and such that $\inf_{\R^N} f >-\infty$. Let $A : \R^N \rightarrow \R^N$ a positive definite self-adjoint linear operator.
Suppose that  the function  $ E_{\lambda}$ is {\rm(KL)} (which is true if $f$ is {\rm(KL)}) where
 $$
 E_{\lambda}(x,u):= \demi \|u\|^2 + f(x) +\lambda \left\langle \nabla f (x), u\right\rangle.
$$

\noindent  Let $x: [0, +\infty[ \to \mathbb R^N $ be a bounded solution trajectory of
$$
 \ddot x(t) + A(\dot x(t))+ \nabla f(x(t)) =0 .
$$
Then, the following properties are satisfied:

\medskip

$(i)$ \, $x(t) \to x_{\infty}$ as $t \to + \infty$, where $ x_{\infty}\in  \crit f$;

\medskip

$(ii)$ \,  $\dot{x} \in L^1 (0, +\infty; \R^N )$, $\dot{x}(t) \to 0$ as $t \to + \infty$;

\medskip

$(iii)$ \, $f(x(t)) \to f(x_{\infty}) \in f(\crit f)$ as $t \to + \infty$.
\end{theorem}

Indeed, we can complete this result with the convergence rates which are linked to the desingularization function provided with the property (KL) for $ f $.

\section{Algorithmic results: an inertial type algorithm}
\label{sec: basic_4}

The following convergence result is a discrete algorithmic version of Theorem \ref{quasi_grad_thm_2}. To stay close to the continuous dynamic
we use a semi-implicit discretization: implicit with respect to the damping potential $\phi$, and explicit with respect to the function  $f$ to minimize. This will make it possible to make minimal assumptions about the damping potential $\phi$, and thus cover various situations.
Like the latter, the underlying structure of the proof is the
quasi-gradient property. We choose to give  direct proof, which is a bit simpler in this case.
Consider the following  temporal discretization of (ADIGE-V) with step size $h>0$
$$
\frac{1}{h^2}\left( x_{n+2}-2x_{n+1}+x_n\right) +\nabla\phi\left(\frac{1}{h}(x_{n+2}-x_{n+1})\right)+\nabla f(x_{n+1})=0.
$$
Equivalently,
\begin{equation}\label{alt-alg}
\frac{ x_{n+2}-x_{n+1}}{h} - \frac{ x_{n+1}-x_{n}}{h}   + h\nabla\phi\left(\frac{ x_{n+2}-x_{n+1}}{h}\right)+h\nabla f(x_{n+1})=0.
\end{equation}
This gives the  proximal-gradient algorithm
\begin{equation}\label{prox-grad-alg}
x_{n+2}= x_{n+1}+ h{\prox}_{h\phi}\left(\frac{1}{h}(x_{n+1}-x_n)-h\nabla f(x_{n+1})\right).
\end{equation}
Recall that, for any  $x \in \mathcal H= \R^N $, for any $\lambda >0$
$${\prox}_{\lambda\phi}(x):=\argmin_{\xi \in\mathcal H}\,\left\lbrace \lambda\phi(\xi)+\frac{1}{2} \| x-\xi \|^2  \right\rbrace.
$$
Let us start by establishing a decrease property for the sequence $ (W_n)_{n \in \N}$ of global energies
$$
W_n:= \demi \|u_n\|^2 + f(x_{n+1}),
$$
where we set $u_n= \frac{1}{h}\left(x_{n+1}-x_{n}\right)$ for the discrete velocity.
\begin{lemma}\label{basic-energy-lem-1}
Suppose that $f: \R^N \to \R$ is a differentiable function whose gradient is  $L$-Lipschitz continuous on a ball containing the iterates $(x_n)_{n \in \N}$,  and such that $\inf_{\R^N} f >-\infty$.
Suppose that $\phi (u) \geq \gamma \|u\|^2$ for all $u \in \R^N$.
Then, for all $n \in \N$
$$
W_{n+1}  -W_n + h\left( \gamma - \demi Lh  \right)\|u_{n+1}\|^2 \leq 0.
$$
As a consequence, under the assumption $\gamma > \demi Lh  $, we have
$$
\sum_{n \in \N} \|u_{n}\|^2 <+\infty \quad\mbox{and} \quad u_{n} \to 0 \;\mbox{ as } \; n \to + \infty.
$$
\end{lemma}
\begin{proof}
Write the algorithm as
$$
u_{n+1} -u_{n} + h \nabla \phi(u_{n+1})  +h\nabla f(x_{n+1})=0.
$$
By taking the scalar product with $u_{n+1}$, we obtain
\begin{equation}\label{energy-est-2021-a}
\langle u_{n+1} -u_{n}, u_{n+1} \rangle +
 h\langle \nabla \phi(u_{n+1}) , u_{n+1} \rangle
  +  \langle \nabla f(x_{n+1}), x_{n+2}-x_{n+1} \rangle =0.
\end{equation}
We have
\begin{eqnarray*}
&&\langle u_{n+1} -u_{n}, u_{n+1} \rangle \geq  \demi \|u_{n+1}\|^2 -\demi \|u_n\|^2 \\
&& \langle \nabla \phi(u_{n+1}) , u_{n+1} \rangle \geq \phi (u_{n+1}) \geq \gamma \|u_{n+1}\|^2 \\
&& f(x_{n+2}) - f(x_{n+1}) \leq \langle \nabla f(x_{n+1}), x_{n+2}-x_{n+1} \rangle + \frac{Lh^2}{2} \| u_{n+1}\|^2,
\end{eqnarray*}
where the last inequality follows from the descent gradient lemma.
By combining the above inequalities with (\ref{energy-est-2021-a}), we obtain
\begin{equation}\label{energy-est-2021-b}
\demi \|u_{n+1}\|^2 -\demi \|u_n\|^2  +
 h\gamma \|u_{n+1}\|^2
  + f(x_{n+2}) - f(x_{n+1}) - \frac{Lh^2}{2} \| u_{n+1}\|^2 \leq 0.
\end{equation}
Equivalently,
$$
W_{n+1}  -W_n + h\left( \gamma - \demi Lh  \right)\|u_{n+1}\|^2 \leq 0.
$$
By summing the above inequalites, and since $f$ is minorized, we get
$$
h\left( \gamma - \demi Lh  \right) \sum_{n \geq 1} \|u_{n}\|^2 \leq W_0 - \inf f.
$$
Since $\gamma - \demi Lh  >0$, we get $ \sum_{n \in \N} \|u_{n}\|^2  <+\infty$, and hence $u_n \to 0$ as $n \to +\infty$.
\end{proof}
\begin{theorem} \label{quasi_grad_thm_algo}
Let $f: \R^N \to \R$ be a $\mathcal C^2$ function whose gradient is Lipschitz continuous on the bounded sets, and such that $\inf_{\R^N} f >-\infty$. Let $\phi : \R^N \to \R_+$ be a damping potential (see Definition \ref{def1}) which is differentiable.
 Let $(x_n)_{n\in \N} $ be a bounded sequence  generated by the algorithm
\begin{equation}\label{prox-grad-alg-b}
x_{n+2}= x_{n+1}+ h{\prox}_{h\phi}\left(\frac{1}{h}(x_{n+1}-x_n)-h\nabla f(x_{n+1})\right).
\end{equation}
We make the following assumptions on the data $f$, $\phi$, and $h$:\\
\smallskip
\noindent $\bullet$ (assumption on $f$): Suppose that  the function  $H$ satisfies the {\rm (KL)} property, where
 $H: \R^N \times \R^N\to \R $ is defined for all $(x,y)\in  \R^N \times \R^N$ by
 $$
H(x,y):= f(x) + \frac{1}{2h^2}\|x-y\|^2 .
$$
\noindent$\bullet$ (assumption on $\phi$):
Suppose that   $\phi$ satisfies the following growth conditions:
\medskip
there exist positive $\gamma$, $\varepsilon$ and $\delta$ such that $\phi (u) \geq \gamma \|u\|^2$ for all $u$ in $\R^N$, and $\|\nabla \phi(u)\|\leq \delta \|u\|$ for all $u$ with 
$\|u\|\leq \varepsilon$.\\
\medskip
\noindent $\bullet$ (assumption on $h$):
Suppose that the step size $h$ is taken small enough to satisfy
$$
0 < h <  \frac{2\gamma}{L},
$$
where
  $L$  is the Lipschitz constant of  $\nabla f$ on the ball centered at the origin and with radius $R= \sup_{n \in \N} \|x_n\|$.\\
\smallskip
\noindent Then, the following properties are satisfied:\\
\medskip
\quad $(i)$ \, $x_n \to x_{\infty}$ as $n \to + \infty$, where $ x_{\infty}\in  \crit f$;\\
\medskip
\quad $(ii)$ \,  $\sum_{n\in \N}  \| x_{n+1} -  x_n\| < +\infty  $.
%$(iii)$ \, For $ \lambda $ small enough, and $ n $ large enough
%$$ \| x_n -  x_{\infty}\| \leq \frac{1}{\alpha} \theta \Big(  %E_{\lambda}(x_n,u_n) - E_{\lambda} (x_{\infty},0)\Big),
%$$
%
% where $\theta$ is the desingularizing function for $ E_{\lambda}$ at  %$(x_{\infty},0)$, and $\alpha$ enters the corresponding angle %condition.
\end{theorem}
\begin{proof} By assumption, the sequence $(x_n)_{n\in\N}$ is bounded. From Lemma \ref{basic-energy-lem-1} we have that $(u_n)_{n\in\N}$ tends to zero where $u_n= \frac{1}{h}\left(x_{n+1}-x_{n}\right)$. In addition,
$$
W_{n+1}  -W_n + h\left( \gamma - \demi Lh  \right)\|u_{n+1}\|^2 \leq 0
$$
for all $n \in \N$, where
$$
W_n:= \demi \|u_n\|^2 + f(x_{n+1}).
$$
Equivalently, by setting
$$
H(x,y)= f(x) + \frac{1}{2h^2}\|x-y\|^2,
$$
 we have for all $n\in \N$
\begin{equation}\label{alt-decr-H-bb}
 H(x_{n+2},x_{n+1}) + C\|x_{n+2}-x_{n+1}\|^2 \leq H(x_{n+1},x_{n}),
\end{equation}
where  $C=\frac{1}{h}\left( \gamma - \demi Lh  \right) > 0$.
The rest of the proof is classical in the framework of the KL theory; we refer the reader to \cite{BotCseLaEUR,ipiano} to similar techniques relying on the above decreasing property.
Relation \eqref{alt-decr-H-bb} implies that there exists
\begin{equation}\label{e-lim-H} \lim_{n\to+\infty}H(x_{n+1},x_n)\in\R.
\end{equation}
Further, let us denote by $\omega ((x_n)_{n \in \N})$ the set of cluster points of the sequence $(x_n)_{n\in \N}$, and by
$\crit f=\{x \in \R^N : \nabla f(x)=0\}$ the set of critical points of $f$.  \\
We easily derive 
$$\crit H=\{(x,x)\in\R^N\times\R^N:x\in \crit f\} $$
and notice that $\omega((x_n)_{n \in \N})\subseteq \crit f$, thus $\omega((x_{n+1},x_n)_{n \in \N})\subseteq \crit H.$ From \eqref{e-lim-H} one can easily conclude that $H$ is constant on $\omega((x_{n+1},x_n))$. Indeed, for $x^*\in\omega((x_n)_{n \in \N})$, we have from above (and the definition of $H$) that
\begin{equation}\label{lim_H-f}\lim_{n\to+\infty}H(x_{n+1},x_n)=f(x^*)=H(x^*,x^*).
\end{equation}
\noindent Assume now that $H$ satisfies the (KL) property with corresponding desingularizing function $\theta$.
We consider two cases.\\
I. There exists $\overline{n} \geq 0$ such that $H(x_{\overline{n}+1},x_{\overline{n}})=H(x^*,x^*)$.
From the decreasing property \eqref{alt-decr-H-bb} we obtain that $(x_n)_{n\geq \overline{n}}$ is a constant sequence and the conclusion follows.\\
II.  For all $n \geq 0$ we have $H(x_{n+1},x_{n}) > H(x^*,x^*)$. Since $\theta$ is concave and $\theta'>0$,
we derive from \eqref{alt-decr-H-bb} that there exists $n' \geq 0$ such that for all $n\geq n'$ it holds
\begin{eqnarray} \Delta_{n,n+1}: &=& \theta\left(H(x_{n+1},x_n)-H(x^*,x^*)\right) -
\theta\left(H(x_{n+2},x_{n+1})-H(x^*,x^*)\right)
\nonumber \\
&\geq & \theta'\left(H(x_{n+1},x_n)-H(x^*,x^*)\right)\cdot \left(H(x_{n+1},x_n)-H(x_{n+2},x_{n+1})\right)\nonumber \\
&\geq& \theta'\left(H(x_{n+1},x_n)-H(x^*,x^*)\right)\cdot C\cdot \|x_{n+2}-x_{n+1}\|^2 \nonumber\\
&\geq & \frac{C \|x_{n+2}-x_{n+1}\|^2}{\|\nabla H(x_{n+1},x_n)\|}, \label{last_ineq}
\end{eqnarray}
where the last inequality \eqref{last_ineq} follows from the uniformized (KL) property (\cite[Lemma 6]{BST}) applied to  the nonempty compact and connected set $\Omega=\omega((x_{n+1},x_n)_{n \in \N})$ (according to \cite[Remark 5]{BST} the connectedness of this set is generic for sequences satisfying $\lim_{n\to+\infty}(x_{n+1}-x_n)=0$). \\
Further, since $\nabla H(x,y)=(\nabla f(x)+\frac{1}{h^2}(x-y),\frac{1}{h^2}(y-x))$, we derive from \eqref{alt-alg}, 
the fact that $\lim_{n\to+\infty}(x_{n+1}-x_n)=0$ and the properties of $\phi$ that there exists $C_2>0$ such that
$$\|\nabla H(x_{n+1},x_n)\|\leq C_2(\|x_{n+2}-x_{n+1}\|+\|x_{n+1}-x_{n}\|) \ \quad \forall n \in \N.$$
Hence there exist $C_3>0$ and $n^{''} \in \N$ such that for all $n\geq n^{''}$
$$\frac{a_{n+1}^2}{a_{n+1}+a_n}\leq C_3\Delta_{n,n+1},$$
where $a_n:=\|x_{n+1}-x_n\|$. From here we get that for all $n\geq n^{''}$
$$a_{n+1}= \sqrt{C_3\Delta_{n,n+1}(a_{n+1}+a_n)}\leq \frac{a_{n+1}+a_n}{4} + C_3\Delta_{n,n+1},$$
which implies 
$$a_{n+1}\leq \frac{1}{3}a_n+\frac{4}{3}C_3\Delta_{n,n+1}.$$
Summing up the last inequality we obtain $\sum_{n\in \N}  \| x_{n+1} -  x_n\| < +\infty  $. This classically implies that $(x_n)_{n\in \N}$ is a Cauchy sequence in $\R^N$, hence it converges to a critical point of $f$.  
\end{proof}
\begin{remark}
$(i)$ If $f: \R^N \to \R$ is a $\mathcal C^2$ coercive function whose gradient is Lipschitz continuous on $\R^N$, then the boundedness of the sequence $(x_n)_{n \in \N}$ follows from \eqref{alt-decr-H-bb}. The function $H$ is a KL function if $f$ is, for instance, semialgebraic; we refer to \cite{LiPongFCM} for other results related to the preservation of the KL property under addition.\\
$(ii)$ For a general damping function $\phi$ we obtain at the limit
$$
\nabla f(x_{\infty}) + \partial \phi (0) \ni 0.
$$
When $\phi$ is differentiable at the origin, it attains its minimum at this point, and hence  $\nabla \phi (0)=0$. So,  we get $\nabla f(x_{\infty})=0$, \ie $x_{\infty}$ is a critical point of $f$, that's the situation considered above. In the case of dry friction, for example
$\phi(u)=r \|u\|$, we get $
\|\nabla f(x_{\infty}) \| \leq r
$
which gives  an approximate critical point, see \cite{AA0,AA,AAC}. \\
$(iii)$ The above result has been given as an illustration of our results,
showing that the continuous dynamic approach gives a valuable guideline to develop corresponding algorithmic results. In the particular case $\phi (u)=\|u\|^2$ one can also consult  \cite{GP}, \cite{ipiano}. The explicit discretization gives rise to inertial gradient algorithms, an interesting subject to explore in this general setting.
\end{remark}

\section{Closed-loop velocity control with Hessian driven damping}\label{Sec:Hessian}

\subsection{Hessian damping}
We propose to tackle questions similar to the previous sections, concerning the combination of closed-loop velocity control  with Hessian driven damping.
The following  system combines closed-loop velocity control with Hessian driven damping:
\begin{equation}\label{DIN}
\qquad \ddot{x}(t) + \partial \phi(\dot{x}(t)) +  \beta  \nabla^2  f (x(t)) \dot{x} (t) + \nabla  f (x(t)) = 0.
\end{equation}
 This autonomous system will be our main subject of study in this section.

 \smallskip

\noindent $\bullet$
The case  $\phi (u)= \frac{\gamma}{2} \|u\|^2$ of a fixed viscous coefficient   was first considered by
 Alvarez--Attouch--Bolte--Redont in \cite{AABR}.
In this case, (\ref{DIN}) can be equivalently written as a first-order system in time and space (different from the Hamiltonian formulation), which allows to extend this system naturally  to the case of a nonsmooth function $f$. This property has been exploited by Attouch--Maing\'e--Redont \cite{AMR}
 for modeling  non-elastic shocks in unilateral mechanics.
 To accelerate this system, several recent studies considered the case where the viscous damping is vanishing, that is
\begin{equation}\label{DIN_AVD}
\qquad \ddot{x}(t) + \frac{\alpha}{t}\dot{x}(t) +  \beta  \nabla^2  f (x(t)) \dot{x} (t) + \nabla  f (x(t)) = 0;
\end{equation}
see  \cite{APR}, \cite{ACFR},
 \cite{BCL},
 \cite{CBFP},  \cite{Kim},  \cite{LJ},  \cite{SDJS}, and Section
 \ref{sec:Hessian_intro} for the properties of this system.

 \smallskip

\noindent $\bullet$ The case $\phi(u)=\frac{\gamma}{2} \|u\|^2 + r\|u\|$ which combines viscous friction with dry friction and Hessian damping
has been considered by Adly--Attouch  \cite{AA-preprint-jca}, \cite{AA}.

 \smallskip

\noindent $\bullet$ By taking $\phi (u)= \frac{r}{p} \|u\|^p$, we get
\begin{equation}\label{Hessian_1}
\qquad \ddot{x}(t) + r\| \dot{x}(t) \|^{p-2} \dot{x}(t) +  \beta  \nabla^2  f (x(t)) \dot{x} (t) + \nabla  f (x(t)) = 0,
\end{equation}
for which we will  address issues similar to those of the previous theme.
In addition to the fast minimization property, one can expect obtaining too the fast convergence of the gradients to zero.

\subsection{Existence and uniqueness results}

Let us consider the  differential inclusion
\begin{equation}\label{Hessian_def_1}
\mbox{\rm (ADIGE-VH)}\quad  \ddot{x}(t) + \partial \phi (\dot{x}(t))+ \beta \nabla^2 f (x(t))\dot{x}(t)  + \nabla f (x(t)) \ni 0,
\end{equation}
which involves  a  damping potential $\phi$ (see Definition \ref{def1}), and a geometric damping driven by the Hessian of $f$.
The suffix V makes reference to the velocity and H  to the Hessian.
They both enter the damping terms.
It allows to cover different situations, in particular system \eqref{Hessian_1} corresponds  to $\phi (u)= \frac{r}{p} \| u\|^{p}$ for  $p>1$.
To prove  existence and uniqueness results for the associated Cauchy problem,  we  make additional assumptions. We assume that $f$ is convex, and
 that the Hessian mapping $x \in \cH  \mapsto \nabla^2 f (x) \in \mathcal L (\cH, \cH) $ is Lipschitz continuous on the bounded sets,
where $\mathcal L (\cH, \cH) $ is equipped with the norm operator.
Note that this property implies that  $\nabla f$ is Lipschitz continuous on the bounded subsets of $\cH$ (apply the mean value theorem in the vectorial case).
However, in the following statement, we formulate the two hypotheses for the sake of clarity.

\begin{theorem}\label{th.existence_uniqueness}
Let $f:\cH \to \R$ be a convex  function which is twice continuously differentiable, and    such that $\inf_{\cH} f >-\infty$. We suppose that

\medskip

$(i)$  $\nabla f$ is Lipschitz continuous on the bounded subsets of $\cH$;

\smallskip

$(ii)$  $\nabla^2 f$ is Lipschitz continuous on the bounded subsets of $\cH$.

\smallskip

\noindent Let  $\phi: \cH \to \R$ be a convex continuous damping function.
 Then, for  any Cauchy data  $(x_0, x_1 ) \in \cH \times \cH$, there exists a unique strong global solution $x : [0, +\infty[ \to \cH$ of
{\rm (ADIGE-VH)}
satisfying $x(0) = x_0$, and $\dot{x}(0)=x_1  $.
\end{theorem}
\begin{proof} To make the reading of the proof easier,  we distinguish several steps.

\smallskip

\textbf{Step 1}: \textit{A priori estimate}.  Let's  establish  a priori energy estimates on the solutions of \eqref{Hessian_def_1}. After taking the scalar product of
 \eqref{Hessian_def_1} with $\dot{x}(t) $, we get
$$
\frac{d}{dt} \mathcal E (t) + \left\langle \partial \phi (\dot{x}(t)), \dot{x}(t) \right\rangle + \beta
\left\langle \nabla^2 f (x(t))\dot{x}(t), \dot{x}(t) \right\rangle =0,
$$
where
$
\mathcal E (t):= f(x(t)) -\inf_{\cH}f + \demi \| \dot{x}(t) \|^2 $
is the global energy.
Since $\phi$ is a damping potential, the subdifferential inequality for convex functions, combined with $\phi(0)=0$, gives
$$
\left\langle \partial \phi (\dot{x}(t)), \dot{x}(t) \right\rangle  \geq \phi (\dot{x}(t)).
$$
Since $f$ is convex, we have that $\nabla^2 f$ is  positive semidefinite, which gives
$$
\left\langle \nabla^2 f (x(t))\dot{x}(t), \dot{x}(t) \right\rangle \geq 0.
$$
Collecting the above results, we obtain the following decay property of the energy
\begin{equation}\label{closed_loop_2b}
 \frac{d}{dt} \mathcal E (t) + \phi (\dot{x}(t)) \leq 0.
\end{equation}
Therefore, the energy is nonincreasing, which implies that,  as long as the trajectory is defined
\begin{equation}\label{closed_loop_2c}
  \| \dot{x}(t) \|^2  \leq 2 \mathcal E (0).
\end{equation}

\textbf{Step 2}: \textit{Hamiltonian formulation of \eqref{Hessian_def_1}}.
 According to  the Hamiltonian formulation of \eqref{Hessian_def_1}, it is equivalent
to solve  the first-order system
$$  \quad \left\{
\begin{array}{l}
\dot x(t)   -u(t) =0;  	 \\
\rule{0pt}{18pt}
 \dot{u}(t) +\partial \phi(u(t))  + \nabla f(x(t)) + \beta
\nabla^2 f (x(t))u(t) \ni 0 ,
 \hspace{2.3cm}
\end{array}\right.
$$
 with the Cauchy data
$x(0) =x_0$, \, $u(0)= x_1$.
Set
$Z(t) = (x(t), u(t)) \in \cH \times \cH .$\\
The above system can be written equivalently as
$$
\dot{Z}(t) + F( Z(t))\ni 0, \quad Z(0) = (x_0, x_1),
$$
where  $F: \cH \times \cH\rightrightarrows \cH \times \cH,\;\;(x,u)\mapsto F(x,u)$ is defined by
$$
F(x,u)= \Big( 0,  \partial \phi(u) \Big) +
 \Big( -u, \nabla f(x)  +\beta
\nabla^2 f (x)u  \Big).
$$
Hence $F$ splits as follows
$
F(x,u) = \partial \Phi (x,u) + G (x,u),
$
where
\begin{equation}\label{Hamilton_Hessian}
\Phi (x,u) =   \phi(u)
\, \mbox{ and } \,
G(x,u) = \Big(  -u, \, \nabla f(x)  +\beta
\nabla^2 f (x)u  \Big).
\end{equation}
Therefore, it is equivalent to solve the following first-order differential inclusion with Cauchy data
\begin{equation}
\label{1odd}
\dot{Z}(t) +\partial\Phi(Z(t)) + G( Z(t))\ni 0, \quad Z(0) = (x_0, x_1).
\end{equation}
Let us prove that  the mapping $(x,u)\mapsto G(x,u)$ is  Lipschitz continuous on the bounded subsets of $\cH\times\cH$.
 For any  $(x, u) \in \cH \times \cH $, set  $G(x,u) = (-u, K(x,u))$ where
 $$
 K(x,u):=   \nabla f(x)  +\beta \nabla^2 f (x)u .
$$
Let $L_R$ be the Lipschitz constant of $\nabla f$ and $\nabla^2 f$ on the ball centered at the origin and with radius $R$, and
set $M_R = \sup_{\|x\|\leq R} \|\nabla^2 f (x)  \|$.
Take $(x_i,u_i) \in \cH \times \cH$, $i=1,2$ with $\| (x_i,u_i)\| \leq R$.  We have
$$
K (x_2,u_2) -K (x_1,u_1)=  \nabla f(x_2)- \nabla f(x_1)   +\beta
(\nabla^2 f (x_2)u_2 -\nabla^2 f (x_1)u_1).
$$
According to  the triangle inequality, and the local Lipschitz continuity property  of $\nabla f$ and $\nabla^2 f$
\begin{eqnarray*}
\| K (x_2,u_2) -K (x_1,u_1) \| &\leq&  \|\nabla f(x_2)- \nabla f(x_1) \|
 +\beta
\| \nabla^2 f (x_2)u_2 -\nabla^2 f (x_1)u_2\| \\
& + &
\beta
\| \nabla^2 f (x_1)u_2 -\nabla^2 f (x_1)u_1\| \\
&\leq& L_R \| x_2 -x_1 \|
+\beta L_R \| x_2 -x_1 \|\|u_2\| + \beta M_R \| u_2 -u_1 \|\\
&\leq& L_R (1+ R\beta)\| x_2 -x_1 \|+ \beta M_R \| u_2 -u_1 \|.
\end{eqnarray*}
Therefore,
\begin{eqnarray}\label{Lip_G_1}
\| G (x_2,u_2) -G (x_1,u_1) \|&\leq& L_R (1+ R\beta)\| x_2 -x_1 \|+ (1+\beta M_R) \| u_2 -u_1 \|,
\end{eqnarray}
which gives  that the mapping $(x,u)\mapsto G(x,u)$ is  Lipschitz continuous on the bounded subsets of $\cH\times\cH$.

\smallskip

\textbf{Step 3}: \textit{Approximate dynamics}. We proceed in a similar way as in Theorem \ref{basic_exist_thm} (which corresponds to the  case $\beta=0$), and  consider the approximate dynamics
\begin{equation}\label{hbdf_lambda_existence}
 \ddot{x}_{\lambda}(t) +  \nabla \phi_{\lambda} (\dot{x}_{\lambda}(t)) + \beta \nabla^2 f (x_{\lambda}(t))\dot{x}_{\lambda}(t) + \nabla f (x_{\lambda}(t)) = 0,\; t\in [0,+\infty[
\end{equation}
 which uses the Moreau-Yosida approximates $(\phi_{\lambda})$ of $\phi$.
We will prove that the filtered sequence $(x_{\lambda})$
converges uniformly as $\lambda \to 0$ over the bounded time intervals towards a solution of  \eqref{Hessian_def_1}.
The Hamiltonian formulation of \eqref{hbdf_lambda_existence} gives the first-order (in time) system
$$  \quad \left\{
\begin{array}{l}
\dot x_{\lambda}(t)   -u_{\lambda}(t) =0;  	 \\
\rule{0pt}{18pt}
 \dot{u}_{\lambda}(t) +\nabla \phi_{\lambda}(u_{\lambda}(t)) + \nabla f(x_{\lambda}(t)) + \beta
\nabla^2 f (x_{\lambda}(t))u_{\lambda}(t) = 0 ,
 \hspace{2.3cm}
\end{array}\right.
$$
 with the Cauchy data
$x_{\lambda}(0) =x_0$, \, $u_{\lambda}(0)= x_1       $.
Set
$Z_{\lambda}(t) = (x_{\lambda}(t), u_{\lambda}(t)) \in \cH \times \cH .$\\
The above system can be written equivalently as
$$
\dot{Z}_{\lambda}(t) + F_{\lambda}( Z_{\lambda}(t))\ni 0, \quad Z_{\lambda}(t_0) = (x_0, x_1),
$$
where  $F_{\lambda}: \cH \times \cH\rightarrow \cH \times \cH,\;\;(x,u)\mapsto F_{\lambda}(x,u)$ is defined by
$$
F_{\lambda}(x,u)= \Big( 0,  \nabla \phi_{\lambda}(u) \Big) +
 \Big( -u, \nabla f(x)  +\beta
\nabla^2 f (x)u  \Big).
$$
Hence $F_{\lambda}$ splits as follows
$
F_{\lambda}(x,u) = \nabla \Phi_{\lambda} (x,u) + G (x,u)
$
where  $\Phi $ and  $G$ have been defined in \eqref{Hamilton_Hessian}.
Therefore, the approximate equation is equivalent to the  first-order differential system with Cauchy data
\begin{equation}
\label{1odd_existence_b}
\dot{Z}_{\lambda}(t) +\nabla \Phi_{\lambda}(Z_{\lambda}(t)) + G( Z_{\lambda}(t))= 0, \quad Z_{\lambda}(0) = (x_0, x_1).
\end{equation}
Let's argue with $\lambda >0$ fixed.
 According to the Lipschitz continuity of $\nabla \Phi_{\lambda}$,  and the fact that $G$ is Lipschitz continuous on the bounded sets, we have that the sum operator $ \nabla \Phi_{\lambda} + G$ which governs \eqref{1odd_existence_b} is  Lipschitz continuous on the bounded sets.
As a consequence, the existence of a local solution to \eqref{1odd_existence_b} follows from the classical Cauchy--Lipschitz theorem.
To pass from a local solution to a global solution, we use the a priori estimate  obtained in Step 1 of the proof. Note that this estimate is valid for any damping potential, in particular for $\phi_{\lambda}$.
According to the Cauchy data, and $f$ minorized,
this implies that, on any bounded time interval, the functions
$(x_{\lambda})$ and $(\dot{x}_{\lambda}) $ are bounded.
According to the property \eqref{ineq_phi_b} of the Yosida approximation, and the property $(iii)$ of the
 damping potential $\phi$, this implies that
$$
\| \nabla \phi_{\lambda} (x_{\lambda}(t))\| \leq \| (\partial \phi )^{0} (x_{\lambda}(t))\|
$$
is also bounded uniformly with respect to  $t$ bounded.
Moreover, according to the local boundedness assumption made  on the gradient and the Hessian of $f$, we have that $\nabla f (x_{\lambda}(t))$ and
$\nabla^2 f (x_{\lambda}(t))\dot{x}_{\lambda}(t)$ are also bounded.
According to the constitutive equation \eqref{hbdf_lambda_existence}, this in turn implies that  $(\ddot{x}_{\lambda} )$ is also bounded.
This implies that if a maximal solution is defined on a finite time interval $[0, T[$, then the limits of $x_{\lambda}(t)$ and $\dot{x}_{\lambda} (t)$
exist, as $t \to T$. According to this property, passing from a local to a global solution
is a classical argument.
So for any $\lambda >0$ we have a unique global solution of
\eqref{hbdf_lambda_existence} with satisfies the Cauchy data $x_{\lambda}(0) =x_0$, $\dot{x}_{\lambda}(0)= x_1 $.

\smallskip

\textbf{Step 4}: \textit{Passing to the limit as $\lambda \to 0$}.
Take  $T >0$, and $ \lambda , \mu >0$.
Consider the corresponding solutions on $[0, T]$
\begin{eqnarray*}
&&  \dot{Z}_{\lambda}(t) +\nabla \Phi_{\lambda}(Z_{\lambda}(t)) + G( Z_{\lambda}(t))= 0, \quad Z_{\lambda}(0) = (x_0, x_1)
\\
&&\dot{Z}_{\mu}(t) +\nabla \Phi_{\mu}(Z_{\mu}(t)) + G( Z_{\mu}(t))= 0, \quad Z_{\mu}(0) = (x_0, x_1).
\end{eqnarray*}
Let's make the difference between the two equations, and take the scalar product by $Z_{\lambda}(t) - Z_{\mu}(t)$. We get
\begin{eqnarray}
\demi \frac{d}{dt}\| Z_{\lambda}(t) - Z_{\mu}(t) \|^2 &+ &
\left\langle  \nabla \Phi_{\lambda}(Z_{\lambda}(t)) - \nabla \Phi_{\mu}(Z_{\mu}(t)) ,  Z_{\lambda}(t) - Z_{\mu}(t) \right\rangle \nonumber\\
&+& \left\langle  G( Z_{\lambda}(t)) - G( Z_{\mu}(t)) ,  Z_{\lambda}(t) - Z_{\mu}(t) \right\rangle =0 . \label{basic_ex_Y_b}
\end{eqnarray}
We now use the following ingredients:

\medskip

 i) According to the  properties of the Yosida approximation (see \cite[Theorem 3.1]{Brezis}), we have
$$
\left\langle  \nabla \Phi_{\lambda}(Z_{\lambda}(t)) - \nabla \Phi_{\mu}(Z_{\mu}(t)) ,  Z_{\lambda}(t) - Z_{\mu}(t) \right\rangle
\geq -\frac{\lambda}{4} \|\nabla \Phi_{\mu}(Z_{\mu}(t))  \|^2 -
\frac{\mu}{4} \|\nabla \Phi_{\lambda}(Z_{\lambda}(t))  \|^2.
$$
According to  the  energy estimates, the sequence $(Z_{\lambda})$ is uniformly bounded on $[0, T]$, let
$$\| Z_{\lambda}(t)\|\leq C_T .$$
 From these properties we immediately infer
$$
 \|\nabla \Phi_{\lambda}(Z_{\lambda}(t))  \|  \leq \sup_{\|\xi\|\leq C_T} \|(\partial \phi)^0(\xi)  \|= M_T <+\infty,
$$
because our assumption on $\phi$ gives that $(\partial \phi)^0$ is bounded on the bounded sets.
Therefore
$$
\left\langle  \nabla \Phi_{\lambda}(Z_{\lambda}(t)) - \nabla \Phi_{\mu}(Z_{\mu}(t)) ,  Z_{\lambda}(t) - Z_{\mu}(t) \right\rangle
\geq -\frac{1}{4} M_T (\lambda +\mu).
$$

ii)  Since the mapping $G : \cH \times  \cH  \to \cH \times  \cH$ is Lipschitz continuous on the bounded sets, and
using again that the sequence $(Z_{\lambda})$ is uniformly bounded on $[0, T]$, we deduce that there exists a constant $L_T$ such that
$$
\|  G( Z_{\lambda}(t)) - G( Z_{\mu}(t)) \| \leq L_T \|   Z_{\lambda}(t) -  Z_{\mu}(t) \|.
$$
Combining the above results, and using Cauchy--Schwarz inequality, we deduce from
\eqref{basic_ex_Y_b} that
$$
\demi \frac{d}{dt}\| Z_{\lambda}(t) - Z_{\mu}(t) \|^2
\leq  \frac{1}{4} M_T (\lambda +\mu) + L_T \|   Z_{\lambda}(t) -  Z_{\mu}(t) \|^2 .
$$
We now proceed with the integration of this differential inequality.
According to the fact that $ Z_{\lambda}(0) - Z_{\mu}(0) =0$, elementary calculus gives
$$
\| Z_{\lambda}(t) - Z_{\mu}(t) \|^2 \leq \frac{M_T}{4L_T}(\lambda +\mu) \Big( e^{2L_T (t-t_0}  -1  \Big).
$$
Therefore, the filtered sequence $(Z_{\lambda})$  is a Cauchy sequence for the uniform convergence on $[0, T]$, and hence it converges uniformly.
This means the uniform convergence on $[0, T]$ of $x_{\lambda}$ and $\dot{x}_{\lambda}$ to $x$ and $\dot{x}$ respectively.
Proving that $x$ is solution of \eqref{Hessian_def_1}  is obtained in a similar way as
in Theorem \ref{basic_exist_thm}. Just rely on the classical derivation chain rule
$\frac{d}{dt}\left( \nabla f (x_{\lambda}(t)) \right) =  \nabla^2 f (x_{\lambda}(t))\dot{x}_{\lambda}(t)          $ to pass to the limit on the Hessian term.
\end{proof}

\subsection{Convergence based on the quasi-gradient approach}\label{rem-quasi-hessian}

Our objective is to address, from the perspective of quasi-gradient systems, the  system (ADIGE-VH)
\begin{equation}\label{Hessian_def_1-quasi}
 \ddot{x}(t) + \nabla \phi (\dot{x}(t))+ \beta \nabla^2 f (x(t))\dot{x}(t)  + \nabla f (x(t)) = 0,
\end{equation}
as it was done in Section \ref{Sec-conv-KL}.
 We assume that $\cH = \R^N$ is a finite-dimensional Hilbert space, and that the hypotheses of
Theorem \ref{quasi_grad_thm_2} and Theorem \ref{th.existence_uniqueness} hold.
We follow the steps of the proof of Theorem \ref{quasi_grad_thm_2}.
By using the estimates in Step 1 of the proof of Theorem \ref{th.existence_uniqueness}, we easily derive the first part of
the proof of Theorem \ref{quasi_grad_thm_2}, namely that the trajectory
$t \mapsto (x(t),\dot{x}(t))$ in the phase space $\R^N \times  \R^N$
belongs to the closed bounded set $\Gamma = \bar{B}(0,R) \times \bar{B}(0,\epsilon) $.
According to Step 2 in the proof of Theorem \ref{th.existence_uniqueness}, the  Hamiltonian formulation of  \eqref{Hessian_def_1-quasi} gives the first-order differential system
\begin{equation}\label{first_order_cl_loop_quasi_1-hessian}
\dot z(t) + F(z(t)) =0,
\end{equation}
where $z(t)=(x(t), \dot x(t)) \in \R^N \times \R^N $, and
   $F: \R^N \times  \R^N \to \R^N \times \R^N$
is defined by
$$F(x,u)=(-u, \nabla \phi(u)+ \nabla f(x)+\beta\nabla^2f(x)u).
$$
Let's focus on the key point which is the angle condition ($E_{\lambda}$ is defined as in Theorem \ref{quasi_grad_thm_2}). We have
\begin{center}
$\left\langle  \nabla  E_{\lambda}(x,u), F(x,u) \right\rangle
= \left\langle  \Big( \nabla f (x)+ \lambda \nabla^2 f (x)u, \, u + \lambda \nabla f (x) \Big), \Big(-u, \nabla \phi(u)+ \nabla f(x) + \beta\nabla^2f(x)u\Big) \right\rangle .$
\end{center}
After development and simplification, we get
\begin{eqnarray*}
\left\langle  \nabla  E_{\lambda}(x,u), F(x,u) \right\rangle
&=&  - \lambda \left\langle  \nabla^2 f (x)u, \, u  \right\rangle + \left\langle  u, \, \nabla \phi(u) \right\rangle
+ \lambda \left\langle  \nabla f (x) , \, \nabla \phi(u)  \right\rangle
+ \lambda  \|  \nabla f (x) \|^2\\
&& + \beta\left\langle u+\lambda\nabla f(x), \nabla^2 f (x)u\right\rangle\\
&\geq& - \lambda \left\langle  \nabla^2 f (x)u, \, u  \right\rangle + \left\langle  u, \, \nabla \phi(u) \right\rangle
+ \lambda \left\langle  \nabla f (x) , \, \nabla \phi(u)  \right\rangle
+ \lambda  \|  \nabla f (x) \|^2\\
&& + \lambda\beta\left\langle \nabla f(x), \nabla^2 f (x)u\right\rangle,
\end{eqnarray*}
where we used that $\nabla^2 f(x)$ is positive semidefinite.
The only difference with respect to the next step in the proof of Theorem \ref{quasi_grad_thm_2} is that we need
to estimate the extra term $\lambda\beta\left\langle \nabla f(x), \nabla^2 f (x)u\right\rangle$. We do this by writing
$\lambda\beta\left\langle \nabla f(x), \nabla^2 f (x)u\right\rangle\geq -\frac{\lambda}{4}\|\nabla f(x)\|^2
-\lambda\beta^2 M^2\|u\|^2$,
and get
\begin{eqnarray}
\left\langle  \nabla  E_{\lambda}(x,u), F(x,u) \right\rangle
&\geq &   \Big( \gamma -\lambda M   - \frac{\lambda}{2} \delta^2 -\lambda \beta^2M^2 \Big)  \|u\|^2
  + \frac{\lambda}{4}  \|  \nabla f (x) \|^2  . \label{quasi_gradient_2_hessian}
\end{eqnarray}
Take $ \lambda $ small enough to satisfy
$
\gamma > \lambda \left(M +  \frac{\delta^2}{2} +\beta ^2M^2  \right).
$
Then
\begin{eqnarray}
\left\langle  \nabla  E_{\lambda}(x,u), F(x,u) \right\rangle
&\geq &   \alpha_0 ( \|u\|^2
  +   \|  \nabla f (x) \|^2),   \label{quasi_gradient_3_hessian}
\end{eqnarray}
with $\alpha_0:= \min\{\gamma - \lambda \left(M +  \frac{\delta^2}{2} +\beta^2M^2     \right)  , \frac{\lambda}{4}  \}$.
On the other hand, as in Theorem \ref{quasi_grad_thm_2},
\begin{eqnarray*}
\| \nabla  E_{\lambda}(x,u)\| &\leq& C_1 ( \|u\|^2
  +   \|  \nabla f (x) \|^2)^{\demi} \\
\| F(x,u)\| &\leq& C_2  ( \|u\|^2
  +   \|  \nabla f (x) \|^2)^{\demi} ,
\end{eqnarray*}
where $C_2=\sqrt{4+3\delta^2+3\beta^2M^2}$.
Therefore
\begin{eqnarray}
\| \nabla  E_{\lambda}(x,u)\|  \| F(x,u) \|
&\leq &   C_1 C_2 ( \|u\|^2
  +   \|  \nabla f (x) \|^2).   \label{quasi_gradient_4_hessian}
\end{eqnarray}
Therefore, for
$
\alpha := \frac{\alpha_0 }{C_1C_2},
$
the angle condition
$
\left\langle  \nabla E (z), F(z) \right\rangle
\geq \alpha \|\nabla E (z) \|  \| F(z) \|
$
is satisfied on $\Gamma$.
Let us summarize the above results.
\begin{theorem} \label{quasi_thm_VH}
Let $f:\cH \to \R$ be a convex  function which is twice continuously differentiable, and    such that $\inf_{\cH} f >-\infty$. We suppose that

\medskip

$(i)$  $\nabla f$ is Lipschitz continuous on the bounded subsets of $\cH$;

\smallskip

$(ii)$  $\nabla^2 f$ is Lipschitz continuous on the bounded subsets of $\cH$.

\smallskip

\noindent  Let  $ E_{\lambda}: \R^N \times \R^N\to \R $ be  defined by: for all $(x,u)\in  \R^N \times \R^N$
 $$
 E_{\lambda}(x,u):= \demi \|u\|^2 + f(x) +\lambda \left\langle \nabla f (x), u\right\rangle.
$$
Suppose that    $ E_{\lambda}$ satisfies the {\rm (KL)} property.
Let $\phi : \R^N \to \R_+$ be a damping potential (see Definition \ref{def1}) which is differentiable, and which satisfies the following growth conditions (i), and (ii):

\medskip

$(i)$ (local) there exists positive constants $\gamma$,  $\delta$, and  $\epsilon >0$ such that, for all $u$ in
$\R^N$ with $\|u\| \leq \epsilon$
$$\phi (u) \geq \gamma \|u\|^2   \mbox{ and }    \|\nabla \phi (u) \| \leq \delta \|u\|.$$

$(ii)$ (global) there exists  $p\geq 1$, $c>0$, such that for all $u$ in $\R^N$, $\phi (u) \geq c\|u\|^p$.

\medskip

\noindent  Let $x: [0, +\infty[ \to \mathbb R^N $ be a bounded solution trajectory of
$$
 \ddot{x}(t) + \nabla \phi (\dot{x}(t))+ \beta \nabla^2 f (x(t))\dot{x}(t)  + \nabla f (x(t)) = 0.
$$
Then, the following properties are satisfied:

\medskip

$(i)$ \, $x(t) \to x_{\infty}$ as $t \to + \infty$, where $ x_{\infty}\in  \crit f$;

\medskip

$(ii)$ \,  $\dot{x} \in L^1 (0, +\infty; \R^N )$ , $\dot{x}(t) \to 0$ as $t \to + \infty$;

\medskip

$(iii)$ \,  For $\lambda$ sufficently small, and $t$ sufficiently large
$$ \| x(t) -  x_{\infty}\| \leq \frac{1}{\alpha} \theta \Big(  E_{\lambda}(x(t),u(t)) - E_{\lambda} (x_{\infty},0)\Big)
$$

\smallskip

\noindent where $\theta$ is the desingularizing function for $ E_{\lambda}$ at  $(x_{\infty},0)$, and $\alpha$ enters the  angle condition.
\end{theorem}

\subsection{Numerical illustrations}\label{sec:num_H}
We revisit the numerical examples of section \ref{sec:num} where we introduce an additional Hessian damping.
\begin{small}
\begin{figure}[h!]
	\centering
	%\captionsetup[subfigure]{position=top}
	%\subfloat[...]
	{\includegraphics*[viewport=78 200  540 600,width=0.325\textwidth]{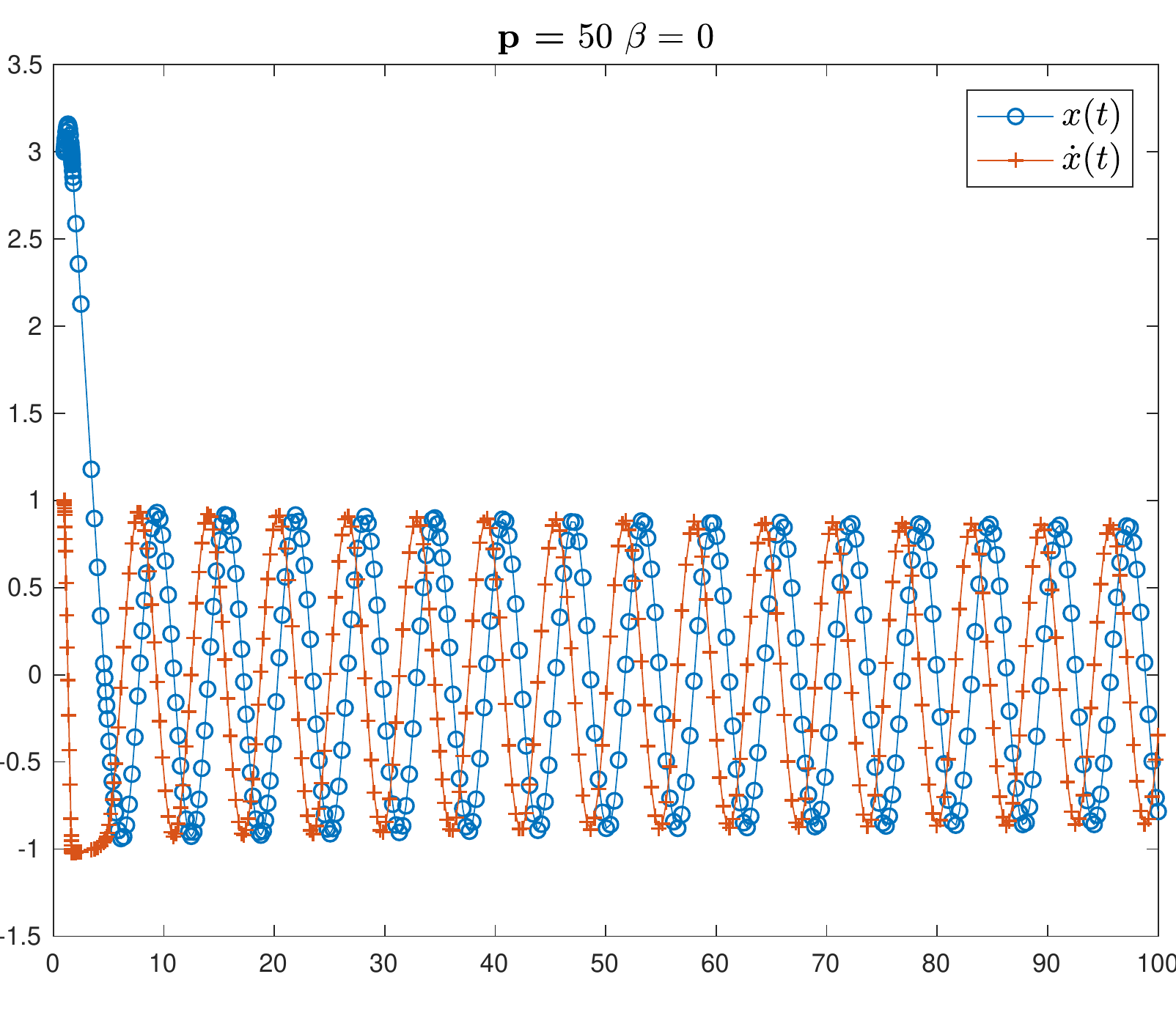}}\hspace{0.03cm}
	{\includegraphics*[viewport=78 200  540 600,width=0.325\textwidth]{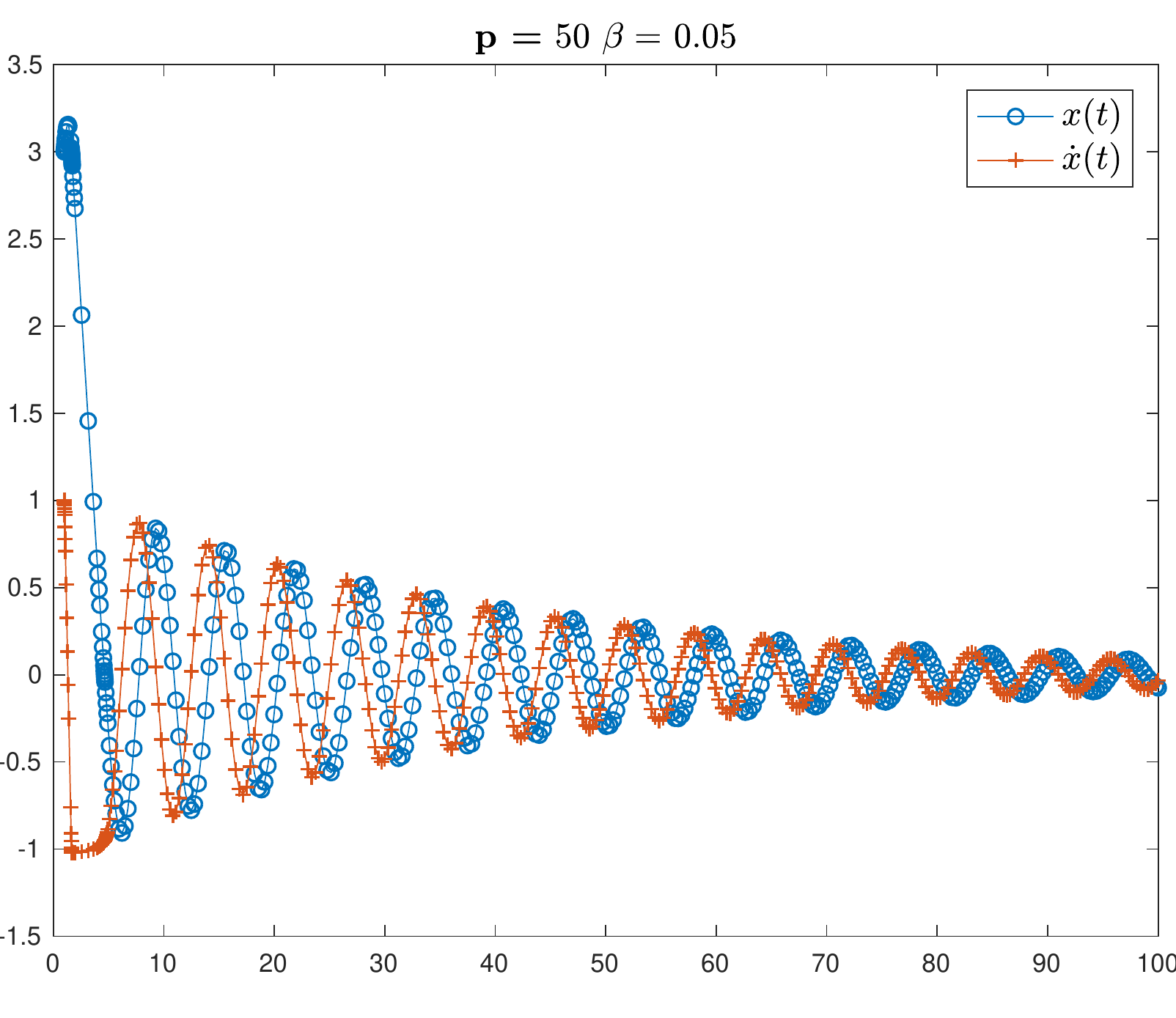}}\hspace{0.03cm}
	{\includegraphics*[viewport=78 200  540 600,width=0.325\textwidth]{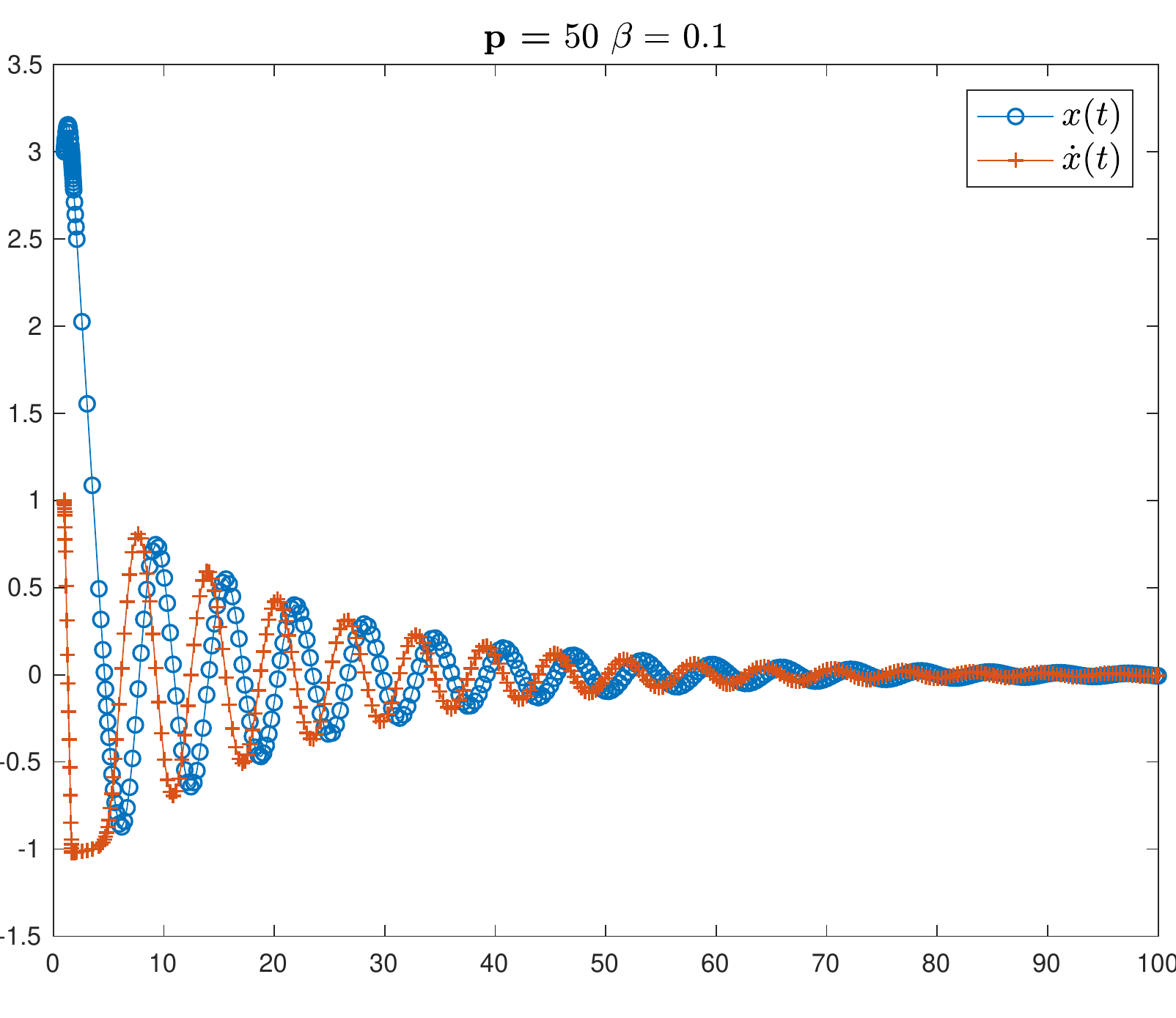}}\\
	%\subfloat[...]
	{\includegraphics*[viewport=78 200  540 600,width=0.325\textwidth]{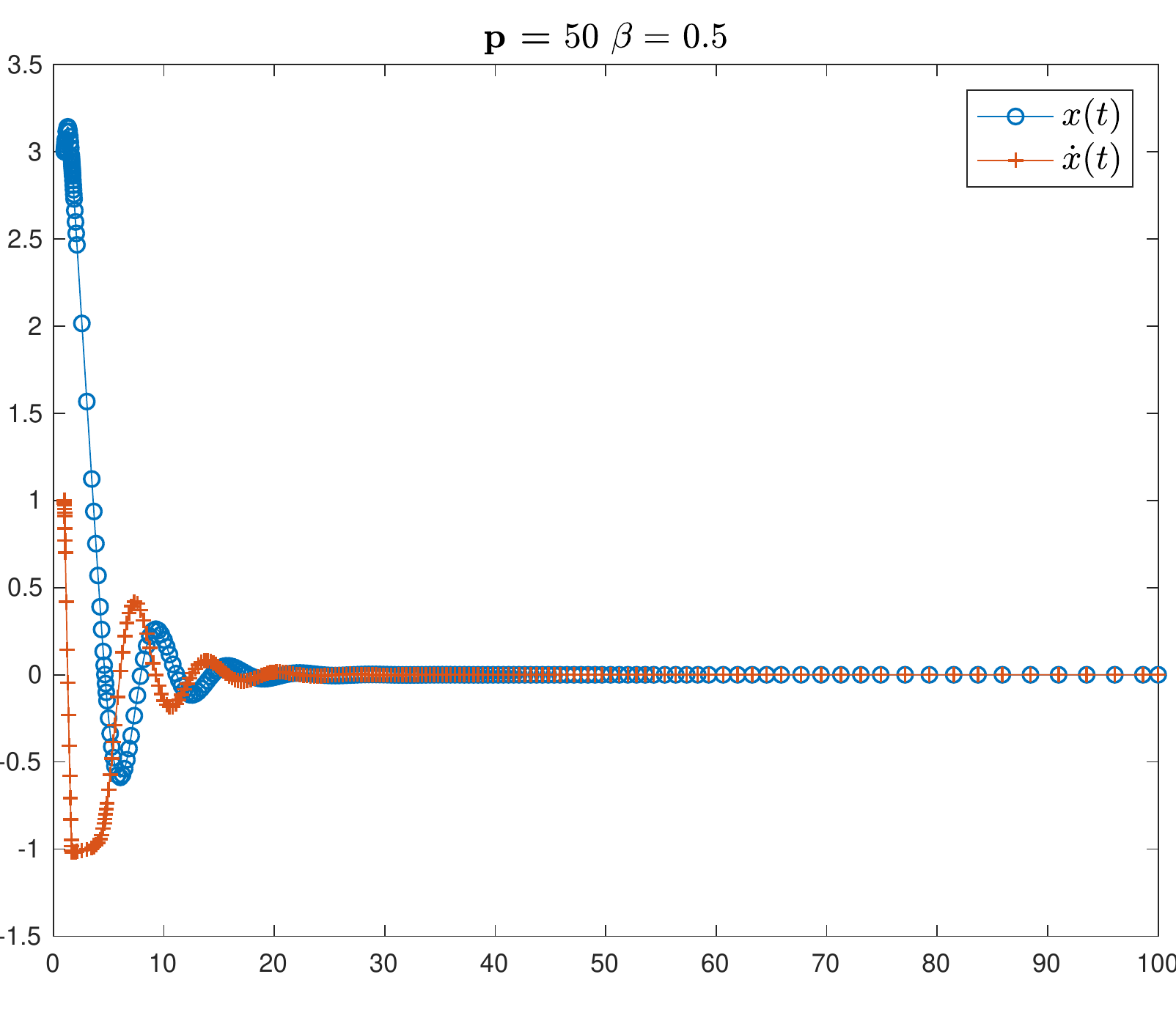}}\hspace{0.03cm}
	%\subfloat[...]	
	{\includegraphics*[viewport=78 200  540 600,width=0.325\textwidth]{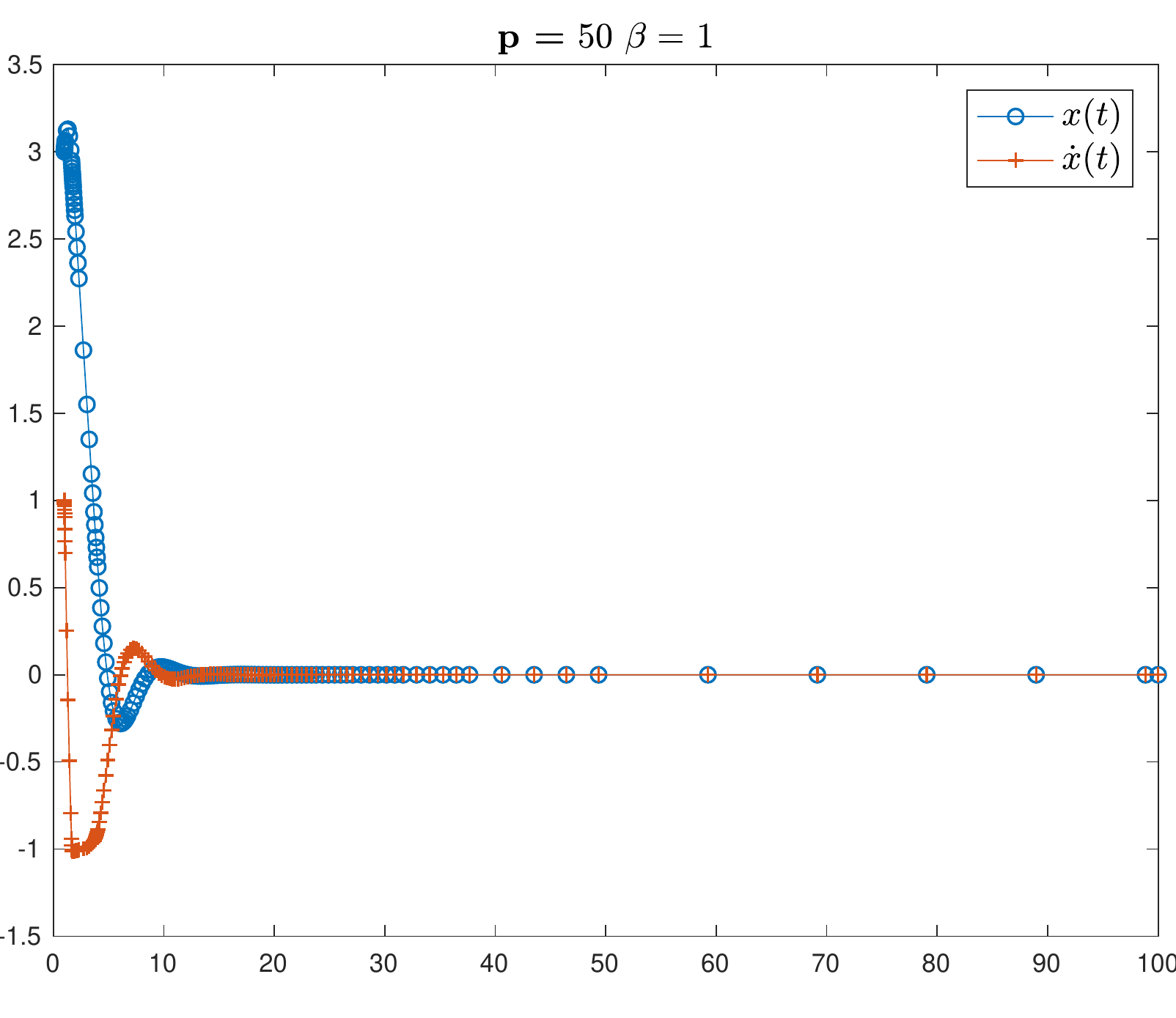}}\hspace{0.03cm}
	%\subfloat[...]
	{\includegraphics*[viewport=78 200  540 600,width=0.325\textwidth]{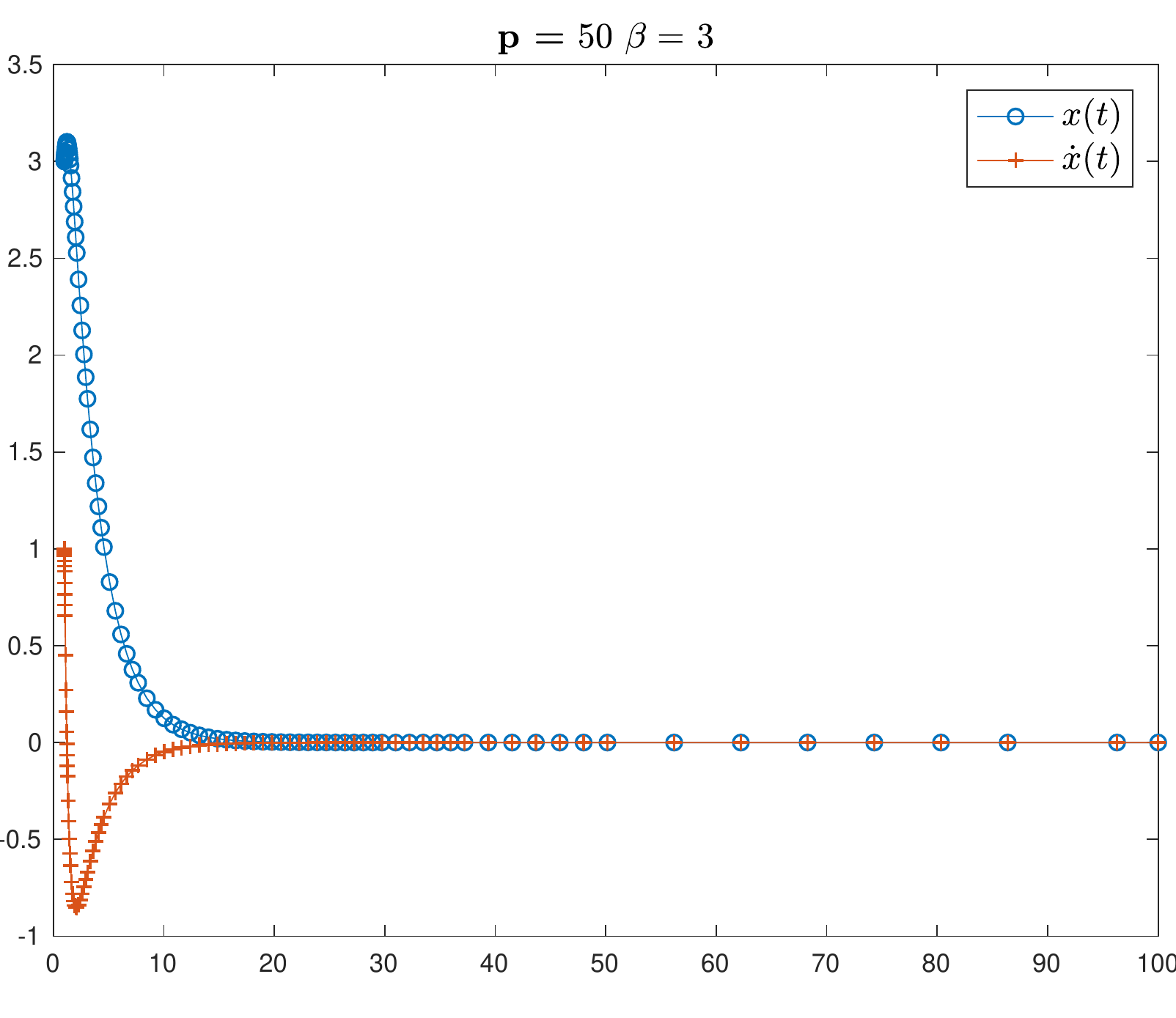}}\\
	%\subfloat[...]
	{\includegraphics*[viewport=78 200  540 600,width=0.325\textwidth]{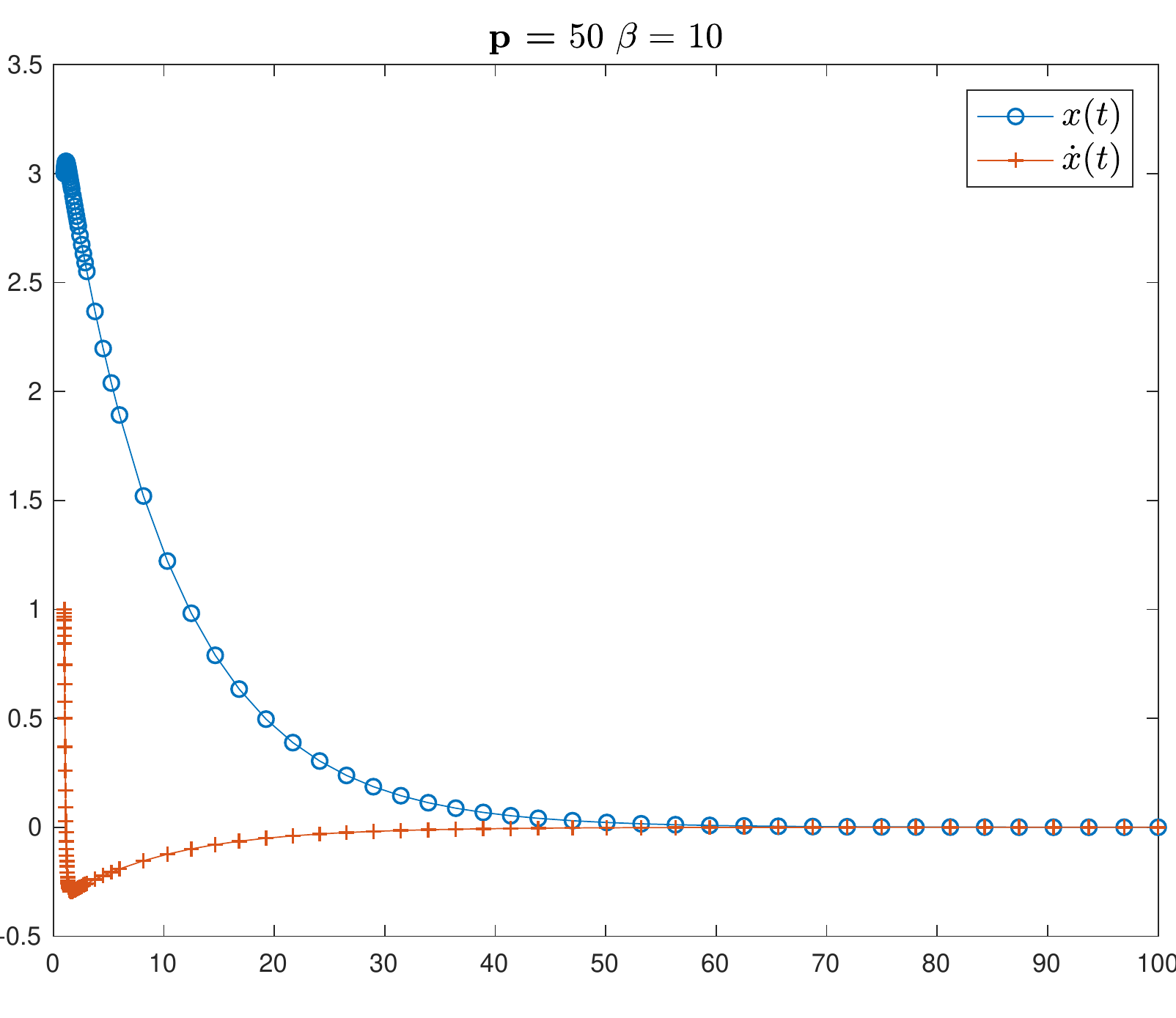}}\hspace{0.03cm}
	%\subfloat[...]
{\includegraphics*[viewport=78 200  540 600,width=0.325\textwidth]{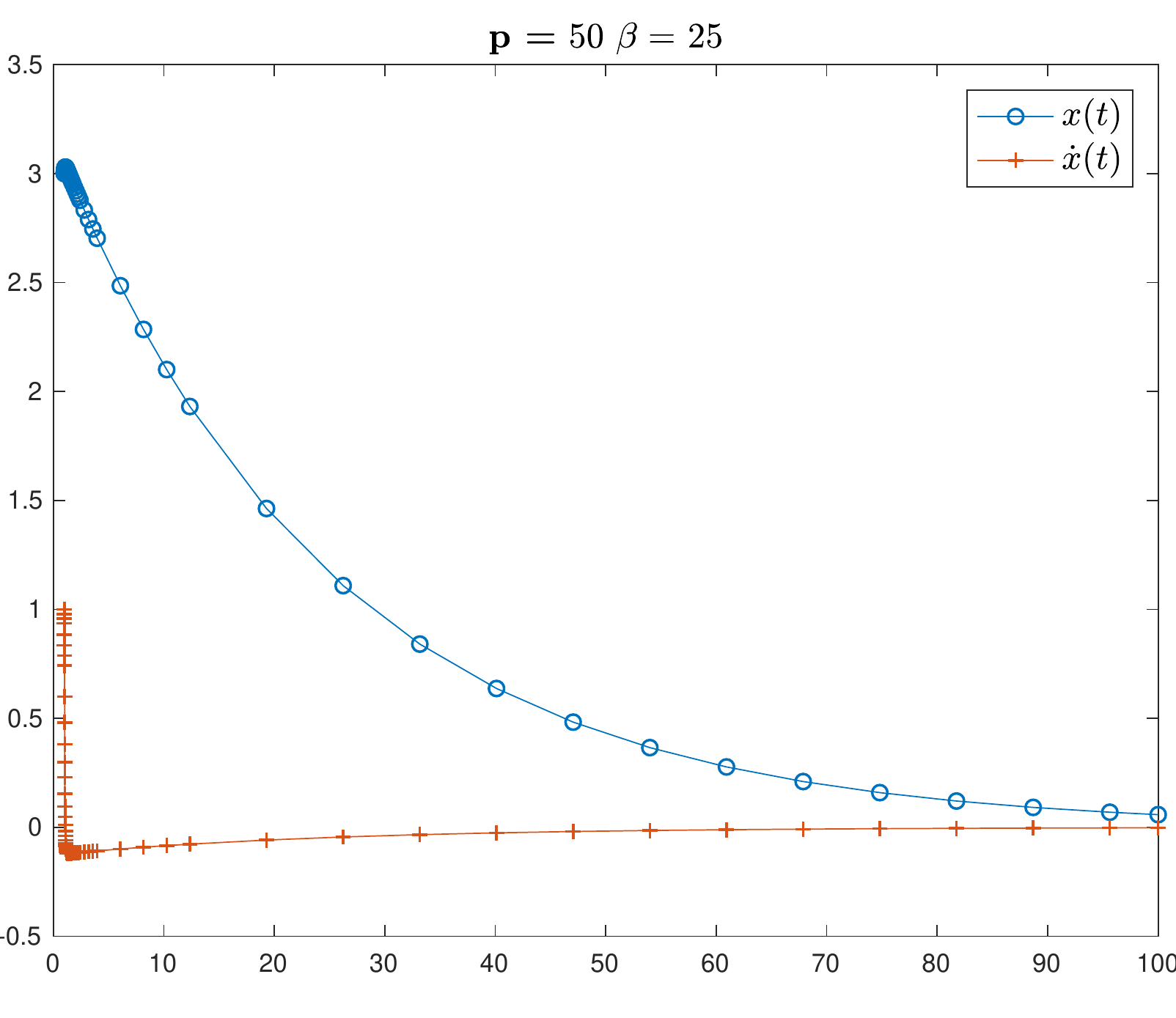}}\hspace{0.03cm}
	{\includegraphics*[viewport=78 200  540 600,width=0.325\textwidth]{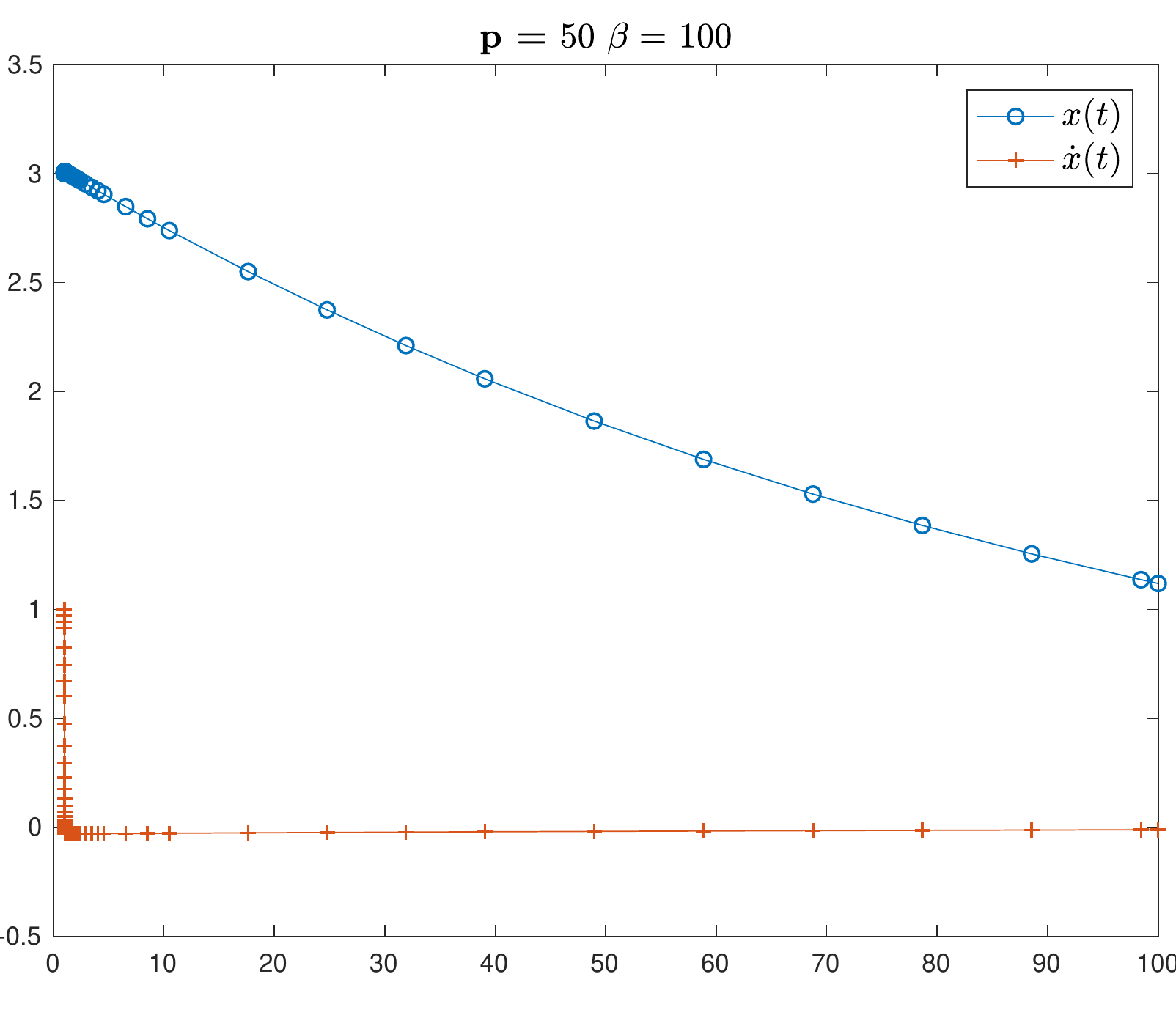}}
	\caption{\small Evolution of the trajectories $x(t)$ (blue) and $\dot x(t)$ (red line) of  \eqref{adige_v_one_dim_H} for different values of $\beta $.
}
	\label{fig:ex3}	
\end{figure}
\end{small}

%\vspace{3mm}

So, we take $\cH = \R$, $f(x) = \demi |x|^2$, and $\phi (u) = \frac{1}{p} |u|^p$ with $p>1$.
Then, (ADIGE-VH) writes
\begin{equation}\label{adige_v_one_dim_H}
\ddot x(t) +  |\dot x(t)|^{p-2} \dot x(t)+ \beta  \dot x(t) + x(t)=0.
\end{equation}

For $\beta >0$, we are in the framework of Theorem \ref{strong_convex_thm}, with $\phi (u)= \frac{\beta}{2}|u|^2 + \frac{p}{2}|u|^p $. So we have convergence at an exponential rate of $x(t)$ and $\dot{x}(t)$ towards zero. This makes a big contrast with the case $\beta=0$, for which we have convergence towards zero, but with many oscillations in the case of weak damping ($p$ large).
Note that that even for very small $\beta >0$, we have a rapid stabilization of the trajectory towards the origin. On the other hand, taking large $\beta$ is not beneficial, we can observe on Figure \ref{fig:ex3}  that the quality of convergence is degraded in this case.
Indeed, since the damping attached to $|\dot x(t)|^{p-2} \dot x(t)$ is negligeable for large $p$ with respect   to the damping attached to  $\beta \dot x(t)$, the "optimal" value of $\beta$ is close  to the optimal value for (HBF). So, according to Theorem \ref{strong-conv-thm}, it is close  to  $2\sqrt{\mu}$ where $\mu$ is the coefficient of strong convexity of $f$ (see Theorem \ref{strong-conv-thm}). In our situation, this gives
$\beta \sim 2$.

\subsection{Link with the regularized Newton method}

Let us specify the link between our study and Newton's method for solving $Ax \ni0$, where $A$ is a general maximally monotone operator (for convex minimization take $A =\partial f$).
To overcome the ill-posed character of the  continuous Newton method, the following first-order evolution system  was studied by
Attouch--Svaiter \cite{ASv},
\begin{equation*}
 \left\{
\begin{array}{l}
v(t)  \in     A(x(t)) \hspace{2cm} \\
\rule{0pt}{15pt}
 \gamma(t)  \dot{x}(t)  + \beta   \dot{v}(t) + v(t) =0 .
 \end{array}\right.
\end{equation*}
The system can be considered as a continuous version of the
Levenberg--Marquardt, which acts as a regularization of the Newton method.
Under a fairly general assumption on the regularization parameter $\gamma (t)$, this system is well-posed and generates trajectories that converge weakly to equilibria. Parallel results have been obtained for the associated proximal algorithms obtained by implicit temporal discretization, see   \cite{AAS}, \cite{AMAS}, \cite{ARS}.
Formally when $A$ is differentiable, this system writes as
$
\gamma(t)  \dot{x}(t) + \beta   \frac{d}{dt} \left( A(x(t))\right)   + A(x(t)) = 0.
$
When $A =\nabla f$ we obtain
\begin{equation}\label{ARS_Newton}
\gamma(t)  \dot{x}(t) + \beta   \nabla^2 f (x(t)) \dot{x}(t)  + \nabla f(x(t)) = 0.
\end{equation}
 The system (ADIGE-VH) considered in the previous section can be seen as an inertial version of the above system \eqref{ARS_Newton}.
Most interesting, Attouch--Redont--Svaiter developed in \cite{ARS} a closed-loop version of the above results.
% see  Attouch-Marques Alves-Svaiter  \cite{AMAS} for the corresponding algorithmic results.
They showed the convergence of the trajectories generated by the closed-loop control system when $0<p<1$, where $A$ is a general maximally monotone operator:
\begin{equation*}
\left\{
\begin{array}{l}
v(t) \in A(x(t)) \vspace{3mm}\\
 \| v(t) \|^p  \dot{x}(t) +  \dot{v}(t) + v(t)=0 \vspace{3mm}\\
 x(0) =x_0, \, v(0)\in A (x_0),  v_0 \neq 0.
 \end{array}\right.
\end{equation*}

\noindent For optimization problems, this naturally suggests to consider autonomous inertial systems where the damping coefficient is a closed-loop control of the gradient of $f$.
A first answer to this question has been obtained by Lin--Jordan \cite{LJ} who considered the autonomous system
\begin{equation}\label{general_coef}
 \ddot{x}(t) + \gamma (t)  \dot{x}(t) +  \beta (t) \nabla^2  f (x(t)) \dot{x} (t) + b(t)\nabla  f (x(t)) = 0,
\end{equation}
where $\gamma$, $\beta$ and $b$ are defined by the following formulas:.
\begin{equation}\label{general_coef_b}
\left\{
\begin{array}{l}
 | \lambda(t)|^p   \|  \nabla  f (x(t)) \|^{p-1} =\theta \vspace{2mm}
 \\
 a(t)=   \frac{1}{4}\left(  \int_0^t \sqrt{\lambda (s)}ds +c \right)^2
\vspace{2mm}\\
 \gamma(t) = 2 \frac{\dot{a}(t)}{a(t)} - \frac{\ddot{a}(t)}{\dot{a}(t)} \vspace{2mm} \\
 \beta (t) = \left(\frac{\dot{a}(t)}{a(t)}\right)^2  \vspace{2mm}\\
 b(t)= \frac{\dot{a}(t)( \dot{a}(t) + \ddot{a}(t)  )}{a(t)}
\end{array}\right.
\end{equation}

As a specific feature, the damping coefficients are expressed with the help of $\lambda (t)$ which is equal to a power of the inverse of the norm of the gradient of $f$.
The authors give some interesting non-trivial convergence rates for values. According to the presence of the Hessian driven damping term, they show the fast convergence towards zero of the gradient norms.

\section{Closed-loop damping involving the velocity and the gradient}
\label{Sec: combine}
Let's consider the following system, where the damping term $\partial \phi \Big(\dot{x}(t) + \beta \nabla  f (x(t)\Big)$ involves both the velocity vector and  the gradient of the potential function $f$
\begin{equation}\label{closed_loop_inertial_both_1}
\mbox{\rm (ADIGE-VGH)} \quad \ddot{x}(t) + \partial \phi \Big(\dot{x}(t) + \beta \nabla  f (x(t)\Big)  + \beta \nabla^2  f (x(t)) \dot{x} (t) + \nabla  f (x(t)) \ni 0.
\end{equation}
The parameter  $\beta \geq 0$ is  attached to the geometric  damping induced by the Hessian. As previously considered, $\phi$ is a damping potential function. The suffix V,G,H make respectively reference to the Velocity, the Gradient of $f$, and the Hessian of $f$, which enter the damping terms of the above dynamic.
This model makes it possible to encompass several situations.

\medskip

\noindent $\bullet$  When $\beta=0$, we recover the closed loop controled system
\begin{equation}\label{closed_loop_inertial_both_2}
 \ddot{x}(t) + \partial \phi \Big( \dot{x}(t) \Big)  + \nabla  f (x(t)) = 0,
\end{equation}
studied from  Section \ref{sec: basic_1} to \ref{sec: basic_3}. So studying \eqref{closed_loop_inertial_both_1} can be viewed as an extension of our previous study.
Still, we will see that taking $ \beta> 0 $ induces several favorable properties.

\medskip

\noindent $\bullet$  When $\phi (u) = \frac{\gamma}{2} \|u\|^2$, we obtain the system
\begin{equation}\label{closed_loop_inertial_both_3}
 \ddot{x}(t) + \gamma \dot{x}(t)  + \beta \nabla^2  f (x(t)) \dot{x} (t) + (1+ \gamma \beta) \nabla  f (x(t)) = 0,
\end{equation}
 studied in Section \ref{Sec:Hessian}, and which was introduced by Alvarez--Attouch--Bolte--Redont in \cite{AABR}.

\subsection{Existence and uniqueness results}

A key property for studying \eqref{closed_loop_inertial_both_1} is the following equivalent formulation,  different from the Hamiltonian formulation, and whose proof is immediate. Just introduce
the new variable $u(t) := \dot{x}(t) +  \beta \nabla  f (x(t)) $.

\begin{proposition}\label{first_order_both}
The following are equivalent

\medskip

 \quad $(i) $  \,   $\ddot{x}(t) + \partial \phi \Big(\dot{x}(t) + \beta \nabla  f (x(t)\Big)  + \beta \nabla^2  f (x(t)) \dot{x} (t) + \nabla  f (x(t)) \ni 0.$
\begin{equation*}
\quad (ii) \; \left\{
\begin{array}{l}
\dot{x}(t) +  \beta \nabla  f (x(t)) -u(t) = 0 \hspace{8cm} \\
\rule{0pt}{15pt}
 \dot{u}(t) +\partial \phi (u(t)) +   \nabla  f (x(t) \ni 0.\hspace{7.5cm}
 \end{array}\right.
\end{equation*}

\end{proposition}

\noindent A major interest of the the formulation $(ii)$ is that it is a first-order system in time and space (without occurence of the Hessian). As such, it requires fewer regularity assumptions on $f$ than in  Theorem \ref{th.existence_uniqueness}.

\begin{theorem}\label{th.existence_uniqueness_both}
Let $f:\cH \to \R$ be a convex  function which is twice continuously differentiable, and    such that $\inf_{\cH} f >-\infty$. Suppose that $\nabla f$ is Lipschitz continuous on the bounded subsets of $\cH$. Let  $\phi: \cH \to \R$ be a convex continuous damping function.
 Then, for  any Cauchy data  $(x_0, x_1 ) \in \cH \times \cH$, there exists a unique strong global solution $x : [0, +\infty[ \to \cH$ of
\mbox{\rm (ADIGE-VGH)}
satisfying $x(0) = x_0$, and $\dot{x}(0)=x_1  $.
\end{theorem}
\begin{proof} The structure of the proof being similar to Theorem \ref{th.existence_uniqueness}, we only develop the original aspects.

\smallskip

\textbf{Step 1}: \textit{A priori estimate}.   Note that \eqref{closed_loop_inertial_both_1} can be equivalently written as
\begin{equation}\label{closed_loop_inertial_both_4}
\frac{d}{dt} \Big( \dot{x}(t) + \beta \nabla  f (x(t)    \Big)+ \partial \phi \Big(\dot{x}(t) + \beta \nabla  f (x(t)\Big)   + \nabla  f (x(t)) \ni 0.
\end{equation}
After taking the scalar product of
 \eqref{closed_loop_inertial_both_4} with $\dot{x}(t) + \beta \nabla  f (x(t))$, we get
\begin{eqnarray}
\demi \frac{d}{dt} \| \dot{x}(t) &+& \beta \nabla  f (x(t)) \|^2 + \left\langle \partial \phi (\dot{x}(t) + \beta \nabla  f (x(t))), \dot{x}(t) + \beta \nabla  f (x(t)) \right\rangle \nonumber \\
&+&\left\langle \nabla f (x(t)), \dot{x}(t) + \beta \nabla  f (x(t))\right\rangle =0. \label{closed_loop_both_energy_1}
\end{eqnarray}
Since $\phi$ is a damping potential, the subdifferential inequality for convex functions  gives
$$
\left\langle \partial \phi (\dot{x}(t) + \beta \nabla  f (x(t))), \dot{x}(t) + \beta \nabla  f (x(t)) \right\rangle  \geq \phi (\dot{x}(t) + \beta \nabla  f (x(t))).
$$
Collecting the above results, we obtain
\begin{equation}\label{closed_loop_both_b}
 \frac{d}{dt} \left( \demi \| \dot{x}(t) + \beta \nabla  f (x(t)) \|^2    + f(x(t)) - \inf\nolimits_{\cH} f\right) + \phi \Big(\dot{x}(t) + \beta \nabla  f (x(t))\Big) + \beta \|\nabla  f (x(t)))^2 \leq 0.
\end{equation}
Therefore, the energy-like function
\begin{equation}\label{closed_loop_both_energy-decrease}
t \mapsto \demi \| \dot{x}(t) + \beta \nabla  f (x(t)) \|^2    + f(x(t)) - \inf\nolimits_{\cH} f  \quad \mbox{is nonincreasing}.
\end{equation}
This implies that,  as long as the trajectory is defined
\begin{equation}\label{closed_loop_both_c}
  \| \dot{x}(t) + \beta \nabla  f (x(t)) \|^2 \leq C:=  \| x_1 + \beta \nabla  f (x_0) \|^2 + 2 (f(x_0) - \inf\nolimits_{\cH} f ).
\end{equation}
From this,  we will obtain a bound on the trajectory.
We have
$$
\dot{x}(t) + \beta \nabla  f (x(t)) = k(t)
$$
with $\|k(t)\| \leq \sqrt{C}$. Take the scalar product of the above equation with $x(t)-x_0$.
$$
\demi \frac{d}{dt} \|x(t) -x_0 \|^2  + \beta \left\langle \nabla  f (x(t)) -  \nabla  f (x_0), x(t) -x_0 \right\rangle +
\beta \left\langle  \nabla  f (x_0), x(t) -x_0 \right\rangle = \left\langle k(t)
, x(t) -x_0 \right\rangle. $$
According to the convexity of $f$, and hence the monotonicity of $\nabla  f$, and by Cauchy--Schwarz inequality
\begin{equation}\label{dx2_leq_k}
\demi \frac{d}{dt} \|x(t) -x_0 \|^2  \leq ( \|k(t)\| + \beta \|\nabla  f (x_0) \|)\|x(t) -x_0 \|.
\end{equation}
According to the Gronwall inequality, and $\|k(t)\| \leq \sqrt{C}$, we obtain
\begin{equation}\label{closed_loop_both_d}
   \|x(t) -x_0 \| \leq  t \left( \| x_1 + \beta \nabla  f (x_0) \|+ \sqrt{2 (f(x_0) - \inf\nolimits_{\cH} f )} + \beta \|\nabla  f (x_0)\| \right).
\end{equation}
%Therefore the trajectory $x(\cdot)$  is Lipschitz continuous, with a Lipschitz constant that depends only on the initial data.

\medskip

\textbf{Step 2}: \textit{first-order formulation of \eqref{closed_loop_inertial_both_1}}.
 According to  Proposition \ref{first_order_both}, it is equivalent
to solve  the first-order system
\begin{equation*}
 \left\{
\begin{array}{l}
\dot{x}(t) +  \beta \nabla  f (x(t)) -u(t) = 0 \hspace{8cm} \\
\rule{0pt}{15pt}
 \dot{u}(t) +\partial \phi (u(t)) +   \nabla  f (x(t) \ni 0,\hspace{7.5cm}
 \end{array}\right.
\end{equation*}
 with the Cauchy data
$x(0) =x_0$, \, $u(0)= x_1$.
Set
$Z(t) = (x(t), u(t)) \in \cH \times \cH .$\\
The above system can be written equivalently as
$$
\dot{Z}(t) + F( Z(t))\ni 0, \quad Z(0) = (x_0, x_1),
$$
where  $F: \cH \times \cH\rightrightarrows \cH \times \cH,\;\;(x,u)\mapsto F(x,u)$ is defined by
$$
F(x,u)= \Big( 0,  \partial \phi(u) \Big) +
 \Big(\beta \nabla f(x) -u, \nabla f(x)   \Big).
$$
Hence $F$ splits as follows
$$
F(x,u) = \partial \Phi (x,u) + G (x,u),
$$
where
\begin{equation}\label{Hamilton_Hessian_b}
\Phi (x,u) =   \phi(u)
\, \mbox{ and } \,
G(x,u) =  \Big(\beta \nabla f(x) -u, \nabla f(x)   \Big).
\end{equation}
Therefore, it is equivalent to solve the following first-order differential inclusion with Cauchy data
\begin{equation}
\label{1odd_b}
\dot{Z}(t) +\partial\Phi(Z(t)) + G( Z(t))\ni 0, \quad Z(0) = (x_0, x_1).
\end{equation}
According to the local Lipschitz assumption on the gradient of $f$, we immediately obtain that  the mapping $(x,u)\mapsto G(x,u)$ is  Lipschitz continuous on the bounded subsets of $\cH\times\cH$.

\textbf{Step 3}: \textit{Approximate dynamics}. We  consider the approximate dynamics
\begin{equation}\label{hbdf_lambda_existence_b}
 \ddot{x}_{\lambda}(t) +  \nabla \phi_{\lambda} \Big(\dot{x}_{\lambda}(t)) + \beta \nabla f (x_{\lambda}(t))\Big) + \beta \nabla^2  f (x(t)) \dot{x}_{\lambda} (t)  + \nabla f (x_{\lambda}(t)) = 0,\; t\in [0,+\infty[
\end{equation}
 which uses the Moreau-Yosida approximates $(\phi_{\lambda})$ of $\phi$.
We will prove that the filtered sequence $(x_{\lambda})$
converges uniformly as $\lambda \to 0$ over the bounded time intervals towards a solution of  \eqref{closed_loop_inertial_both_1}.
The first-order formulation of \eqref{hbdf_lambda_existence_b} gives the following system
$$  \quad \left\{
\begin{array}{l}
\dot{x}_{\lambda}(t) +  \beta \nabla  f (x_{\lambda}(t)) -u_{\lambda}(t) = 0 ;  	 \\
\rule{0pt}{18pt}
 \dot{u}_{\lambda}(t) +\nabla \phi_{\lambda}(u_{\lambda}(t)) + \nabla f(x_{\lambda}(t)) = 0 ,
 \hspace{2.3cm}
\end{array}\right.
$$
 with the Cauchy data
$x_{\lambda}(0) =x_0$, \, $u_{\lambda}(0)= x_1       $.
Set
$Z_{\lambda}(t) = (x_{\lambda}(t), u_{\lambda}(t)) \in \cH \times \cH .$\\
The above system can be written equivalently as
$$
\dot{Z}_{\lambda}(t) + F_{\lambda}( Z_{\lambda}(t))\ni 0, \quad Z_{\lambda}(t_0) = (x_0, x_1),
$$
where  $F_{\lambda}: \cH \times \cH\rightarrow \cH \times \cH,\;\;(x,u)\mapsto F_{\lambda}(x,u)$ is defined by
$$
F_{\lambda}(x,u)= \Big( 0,  \nabla \phi_{\lambda}(u) \Big) +
 \Big(\beta \nabla f(x) -u, \nabla f(x)   \Big).
$$
Hence $F_{\lambda}$ splits as follows
$
F_{\lambda}(x,u) = \nabla \Phi_{\lambda} (x,u) + G (x,u)
$
where  $\Phi $ and $G$ have  been defined in \eqref{Hamilton_Hessian_b}.
Therefore, the approximate equation is equivalent to the  first-order differential system with Cauchy data
\begin{equation}
\label{1odd_existence_bb}
\dot{Z}_{\lambda}(t) +\nabla \Phi_{\lambda}(Z_{\lambda}(t)) + G( Z_{\lambda}(t))= 0, \quad Z_{\lambda}(0) = (x_0, x_1).
\end{equation}
Let us argue with $\lambda >0$ fixed.
 According to the Lipschitz continuity of $\nabla \Phi_{\lambda}$,  and the fact that $G$ is Lipschitz continuous on the bounded sets, we have that the sum operator $ \nabla \Phi_{\lambda} + G$ which governs \eqref{1odd_existence_bb} is  Lipschitz continuous on the bounded sets.
As a consequence, the existence of a local solution to \eqref{1odd_existence_bb} follows from the classical
 Cauchy--Lipschitz theorem.
To pass from a local solution to a global solution, we use the a priori estimates \eqref{closed_loop_both_c} and \eqref{closed_loop_both_d}  obtained in Step 1 of the proof.
Note that these estimates are valid for any damping potential, in particular for $\phi_{\lambda}$.
Suppose that a maximal solution is defined on a finite time interval
$[0,T[$.
According to \eqref{closed_loop_both_d} we first obtain that  $x_{\lambda}(t)$ remains bounded on $[0,T[$. Then, using \eqref{closed_loop_both_c}  and the fact that the gradient of $f$ is Lipschitz continuous on the bounded sets, we obtain that $\dot{x}_{\lambda}(t) $ is also bounded on $[0,T[$.
According to the property \eqref{ineq_phi_b} of the Yosida approximation, the property $(iii)$ of the
 damping potential $\phi$, and \eqref{closed_loop_both_c}, this implies that
$$
\| \ \nabla \phi_{\lambda} \Big(\dot{x}_{\lambda}(t)) + \beta \nabla f (x_{\lambda}(t))\Big)\| \leq \| (\partial \phi )^{0} \Big(\dot{x}_{\lambda}(t)) + \beta \nabla f (x_{\lambda}(t))\Big)\|
$$
is  also bounded by on $[0,T[$.
Moreover, according to the local boundedness assumption made  on the gradient, and the boundedness of $x_{\lambda}(t)$ and $\dot{x}_{\lambda}(t) $, we have that
$\nabla^2 f (x_{\lambda}(t))\dot{x}_{\lambda}(t)$ is  also bounded.
According to the constitutive equation \eqref{hbdf_lambda_existence}, this in turn implies that  $(\ddot{x}_{\lambda} )$ is also bounded.
This implies that the limits of $x_{\lambda}(t)$ and $\dot{x}_{\lambda} (t)$
exist, as $t \to T$. According to this property, passing from a local to a global solution
is a classical argument.
So, for any $\lambda >0$ we have a unique global solution of
\eqref{hbdf_lambda_existence} with satisfies the Cauchy data $x_{\lambda}(0) =x_0$, $\dot{x}_{\lambda}(0)= x_1 $.

\smallskip

\textbf{Step 4}: \textit{Passing to the limit as $\lambda \to 0$}.
Take  $T >0$, and $ \lambda , \mu >0$.
Consider the corresponding solutions on $[0, T]$
\begin{eqnarray*}
&&  \dot{Z}_{\lambda}(t) +\nabla \Phi_{\lambda}(Z_{\lambda}(t)) + G( Z_{\lambda}(t))= 0, \quad Z_{\lambda}(0) = (x_0, x_1)
\\
&&\dot{Z}_{\mu}(t) +\nabla \Phi_{\mu}(Z_{\mu}(t)) + G( Z_{\mu}(t))= 0, \quad Z_{\mu}(0) = (x_0, x_1).
\end{eqnarray*}
Let's make the difference between the two equations, and take the scalar product by $Z_{\lambda}(t) - Z_{\mu}(t)$. We get
\begin{eqnarray}
\demi \frac{d}{dt}\| Z_{\lambda}(t) - Z_{\mu}(t) \|^2 &+ &
\left\langle  \nabla \Phi_{\lambda}(Z_{\lambda}(t)) - \nabla \Phi_{\mu}(Z_{\mu}(t)) ,  Z_{\lambda}(t) - Z_{\mu}(t) \right\rangle \nonumber\\
&+& \left\langle  G( Z_{\lambda}(t)) - G( Z_{\mu}(t)) ,  Z_{\lambda}(t) - Z_{\mu}(t) \right\rangle =0 . \label{basic_ex_Y_bb}
\end{eqnarray}
We now use the following ingredients:

\medskip

 i) According to the general properties of the Yosida approximation (see \cite[Theorem 3.1]{Brezis}), we have
$$
\left\langle  \nabla \Phi_{\lambda}(Z_{\lambda}(t)) - \nabla \Phi_{\mu}(Z_{\mu}(t)) ,  Z_{\lambda}(t) - Z_{\mu}(t) \right\rangle
\geq -\frac{\lambda}{4} \|\nabla \Phi_{\mu}(Z_{\mu}(t))  \|^2 -
\frac{\mu}{4} \|\nabla \Phi_{\lambda}(Z_{\lambda}(t))  \|^2.
$$
According to  the  energy estimates, the sequence $(Z_{\lambda})$ is uniformly bounded on $[0, T]$, let
$$\| Z_{\lambda}(t)\|\leq C_T .$$
 From these properties we immediately infer
$$
 \|\nabla \Phi_{\lambda}(Z_{\lambda}(t))  \|  \leq \sup_{\|\xi\|\leq C_T} \|(\partial \phi)^0(\xi)  \|= M_T <+\infty
$$
because our assumption on $\phi$ gives that $(\partial \phi)^0$ is bounded on the bounded sets.
Therefore
$$
\left\langle  \nabla \Phi_{\lambda}(Z_{\lambda}(t)) - \nabla \Phi_{\mu}(Z_{\mu}(t)) ,  Z_{\lambda}(t) - Z_{\mu}(t) \right\rangle
\geq -\frac{1}{4} M_T (\lambda +\mu).
$$

ii)  Since the mapping $G : \cH \times  \cH  \to \cH \times  \cH$ is Lipschitz continuous on the bounded sets, and
using again that the sequence $(Z_{\lambda})$ is uniformly bounded on $[0, T]$, we deduce that there exists a constant $L_T$ such that
$$
\|  G( Z_{\lambda}(t)) - G( Z_{\mu}(t)) \| \leq L_T \|   Z_{\lambda}(t) -  Z_{\mu}(t) \|.
$$
Combining the above results, and using Cauchy--Schwarz inequality, we deduce from
\eqref{basic_ex_Y_b} that
$$
\demi \frac{d}{dt}\| Z_{\lambda}(t) - Z_{\mu}(t) \|^2
\leq  \frac{1}{4} M_T (\lambda +\mu) + L_T \|   Z_{\lambda}(t) -  Z_{\mu}(t) \|^2 .
$$
We now proceed with the integration of this differential inequality.
According to the fact that $ Z_{\lambda}(0) - Z_{\mu}(0) =0$, elementary calculus gives
$$
\| Z_{\lambda}(t) - Z_{\mu}(t) \|^2 \leq \frac{M_T}{4L_T}(\lambda +\mu) \Big( e^{2L_T (t-t_0}  -1  \Big).
$$
Therefore, the filtered sequence $(Z_{\lambda})$  is a Cauchy sequence for the uniform convergence on $[0, T]$, and hence it converges uniformly.
This means the uniform convergence on $[0, T]$ of $x_{\lambda}$ and $\dot{x}_{\lambda}$ to $x$ and $\dot{x}$ respectively.
Proving that $x$ is solution of \eqref{closed_loop_inertial_both_1}  is obtained in a similar way as
in Theorem \ref{basic_exist_thm}. Just rely on the property
$\frac{d}{dt}\left( \nabla f (x_{\lambda}(t)) \right) =  \nabla^2 f (x_{\lambda}(t))\dot{x}_{\lambda}(t)          $ to pass to the limit on the Hessian term.
\end{proof}

\subsection{Convergence properties}

We have the following convergence properties for the solutions trajectories of the system \eqref{closed_loop_inertial_both_1}
with  closed loop damping involving both the velocity and the gradient.

\begin{theorem}\label{th.existence_uniqueness_both_conv}
Let $f:\cH \to \R$ be a convex  function which is twice continuously differentiable, and    such that $\argmin _{\cH} f \neq \emptyset$. We suppose that $\nabla f$ is Lipschitz continuous on the bounded subsets of $\cH$. Suppose that $\beta >0$.
 Let  $\phi: \cH \to \R$ be a convex continuous damping function.
 Then, for any solution trajectory $x : [0, +\infty[ \to \cH$ of
\mbox{\rm (ADIGE-VGH)} we have
\begin{eqnarray*}
&&(i) \mbox{ The energy-like function} \, \,  t \mapsto \demi \| \dot{x}(t) + \beta \nabla  f (x(t)) \|^2    + f(x(t))  \mbox{ is nonincreasing};\\
&&(ii)  \int_0^{+\infty}  \phi \Big(\dot{x}(t) + \beta \nabla  f (x(t))\Big) dt <+\infty ;\\
&&(iii) \int_0^{+\infty}    \|\nabla  f (x(t)) \| ^2 dt <+\infty
\end{eqnarray*}
Suppose moreover that there exists $r>0$ such that for all $u\in \cH$
$\phi (u) \geq r\|u\|$.
Then the following properties are satisfied:

\medskip

a) The trajectory $x(\cdot)$ converges weakly as $t\to +\infty$,  and its limit belongs to $\argmin _{\cH} f$.

\medskip

b)  $\dot{x}(t)$ and $\nabla  f (x(t))$ converge strongly to zero as $t\to +\infty$.
\end{theorem}

\begin{proof}
Items $(i)$ to $(iii)$ are direct consequences of the  estimate \eqref{closed_loop_both_b}  established in the Step 1 of the proof of Theorem \ref{th.existence_uniqueness_both}.\\
Let's now make the additional assumption $\phi (u) \geq r\|u\|$.
According to item $(ii)$, we obtain
$$
\int_0^{+\infty}  \| \dot{x}(t) + \beta \nabla  f (x(t))\| dt
\leq \frac{1}{r} \int_0^{+\infty}  \phi \Big(\dot{x}(t) + \beta \nabla  f (x(t))\Big) dt  <+\infty .
$$
Therefore, $x(\cdot)$ is solution of the non-autonomous steepest descent equation
$$
\dot{x}(t) + \beta \nabla  f (x(t)) = k(t)
$$
with $k \in L^1 (0, +\infty; \cH)$.
We can apply  Theorem 3.11 of \cite{Brezis}, which gives the convergence of the trajectory to a point in $\argmin _{\cH} f$.
In particular, the trajectory remains bounded. According to item $(i)$, we get that  $\dot{x}(t)$ is also bounded.
Returning  to the constitutive equation \eqref{closed_loop_inertial_both_1}, we deduce that the acceleration $\ddot{x}(t)$ is also bounded.
This implies that $\xi(t)= \dot{x}(t) + \beta \nabla  f (x(t))$ satisfies
$$\int_0^{+\infty}    \|\xi (t)\| dt <+\infty \, \mbox{ and } \, \|\dot{\xi}(t)\|\leq M
$$
for some $M >0$.
 This classically implies that $\xi(t)= \dot{x}(t) + \beta \nabla  f (x(t))$ tends to zero as $t \to +\infty$.
According to  item $(iii)$, the same argument applied to $\nabla  f (x(t))$ gives that   $\nabla  f (x(t))$ tends to zero as $t \to +\infty$. As a difference of the two previous quantities, we conclude that $\dot{x}(t)$ tends to zero as $t \to +\infty$.
\end{proof}

Indeed, Theorem 3.11 of \cite{Brezis} was proved under the additional sassumption that $f$ is inf-compact. Recent progress based on Opial lemma  \cite{Op} and Bruck theorem \cite{Bru} allows to extend it to general convex function $f$, without making this additional assumption.
This is made precise below.

\begin{proposition} Let $f: \cH \to \Rb $ be a convex lower semicontinuous
proper function such that $\argmin_{\cH} f \neq \emptyset$, and let $k \in L^1 (0, +\infty; \cH)$.
Suppose that $x: [0, +\infty[ \to \cH$ is a strong global solution trajectory of
$$
\dot{x}(t) +  \partial  f (x(t)) \ni k(t).
$$
Then, the trajectory $x(\cdot)$ converges weakly as $t\to +\infty$,  and its limit belongs to $\argmin _{\cH} f$.
\end{proposition}
\begin{proof}
Take $\epsilon >0$. Since $k \in L^1 (0, +\infty; \cH)$, there exists $T_{\epsilon} >0$ such that $\int_{T_{\epsilon}}^{+\infty} \|k(t) \| dt < \epsilon$.
Let's consider the solution $v: [0, +\infty[ \to \cH$ of
$$
\dot{v}(t) + \nabla  f (v(t)) \ni 0; \quad v(0) = x(T_{\epsilon}).
$$
According to the  semigroup of contractions property, we have, for all
$t\geq T_{\epsilon}$
\begin{equation}\label{sg_property}
\| x(t) - v(t-T_{\epsilon})\| \leq \| x(T_{\epsilon}) - v(0)\| +
\int_{T_{\epsilon}}^{t} \|k(t) \| dt \leq \epsilon.
\end{equation}
Take $\xi \in \cH$. By Cauchy--Schwarz inequality, we have
$$
| \left\langle  x(t) - v(t-T_{\epsilon}), \xi \right\rangle | \leq \|  \leq \epsilon \|\xi\|.
$$
By the triangle inequality, we deduce that, for all $t\geq T_{\epsilon}$, $t'\geq T_{\epsilon}$
$$
| \left\langle  x(t), \xi \right\rangle -  \left\langle  x(t'), \xi \right\rangle|  \leq   | \left\langle  v(t-T_{\epsilon}) -v(t'-T_{\epsilon}, \xi \right\rangle|  +  2\epsilon \|\xi\|.
$$
According to the Bruck theorem, we know that the weak limit of $v(t)$
exists. Passing to the limsup on the above inequality we get
$$
\limsup_{t,t' \to +\infty} | \left\langle  x(t), \xi \right\rangle -  \left\langle  x(t'), \xi \right\rangle|  \leq  \limsup_{t,t' \to +\infty} | \left\langle  v(t-T_{\epsilon} -v(t'-T_{\epsilon}, \xi \right\rangle|  +  2\epsilon \|\xi\| \leq  2\epsilon \|\xi\|.
$$
This being true for any $\epsilon >0$, we deduce that the limit of
$\left\langle  x(t), \xi \right\rangle$ exists, which implies that the
weak limit of $x(t)$ exists as $t\to +\infty$, let $x_{\infty}$ its limit.
Passing to the lower limit on \eqref{sg_property}, according to the lower semicontinuity of the norm for the weak topology, we deduce that
\begin{equation}\label{sg_property_b}
\| x_{\infty} - \lim_{t \to +\infty} v(t)\| \leq  \epsilon.
\end{equation}
Since the weak limit of $v(t)$ belongs to $\argmin _{\cH} f$, we deduce that $\dist (x_{\infty}, \argmin _{\cH} f) \leq \epsilon.$
This being true for any $\epsilon >0$, and since $\argmin _{\cH} f$ is closed, we finally get that $x_{\infty} \in \argmin _{\cH} f$.
\end{proof}

\subsection{An approach based on Opial's lemma} \label{rem-quasi-hessian_both}
 Here we will prove the weak convergence of the trajectory
 $x$ to a minimizer of $f$, based on the continuous version of the Opial Lemma \cite{Op}.
 As in the proof of Theorem \ref{th.existence_uniqueness_both_conv}, items $(i)$ to $(iii)$ hold. Assume $\phi (u) \geq r\|u\|$ for all $u\in \cH$.
 According to item $(ii)$ we obtain
$$
\int_0^{+\infty}  \| \dot{x}(t) + \beta \nabla  f (x(t))\| dt <+\infty .
$$
Equivalently, we have
$$
\dot{x}(t) + \beta \nabla  f (x(t)) = k(t)
$$
with $k \in L^1 (0, +\infty; \cH)$.
Let us prove that $x$ is bounded. Relying on step 1 of the proof of Theorem \ref{th.existence_uniqueness_both}, notice that \eqref{dx2_leq_k} holds for
a generic $x_0\in \cH$. Taking an arbitrary $z\in\argmin f$, we derive from \eqref{dx2_leq_k}
\begin{equation}\label{dx2_leq_k_for_opial}
\demi \frac{d}{dt} \|x(t) -z \|^2 \leq  \|k(t)\|\cdot  \|x(t) -z \|.
\end{equation}
Integrating we obtain
\begin{equation}\label{dx2_leq_k_for_opial_integr}
 \frac{1}{2}\|x(T) -z \|^2 \leq \frac{1}{2}\|x_0 -z \|^2 + \int_0^T\|k(t)\|\cdot  \|x(t) -z \|dt \ \  \forall T\geq 0.
\end{equation}
Now apply \cite[Lemme A.5, pag 157]{Brezis} to conclude
$$\|x(T)-z\|\leq \|x_0 -z \| + \int_0^T\|k(t)\|dt  \ \ \forall T\geq 0. $$
Since $k \in L^1 (0, +\infty; \cH)$ we obtain that $x$ is bounded.
Now we can repeat the arguments in the proof of Theorem \ref{th.existence_uniqueness_both_conv} to conclude that
$\lim_{t\to\infty}\dot x(t)=\lim_{t\to\infty}\nabla f( x(t))=0$, so we omit the proof.
Let us pass forward and see how the Opial Lemma \cite{Op} can be applied.\\
Since $x$ is bounded and $k \in L^1 (0, +\infty; \cH)$, we get from \eqref{dx2_leq_k_for_opial} that
$\lim_{t\to\infty}\|x(t)-z\|$ exists, hence the first condition in the Opial Lemma is fulfilled.
To check the second condition in the Opial Lemma is standard. Take $\overline{x}\in \cH $ and $t_n\to +\infty$ such that $x(t_n)$ converges weakly
to $\overline{x}$, as $n\to + \infty$. The convexity of $f$ yields for all
$x\in \cH$ and all $n\in\N$
$$f(x)\geq f(x(t_n))+\langle \nabla f(x(t_n)),x-x(t_n)\rangle.$$
Fixing $x$ and taking the limit as $n\to+\infty$, and relying on the strong convergence of $\nabla f(x(t))$ to $0$ and the boundedness of $x$, we derive
$$f(x)\geq \liminf_{n\rightarrow+\infty}f(x(t_n))\geq f(\overline {x}),$$
where the last inequality follows from the weak lower semicontinuity of the convex function $f$. Since the last inequality holds for an arbitrary $x$,
we obtain $\overline{ x}\in\argmin f$. Therefore, the second condition in the Opial Lemma is fulfilled as well.

\subsection{A finite stabilization property}
As we already noticed, (ADIGE-VGH) can be equivalently written as
$$
\dot{u}(t) + \partial \phi (u(t)) \ni -   \nabla  f (x(t) )
$$
where $ u(t) =\dot{x}(t) +  \beta \nabla  f (x(t))  $.
After taking the scalar product of the above equation with $u(t)$, we get
$$
\demi \frac{d}{dt} \| u(t) \|^2 + \left\langle \partial \phi (u(t)), u(t) \right\rangle = -  \left\langle \nabla  f (x(t)), u(t) \right\rangle .
$$
When $\phi (u) \geq r\|u\|$, and by Cauchy--Schwarz inequality we get
$$
\demi \frac{d}{dt} \| u(t) \|^2 + r \|u(t)\|  \leq \| \nabla  f (x(t))\|  \|u(t) \| .
$$
Since $\nabla  f (x(t))$ converges strongly to zero as $t\to +\infty$
(that's the last point of Theorem \ref{th.existence_uniqueness_both_conv}), we get for $t$ large enough
$\| \nabla  f (x(t))\|  \leq \demi r$, and hence
$$
\demi \frac{d}{dt} \| u(t) \|^2 + \demi r \|u(t)\|  \leq 0 .
$$
This gives that $u(t) \equiv 0$ after a finite time.
Let us summarize the above results in the following Proposition.

\begin{proposition} Under the hypothesis of Theorem \ref{th.existence_uniqueness_both_conv}, and when $\phi (u) \geq r\|u\|$ for some $r>0$, we have that after a finite time
$$
\dot{x}(t) +  \beta \nabla  f (x(t))  \equiv 0,
$$
\ie the trajectory follows the steepest descent dynamic.
\end{proposition}

\subsection{The case \textit{f}  strongly convex: exponential convergence rate}

\begin{theorem}\label{strong_convex_thm-vel-grad} Let $f: \cH \to \R$ be a $\gamma$-strongly convex function (for some $\gamma>0$) which is twice continuously differentiable, and whose gradient is  Lipschitz continuous on the bounded sets. Let $\overline{x}$ be the unique minimizer of $f$.
Let $\phi : \cH \to \R_+$ be a damping potential (see Definition \ref{def1}) which is differentiable, and whose gradient is  Lipschitz continuous on the bounded subsets of $\cH$. Suppose  that $\phi$ satisfies the following growth conditions:

\smallskip

$(i)$ (local) there exists positive constants $\alpha$,   and  $\epsilon >0$ such that, for all $u$ in
$\cH$ with $\|u\| \leq \epsilon$
$$     \left\langle \nabla \phi (u), u \right\rangle \geq \alpha \|u\|^2 .$$

$(ii)$ (global) there exists  $p\geq 1$, $r>0$, such that for all $u$ in $\cH$, $\phi (u) \geq r\|u\|^p$.

\medskip

\noindent Suppose $\beta > 0$. Let  $x: [0, +\infty[ \to \cH$ be a solution trajectory of  \mbox{\rm (ADIGE-VGH)}
\begin{equation}\label{closed_loop_inertial_both_1-quasi-case-str-conv}
\ddot{x}(t) + \nabla \phi \Big(\dot{x}(t) + \beta \nabla  f (x(t)\Big)  + \beta \nabla^2  f (x(t)) \dot{x} (t) + \nabla  f (x(t)) = 0.
\end{equation}
Then, we have  exponential convergence rate to zero  as $t \to +\infty$  for $f(x(t))-f(\overline{x}) $, $\| x(t)-\overline{x}\|$ and $\|\dot x (t)+\beta\nabla f(x(t))\|$.
As a consequence, we also have exponential convergence rate to zero  as $t \to +\infty$  for  $\|\dot x(t)\|$ and $\|\nabla f(x(t))\|$.
\end{theorem}
\begin{proof} Since $f$ is strongly convex, $f$ is a coercive function.  According to the decrease property of the global energy, see \eqref{closed_loop_both_energy-decrease} and Theorem \ref{th.existence_uniqueness_both_conv} $(i)$, we have that $f(x(t))$ is bounded from above, and hence the trajectory $x$ is bounded.
Item $(ii)$ of Theorem \ref{th.existence_uniqueness_both_conv}, and the global  growth assumption on $\phi$ give that, for some $p\geq 1$
$$
\int_0^{+\infty}  \| \dot{x}(t) + \beta \nabla  f (x(t))\|^p dt <+\infty.
$$
By a similar argument as in the proof of Theorem \ref{th.existence_uniqueness_both_conv} (where we argued with $p=1$)
we deduce that $\lim_{t\to +\infty} \| \dot{x}(t) + \beta \nabla  f (x(t))\| =0$.
Therefore, for $t$ sufficiently large
$$\|\dot x (t)+\beta\nabla f(x(t))\|\leq \epsilon.$$
From \eqref{closed_loop_both_energy_1} and the local property (i) we derive
\begin{equation}\label{energy-multiplied-local-pr}
\frac{d}{dt} \left(\demi\| \dot{x}(t) + \beta \nabla  f (x(t))\|^2 + f(x(t)) - f(\overline{x})\right) + \alpha\|\dot x(t)+\beta\nabla f(x(t))\|^2 + \beta\| \nabla  f (x(t)) \|^2\leq 0.
\end{equation}
Since  $\dot{x}(\cdot)+\beta\nabla f(x(\cdot))$ is  bounded,
let $L>0$ be the Lipschitz constant of $\nabla\phi$ on a ball that contains the vector $\dot x(t)+\beta\nabla f(x(t))$ for all $t\geq 0$. Since $\nabla \phi (0) =0$ we have, for all $t\geq 0$
\begin{equation}\label{local_Lip-case-vel-grad}
\| \nabla \phi(\dot x(t))+\beta\nabla f(x(t))\| \leq L \| \dot x(t)+\beta\nabla f(x(t))\|.
\end{equation}
Using successively \eqref{closed_loop_inertial_both_1-quasi-case-str-conv}, \eqref{local_Lip-case-vel-grad}
and \eqref{from-str-conv1}, we obtain
\begin{eqnarray}\frac{d}{dt}\langle x(t)-\overline{x},\dot x(t)+\beta\nabla f(x(t))\rangle  &=&
\|\dot x(t)\|^2  + \beta \frac{d}{dt}(f(x(t)) - f(\overline{x})) \nonumber \\
&& + \langle x(t)-\overline{x},-\nabla \phi(\dot x(t)+\beta\nabla f(x(t)))-\nabla f(x(t))\rangle \nonumber\\
&\leq & \|\dot x(t)\|^2  + \beta \frac{d}{dt}(f(x(t)) - f(\overline{x})) + \frac{L^2}{2\gamma}\|\dot x(t)+\beta\nabla f(x(t))\|^2\nonumber\\ && +\frac{\gamma}{2} \|x(t)-\overline{x}\|^2 + \langle \overline{x} -x(t),\nabla f(x(t))\rangle \nonumber\\
&\leq & \|\dot x(t)\|^2  + \beta \frac{d}{dt}(f(x(t)) - f(\overline{x})) + \frac{L^2}{2\gamma}\|\dot x(t)+\beta\nabla f(x(t))\|^2\nonumber\\ &&  + f(\overline{x}) - f(x(t)). \label{strong_conv_2-case-vel-grad}
\end{eqnarray}
Take now $\varepsilon >0$ (we will specify below how it should be chosen), and define
$$h_{\varepsilon,\beta}(t) := \demi\| \dot{x}(t) + \beta \nabla  f (x(t))\|^2 + (1-\beta\varepsilon) \big(f(x(t)) - f(\overline{x})\big)  + \varepsilon \langle x(t)-\overline{x},\dot x(t)+\beta\nabla f(x(t))\rangle.$$
Multiplying \eqref{strong_conv_2-case-vel-grad} with $\varepsilon$ and adding the result to  \eqref{energy-multiplied-local-pr}, we derive
\begin{eqnarray*}
\dot h_{\varepsilon,\beta}(t) \leq &&- \alpha\|\dot x(t)+\beta\nabla f(x(t))\|^2 - \beta\| \nabla  f (x(t)) \|^2 +\varepsilon\|\dot x(t)\|^2  + \frac{\varepsilon L^2}{2\gamma}\|\dot x(t)+\beta\nabla f(x(t))\|^2 \\
&& - \varepsilon (f(x(t))-f(\overline{x})).
\end{eqnarray*}
We use the inequality
\begin{equation}\label{usef-ineq}\varepsilon\|\dot x(t)\|^2 \leq 2\varepsilon \|\dot x(t)+\beta\nabla f(x(t))\|^2
+2\varepsilon\beta^2 \|\nabla f(x(t))\|^2\end{equation}
and we obtain
\begin{equation}\label{d-h-e-b} \dot h_{\varepsilon,\beta}(t) \leq - \left(\alpha - 2\varepsilon - \frac{\varepsilon L^2}{2\gamma}\right)\|\dot x(t)+\beta\nabla f(x(t))\|^2 - (\beta - 2\varepsilon\beta^2)\| \nabla  f (x(t)) \|^2  - \varepsilon (f(x(t))-f(\overline{x})).
\end{equation}
Choose $\varepsilon > 0$ small enough such that $C_1: =\min\left\{\left(\alpha - 2\varepsilon - \frac{\varepsilon L^2}{2\gamma}\right), \beta - 2\varepsilon\beta^2, \varepsilon\right\} > 0$. We obtain
\begin{equation}\label{d-h-e-b-ok} \dot h_{\varepsilon,\beta}(t) \leq - C_1\Big(\|\dot x(t)+\beta\nabla f(x(t))\|^2 + \| \nabla  f (x(t)) \|^2  + f(x(t))-f(\overline{x})\Big).
\end{equation}
Further, we have
\begin{eqnarray} h_{\varepsilon,\beta}(t) &=& \demi\| \dot{x}(t) + \beta \nabla  f (x(t))\|^2 +
\varepsilon\beta\big(\langle x(t)-\overline{x},\nabla f(x(t))\rangle + f(\overline{x}) - f(x(t))\big)\label{h-e-b-leq}\\
&& + f(x(t)) - f(\overline{x})  + \varepsilon \langle x(t)-\overline{x},\dot x(t))\rangle. \nonumber
\end{eqnarray}
Since $f$ is strongly convex, we have (see for example Theorem 2.1.10 in \cite{Nest2})
\begin{equation}\label{nes-str-cv}\langle x(t)-\overline{x},\nabla f(x(t))\rangle + f(\overline{x}) - f(x(t)) \leq \frac{1}{2\gamma}\|\nabla f(x(t))\|^2.
\end{equation}
Moreover, from \eqref{from-str-conv2} and \eqref{usef-ineq} we get
\begin{eqnarray*} && f(x(t)) - f(\overline{x})  + \varepsilon \langle x(t)-\overline{x},\dot x(t))\rangle \leq
 f(x(t)) - f(\overline{x}) +
\frac{\varepsilon}{2}\|x(t)-\overline{x}\|^2+\frac{\varepsilon}{2}\|\dot x(t)\|^2\\
&&\leq  \left(1+\frac{\varepsilon}{\gamma}\right) \big(f(x(t)) - f(\overline{x})\big)
+ \varepsilon \|\dot x(t)+\beta\nabla f(x(t))\|^2
+\varepsilon\beta^2 \|\nabla f(x(t))\|^2.
\end{eqnarray*}
From this, \eqref{nes-str-cv} and \eqref{h-e-b-leq} we get
\begin{eqnarray*} h_{\varepsilon,\beta}(t) &\leq& \left(\frac{1}{2}+\varepsilon\right)\|\dot x(t)+\beta\nabla f(x(t))\|^2 +
\left(\frac{\varepsilon\beta}{2\gamma} + \varepsilon\beta^2\right)\|\nabla f(x(t))\|^2 +
\left(1+ \frac{\varepsilon}{\gamma}\right)(f(x(t))-f(\overline{x}))\\
&\leq & C_2\Big(\|\dot x(t)+\beta\nabla f(x(t))\|^2 + \| \nabla  f (x(t)) \|^2  + f(x(t))-f(\overline{x})\Big),
\end{eqnarray*}
where $C_2:= \max\left\{ \frac{1}{2}+\varepsilon,  \frac{\varepsilon\beta}{2\gamma} + \varepsilon\beta^2, 1+ \frac{\varepsilon}{\gamma}\right\} > 0$.
Combining this inequality with \eqref{d-h-e-b-ok}, we obtain
$$\dot {h}_{\varepsilon,\beta}(t) + C_3 h_{\varepsilon,\beta}(t)\leq 0,$$
with $C_3:= \frac{C_1}{C_2}  >0$. Then, the Gronwall inequality classically implies
\begin{equation}\label{rate_h-str-conv-vel-grad}h_{\varepsilon,\beta}(t) \leq h_{\varepsilon,\beta}(0)e^{-C_3t}.\end{equation}
Finally, from \eqref{from-str-conv2} and the Cauchy--Schwarz inequality we have
\begin{eqnarray*}h_{\varepsilon,\beta}(t) &\geq & \frac{1}{2}\|\dot x(t)+\beta\nabla f(x(t))\|^2 + (1-\beta\varepsilon)\big(f(x(t)) - f(\overline{x})\big)\\
&&-\frac{\varepsilon}{2}\|x(t)-\overline{x}\|^2 - \frac{\varepsilon}{2}\|\dot x(t)+\beta\nabla f(x(t))\|^2\\
&\geq & \frac{1-\varepsilon}{2}\|\dot x(t)+\beta\nabla f(x(t))\|^2 + \left(1-\beta\varepsilon-\frac{\varepsilon}{\gamma}\right)(f(x(t))-f(\overline{x})).\end{eqnarray*}
Therefore, by taking $\varepsilon$ small enough, we obtain
\begin{equation}\label{rate-1part-vel-grad}h_{\varepsilon,\beta}(t)\geq C_4 \Big(\|\dot x(t)+\beta\nabla f(x(t))\|^2 + f(x(t))-f(\overline{x})\Big),\end{equation}
with $C_4:= \min\left\{\frac{1-\varepsilon}{2},1-\beta\varepsilon-\frac{\varepsilon}{\gamma}\right\} > 0$.
Combining this inequality  with  \eqref{rate_h-str-conv-vel-grad} and  \eqref{from-str-conv2}, we obtain an exponential
convergence rate to zero for $f(x(t))-f(\overline{x}) $, $\|x(t)-\overline{x}\|$ and
$\|\dot x (t)+\beta\nabla f(x(t))\|$.\\
Since $\nabla f$ is Lipschitz continuous on the bounded sets, and $x(t)$ converges to $\overline{x}$, there exists  $L_f >0$ such that for all $t\geq 0$
$$
\|\nabla f (x(t)) \|=\|\nabla f (x(t)) - \nabla f (\overline{x})\| \leq L_f \| x(t) - \overline{x}\|.
$$
Based on the exponential convergence rate of  $\|x(t)-\overline{x}\|$ to zero, we deduce that the same property holds for
$\|\nabla f (x(t)) \|$. By combining this last property with  the exponential convergence rate of $\|\dot x (t)+\beta\nabla f(x(t))\|$ to zero, we finally get that $\|\dot x (t)\|$ converges exponentially to zero when $t \to +\infty$.
\end{proof}

\begin{remark} Similar rates have been reported in \cite[Theorem 4.2]{ACFR} for the heavy ball method
with Hessian driven damping.
\end{remark}

\begin{remark} It is possible to derive similar exponential rates also for the system \eqref{Hessian_def_1},
however for a restrictive choice of $\beta > 0$. To see this, notice that for $\theta > 0$ we have
\begin{eqnarray*}
&&\frac{d}{dt} \left(\demi\| \dot{x}(t) + \beta \nabla  f (x(t))\|^2 + f(x(t)) - f(\overline{x})\right)\\
&& =  -\langle \dot x(t), \nabla\phi(\dot x(t))\rangle  - \beta\langle \nabla \phi(\dot x(t)), \nabla f(x(t))\rangle - \beta\|\nabla f(x(t))\|^2\\
&&\leq  -\alpha \|\dot x(t)\|^2 + \frac{\beta\theta}{2}\|\nabla f(x(t))\|^2 + \frac{\beta L^2}{2\theta}\|\dot x(t)\|^2
- \beta\|\nabla f(x(t))\|^2\\
&&= -\left(\alpha - \frac{\beta L^2}{2\theta}\right)\|\dot x(t)\|^2 - \beta\left(1 - \frac{\theta}{2}\right)\|\nabla f(x(t))\|^2.
\end{eqnarray*}

\end{remark}

\subsection{Further convergence results based on the quasi-gradient approach}\label{rem-quasi-hessian_both-quasi}

Let us consider the dynamical systems (ADIGE-VGH) in case $\phi$ is differentiable, $f: \R^N \to \R$ is a $\mathcal C^2$ function (possible nonconvex) whose gradient is  Lipschitz continuous on the bounded sets, and such that $\inf_{\R^N} f >-\infty$:
 \begin{equation}\label{closed_loop_inertial_both_1-quasi}
\ddot{x}(t) + \nabla \phi \Big(\dot{x}(t) + \beta \nabla  f (x(t))\Big)  + \beta \nabla^2  f (x(t)) \dot{x} (t) + \nabla  f (x(t)) = 0.
\end{equation}
The considerations are similar to those of Section 7.3 and Theorem \ref{quasi_grad_thm_2}.\\
According to Step 2 in the proof of Theorem \ref{th.existence_uniqueness_both}, the first-order reformulation is
\begin{equation}\label{first_order_cl_loop_quasi_1_both-hessian}
\dot z(t) + F(z(t)) =0,
\end{equation}
where $z(t)=(x(t), u(t) ) \in \R^N \times \R^N $, and
   $F: \R^N \times  \R^N \to \R^N \times \R^N$
is defined by
$$F(x,u)=(\beta\nabla f(x)-u, \nabla \phi(u)+ \nabla f(x)).$$
Let us check the angle condition ($E_{\lambda}$ is defined as in Theorem \ref{quasi_grad_thm_2}). We have
\begin{eqnarray*}
\left\langle  \nabla  E_{\lambda}(x,u), F(x,u) \right\rangle
&=& \left\langle  \Big( \nabla f (x)+ \lambda \nabla^2 f (x)u, \, u + \lambda \nabla f (x) \Big), \Big(\beta\nabla f(x)-u, \nabla \phi(u)+ \nabla f(x) \Big) \right\rangle.
\end{eqnarray*}
After development and simplification, we get
\begin{eqnarray*}
\left\langle  \nabla  E_{\lambda}(x,u), F(x,u) \right\rangle
&=&  - \lambda \left\langle  \nabla^2 f (x)u, \, u  \right\rangle + \left\langle  u, \, \nabla \phi(u) \right\rangle
+ \lambda \left\langle  \nabla f (x) , \, \nabla \phi(u)  \right\rangle
+ \lambda  \|  \nabla f (x) \|^2\\
&& + \beta\|\nabla f(x)\|^2+\lambda\beta\left\langle \nabla f(x), \nabla^2 f (x)u\right\rangle.
\end{eqnarray*}
We estimate the term $\lambda\beta\left\langle \nabla f(x), \nabla^2 f (x)u\right\rangle$ by writing
$$\lambda\beta\left\langle \nabla f(x), \nabla^2 f (x)u\right\rangle\geq -\frac{\lambda}{4}\|\nabla f(x)\|^2
-\lambda\beta^2 M^2\|u\|^2$$
and get (as in the proof of Theorem \ref{quasi_grad_thm_2})
\begin{eqnarray}
\left\langle  \nabla  E_{\lambda}(x,u), F(x,u) \right\rangle
&\geq &   \Big( \gamma -\lambda M   - \frac{\lambda}{2} \delta^2 -\lambda \beta^2M^2 \Big)  \|u\|^2
  + \left(\frac{\lambda}{4}+\beta\right)  \|  \nabla f (x) \|^2  . \label{quasi_gradient_2_both_hessian}
\end{eqnarray}
We also have 
\begin{equation*}
\| F(x,u)\| \leq C_2  ( \|u\|^2 +   \|  \nabla f (x) \|^2)^{\demi},
\end{equation*}
  where $C_2=\sqrt{2(2+\beta^2+\delta^2)}$.
 The rest can be done in the lines of the proof of Theorem \ref{quasi_grad_thm_2}.
%\end{itemize}

\section{Conclusion, perspectives}

In this article, from the point of view of optimization, we put forward some classical and new properties concerning the asymptotic convergence of autonomous damped inertial dynamics.
From a control point of view, the damping terms of these dynamics can be considered as closed-loop controls of the current data: position, speed, gradient of the objective function, Hessian of the objective function, and combination of these objects.
Let us cite some of the main results and advantages of the autonomous approach compared to the non-autonomous approach, where damping involves parameters given from the start as functions of time.

\subsection{PRO}
\begin{itemize}
\item Autonomous systems are easy to implement. It is not necessary to adjust the damping coefficient as is the case for non-autonomous systems.

\smallskip

\item When the function to be minimized is strongly convex, there is convergence at an exponential rate, and this is valid for a large class of damping potentials.

\smallskip

\item We were able to exploit the quasi-gradient structure of the autonomous damped dynamics and combine them with the Kurdyka-Lojasiewicz theory to obtain convergence rates for a large class of functions $f$, possibly non-convex.
This is specific to the autonomous case because the theories mentioned above are not developed in the non-autonomous case.

\smallskip

\item The Hessian damping naturally comes within the framework of autonomous systems.
It notably improves the theoretical and numerical behavior of the trajectories, by reducing the oscillatory aspects.
Its introduction into the algorithms does not change their numerical complexity (it makes appear the difference of the gradient  at two consecutive steps). This makes this geometric damping  successful, several recent articles have been devoted to it.

\smallskip

\item The closed-loop approach clearly distinguishes between the strong and weak damping effects, and the transition between them. It also shows the replacement of the theory of convergence by the notion of attractor when the damping becomes too weak.

\smallskip

\item We have introduced a new autonomous system where the damping involves both the speed and the gradient of $ f $, and which benefits from very good convergence properties. At the beginning of time it takes advantage of the inertial effect, then after a finite time it turns into a steepest descent dynamic, thus avoiding the oscillatory aspects. This regime  change  has some similarities with the restart method, and also the recent work of  Poon-Liang \cite{PL} on adaptive acceleration.

\smallskip

\item The closed-loop approach makes it possible to make the link with different fields, such as PDE and control theory, where the stabilization of oscillating systems is a central issue.
Although the simple mathematical framework chosen in this article (single functional space $\cH$, differentiable objective function $f$) does not make it possible to deal directly with the associated PDEs, the Lyapunov analysis that we have developed can naturally be extended to this framework.

\smallskip

\item We have developed an inertial algorithm which shares the good convergence properties of the related continuous dynamics, in the case of the quasi-gradient and Kurdyka-Lojasiewicz approach.
Note that the quasi-gradient approach reflects relative errors in the algorithms, and therefore gives a lot of flexibility.
It is this approach that has made it possible to deal with many different algorithms in Attouch-Bolte-Svaiter \cite{ABS} in the nonconvex nonsmooth case. It would be interesting to develop these aspects for splitting algorithms, such as  proximal gradient algorithms, regularized Gauss--Seidel algorithms, and PALM (see also \cite{BotCseLaJDE} for a continuous-times approach to structured optimization problems).
\end{itemize}

\subsection{CONS}

\smallskip

\begin{enumerate}
\item To date, we do not know in the autonomous case the equivalent of the accelerated gradient method of Nesterov and Su-Boyd-Cand\`es damped inertial dynamic, that is to say an adjustment of the damping potential which guarantees the rate of convergence of values $ 1 / t^2 $ for any convex function.
This is a current research subject, for recent progress in this direction, see Lin-Jordan \cite{LJ}.

\smallskip

\item The general approach based on the quasi-gradient and the Kurdyka-Lojasiewicz theory (as developed in Section \ref{sec: basic_3})  works mainly in finite dimension.
The extension of the (KL) theory to spaces of infinite dimension is a current research subject.

\end{enumerate}

\medskip

\subsection{Perspectives}
$\mbox{ }$

\smallskip

\begin{enumerate}

\item Develop closed-loop versions of the Nesterov accelerated gradient method from a theoretical and numerical point of view.
Our analysis allowed us to better define the type of damping potential $\phi$ capable of doing this, but this remains an open question for study. Indeed, the case $ p = 2 $ (\ie quadratic behavior of the damping potential near the origin) is the critical case separating the weak damping from the strong damping. Taking $p>2$,  with $ p $  close to $ 2$ provides a vanishing viscosity damping coefficient, which is a specific property of the Nesterov accelerated gradient method. Our intuition is that we need to refine the power scale which is not precise enough to provide the correct setting of the vanishing damping term (\ie going from $ p = 2 $ to $ p> 2 $, with $ p $ even very close to $ 2$ is too sudden a change).

\smallskip

\item Extend our study to the case of nonsmooth optimization possibly involving a constraint. This is an important subject, which is closely related to item 6 of this list, because a common device to deal with a constrained optimization problem is to use a gradient-projection method, which falls under fixed point methods.

\smallskip

\item Develop a control perspective with closed-loop damping for the restarting methods.
The restarting methods take advantage of the inertial effect to accelerate the trajectories, then stop when a given criteria is deteriorate. Then restart from the current point with zero velocity, and so on.
In many ways, the dynamic we developed in Section \ref{Sec: combine}  follows a similar strategy. Our results are valid with  general data functions  $ f $ and $ \phi $, while the known results concerning the restart methods only concern the case where $ f $ is  strongly convex. It is an important subject of study, largely to explore.

\smallskip

\item Obtain a closed-loop version of the Tikhonov regularization method, and make the link with the Haugazeau method.
The objective is then, within the framework of the convex optimization, to obtain an autonomous dynamic whose trajectories strongly converge towards the solution of minimum norm; see
  Attouch--Cabot--Chbani--Riahi \cite{AC2R-JOTA}, and
 Bo\c t--Csetnek--L\'{a}szl\'{o} \cite{BCL} for some recent results in the open-loop case (the Tikhonov regularization parameter tends to zero in a controlled manner, not too fast) and references therein.

\smallskip

\item Develop the corresponding algorithmic results.
Continuous dynamics provide a valuable guide to introduce and analyze algorithms that benefit from similar convergence properties.
In Theorem \ref{quasi_grad_thm_algo} we have analyzed the convergence property of an inertial algorithm with general damping potential $\phi$ and general (tame) function $f$. A similar analysis can certainly be developed on the basis of the Theorems \ref{quasi_thm_VH} and \ref{th.existence_uniqueness_both_conv} which also involve the Hessian-driven damping.
Natural extensions then consist in studying structured optimization problems and the corresponding proximal-based algorithms.

\smallskip

\item  In recent years, most of the previous themes have been extended (in the open-loop case) to the case of  maximally monotone operators, see \'Alvarez--Attouch \cite{AA1}, Attouch--Maing\'e \cite{AM}, Attouch--Peypouquet \cite{AP-max}, Attouch--Cabot \cite{AC1}, Attouch--L\'{a}szl\'{o} \cite{AL},  Bo\c t--Csetnek \cite{BotCse}. It would be interesting to consider the closed-loop version of these dynamics, as was done by Attouch-Redont-Svaiter \cite{ARS} for first-order Newton-like evolution systems.

\smallskip

\item Time rescaling is a powerful tool to accelerate the inertial systems; see   Attouch--Chbani--Riahi \cite{ACR-Pafa},
Shi--Du--Jordan--Su \cite{SDJS} and references therein. It leads naturally to non-autonomous dynamics. It would be interesting to study autonomous closed-loop version. This means first extracting quantities which tend monotonically to $+\infty$.

\smallskip

\item Study of the stability of the dynamics and algorithms with respect to perturbations/errors. This is an an important topic from a numerical point of view, see \cite{AC2R-JOTA}, \cite{ACPR}, \cite{ACR-Pafa}, \cite{VSBV}.

\smallskip

\item The concepts of control theory and dissipative dynamical systems have proven to be useful and intuitive design guidelines for speeding up stochastic gradient methods, especially  for the variance-reduction methods for the finite-sum
problem, see \cite{HWL}  and  accompanying bibliography.
It is likely that our approach fits these questions well.
\end{enumerate}


\begin{thebibliography}{10}

\bibitem{AAS} {\sc B. Abbas, H. Attouch, B. F. Svaiter},   {\it Newton-like dynamics and forward-backward methods for structured monotone inclusions
in Hilbert spaces}, J. Optim. Theory Appl., 161(2) (2014),  pp. 331-360.

\bibitem{AA-preprint-jca} {\sc S. Adly, H. Attouch},   \textit{Finite time stabilization of continuous inertial dynamics
combining dry friction with Hessian-driven damping}, J.  Conv. Anal., 28 (2) (2021),
https://hal.archives-ouvertes.fr/hal-02557928v3.

 \bibitem{AA0} {\sc S. Adly, H. Attouch},   \textit{Finite convergence of proximal-gradient inertial algorithms with dry friction damping}, Math. Program., (2020),  https://hal.archives-ouvertes.fr/hal-02388038.

 \bibitem{AA} {\sc S. Adly, H. Attouch},   \textit{Finite convergence of proximal-gradient inertial algorithms combining dry friction with Hessian-driven damping},   SIAM J. Optim., 30(3)  (2020), pp. 2134--2162.

\bibitem{AAC} {\sc S. Adly, H. Attouch, A.  Cabot},  \textit{Finite time stabilization of nonlinear oscillators subject to dry friction}, Nonsmooth Mechanics and Analysis,  Adv. Mech. Math., 12  (2006),  pp.  289--304.

\bibitem{APT} {\sc  F. Alabau-Boussouira, Y. Privat, E.L. Tr\'elat}, \textit{Nonlinear damped partial differential
equations and their uniform discretizations}, J. Func. Anal., 273 (1) (2017), pp. 352--403.

\bibitem{ALP} {\sc  C.D. Alecsa, S. L\'aszl\'o, T. Pinta}, \textit{An extension of the second order dynamical system that models Nesterov’s convex gradient method}, App. Math. Opt., (2020), doi:10.1007/s00245-020-09692-1.


\bibitem{Alvarez} {\sc F. \'Alvarez},   \textit{On the minimizing property of a second-order dissipative system in Hilbert spaces},
SIAM J. Control Optim., 38 (4) (2000), pp. 1102-1119.

\bibitem{AA1} {\sc F. \'Alvarez, H. Attouch}, {\it  An inertial proximal method for maximal monotone operators via discretization of a nonlinear oscillator with damping}, Set-Valued Anal.,  9 (1-2) (2001), pp.  3--11.

 \bibitem{AA2} {\sc F. Alvarez, H. Attouch},   {\it Convergence and asymptotic stabilization for some damped hyperbolic equations with non-isolated equilibria},
    ESAIM Control Optim. Calc. of Var.,  6 (2001), pp.  539--552.


\bibitem{AABR}{\sc F. \'Alvarez, H. Attouch, J. Bolte, P. Redont}, {\it A second-order gradient-like dissipative dynamical system with Hessian-driven damping. Application to optimization and mechanics},   J. Math. Pures Appl.,  81 (8) (2002),  pp.  747--779.


\bibitem{AmaDia} {\sc H. Amann,  J. I. D\'\i az}, \textit{ A note on the dynamics of an
oscillator in the presence of strong friction},  Nonlinear Anal.,
 55 (2003), pp. 209--216.


\bibitem{AAD}{\sc V. Apidopoulos, J.-F. Aujol,  Ch. Dossal},
{\it Convergence rate of inertial Forward-Backward algorithm beyond Nesterov's rule}, Math. Program., 180 (2020), pp. 137-–156.

\bibitem{AADR}{\sc V. Apidopoulos, J.-F. Aujol,  Ch. Dossal, A. Rondepierre},
 \textit{Convergence rates of an inertial gradient descent algorithm under growth and flatness conditions}, (2019), https://hal.archives-ouvertes.fr/hal-01965095v2.

 \bibitem{Att00} {\sc H. Attouch}, {\it Variational Convergence for Functions and Operators},   Pitman Advanced Publishing Program, Applicable Mathematics Series, 1984.

 \bibitem{ABS} {\sc H. Attouch, J. Bolte, B. F. Svaiter},
  \textit{Convergence of descent methods for semi-algebraic and tame problems: proximal algorithms, forward-backward splitting, and regularized Gauss-Seidel methods},
  Math. Program.,  137 (1) (2013), pp. 91--129.

\bibitem{abrs1}  {\sc H. Attouch, J. Bolte, P. Redont, A. Soubeyran}, \textit{Proximal alternating minimization and projection methods for nonconvex problems.
An approach based on the Kurdyka-Lojasiewicz inequality},
 Math. Oper. Res.,  35 (2) (2010), pp. 438-457.


 \bibitem{ABM} {\sc H. Attouch, G. Buttazzo, G. Michaille}, {\it Variational Analysis in Sobolev and BV Spaces. Applications to PDE's and Optimization, Second Edition}, MOS/SIAM Series on Optimization, MO 17, Society for Industrial and Applied Mathematics (SIAM), Philadelphia, PA, (2014), 793 pages.
%

\bibitem{AC1} {\sc H. Attouch, A.  Cabot}, {\it Convergence of a relaxed inertial proximal algorithm for
maximally monotone operators},
 Math. Program.,  184 (2020), pp. 243--287.


\bibitem{AC10} {\sc H. Attouch, A.  Cabot}, {\it Asymptotic stabilization of inertial gradient dynamics with time-dependent viscosity},   J. Differ. Equ., 263 (9) (2017), pp. 5412--5458.


\bibitem{AC2} {\sc H. Attouch, A.  Cabot},  {\it Convergence rates of inertial forward-backward algorithms},   SIAM J. Optim.,
 28 (1) (2018), pp. 849--874.

 \bibitem{AC2R-JOTA} {\sc   H. Attouch, A.  Cabot, Z. Chbani, H. Riahi},
\textit{Accelerated forward-backward algorithms with perturbations. Application to Tikhonov regularization},
 J. Optim. Theory Appl., 179 (1) (2018),   pp. 1--36.

 \bibitem{ACR} {\sc  H. Attouch, A.  Cabot, P. Redont},  \textit{The dynamics of elastic shocks via epigraphical
 regularization of a differential inclusion},   Adv. Math. Sci. Appl.,  12 (1)  (2002), pp.  273--306.


\bibitem{ACFR}  {\sc H. Attouch, Z. Chbani, J. Fadili, H. Riahi}, {\it First-order algorithms via inertial systems with Hessian driven damping},  Math. Program., (2020), https://doi.org/10.1007/s10107-020-01591-1.


\bibitem{ACPR} {\sc H. Attouch,  Z. Chbani, J. Peypouquet, P. Redont},  {\it Fast convergence of inertial dynamics and algorithms
with asymptotic vanishing viscosity}, Math. Program., 168 (2018), pp.  123--175.

\bibitem{ACR-Pafa} {\sc H. Attouch,  Z. Chbani, H. Riahi},
\textit{Fast convex optimization via time scaling of damped inertial gradient dynamics},
Pure and Applied Functional Analysis, (2019).


\bibitem{ACR-subcrit}  {\sc H. Attouch,  Z. Chbani, H. Riahi}: {Rate of convergence  of the Nesterov accelerated gradient method  in the subcritical case  $\alpha \leq 3$}. ESAIM Control Optim. Calc. of Var., 25 (2019),  DOI:10.1051/cocv/2017083.


\bibitem{AGR} {\sc H. Attouch, X.  Goudou, P. Redont},  \textit{The heavy ball with friction method. The continuous dynamical system, global exploration of the local minima
 of a real-valued function by asymptotical analysis of a dissipative dynamical system},
 Commun. Contemp. Math., 2 (1)  (2000), pp.  1--34.


\bibitem{AL} {\sc H. Attouch, S. C. L\'aszl\'o},
\textit{Newton-like inertial dynamics and proximal algorithms governed by maximally monotone operators}, SIAM J. Optim.,  30(4)  (2020), pp. 3252--3283.


\bibitem{AL2} {\sc H. Attouch, S. C. L\'aszl\'o},
\textit{Continuous Newton-like inertial dynamics for monotone inclusions},
Set-Valued Var. Anal., (2020), https://doi.org/10.1007/s11228-020-00564-y.


\bibitem{AM} {\sc H. Attouch, P.E. Maing\'e}, {\it Asymptotic behavior of second order dissipative evolution equations combining potential with non-potential effects}, ESAIM Control Optim. Calc. of Var.,  17 (3) (2011), pp. 836--857.

\bibitem{AMR} {\sc H. Attouch, P.E. Maing\'e, P. Redont}, {\it A second-order differential system with Hessian-driven damping; Application to non-elastic shock laws}, Differ. Equ. Appl.,  4 (1) (2012), pp. 27--65.

\bibitem{AMAS}  {\sc H. Attouch, M. Marques Alves, B. F. Svaiter},
\textit{A dynamic approach to a proximal-Newton method for monotone inclusions in Hilbert Spaces, with complexity $\mathcal O(1/n^2)$},  J. Conv. Anal., 23 (1) (2016), pp. 139--180.

\bibitem{AP-max}  {\sc H. Attouch, J. Peypouquet}, {\it Convergence of inertial dynamics and proximal algorithms governed by maximal monotone operators},
  Math. Program.,   174 (1-2)   (2019), pp. 391--432.

  \bibitem{AP} {\sc H. Attouch, J. Peypouquet},
 {\it The rate of convergence of Nesterov's accelerated forward-backward method is actually faster than $1/k^2$}, SIAM J. Optim., 26 (3) (2016), pp. 1824--1834.

 \bibitem{APR} {\sc H. Attouch, J. Peypouquet, P. Redont},  {\it Fast convex minimization via inertial dynamics with Hessian driven damping},   J. Differ. Equ., 261 (10),  (2016), pp. 5734--5783.

\bibitem{ARS} {\sc H. Attouch, P. Redont, B. F. Svaiter},
 {\it Global convergence of a closed-loop regularized Newton method for solving monotone inclusions in Hilbert spaces},  J.  Optim. Theory  Appl., 157 (3) (2013), pp. 624--650.

\bibitem{ASv} {\sc H. Attouch, B. F. Svaiter},
  {\it A continuous dynamical Newton-Like approach to solving monotone inclusions},   SIAM J. Control Optim., 49 (2)   (2011),  pp. 574--598.

  \bibitem{AE} {\sc J.P. Aubin, I. Ekeland}, {\it Applied Nonlinear Analysis}, Pure and Applied Mathematics,   Wiley, 1984.

   \bibitem{ADR} {\sc J.F. Aujol, Ch. Dossal, A. Rondepierre}
 \textit{ Convergence rates of the Heavy-Ball method for quasi-strongly convex optimization}, (2020), https://hal.archives-ouvertes.fr/hal-02545245.

  \bibitem{BCD}  {\sc  B. Baji, A. Cabot, J.I. Diaz},  \textit{ Asymptotics for some nonlinear damped wave equation:
finite time convergence versus exponential decay results} Annales de l'I.H.P. Analyse non lin\'eaire, 24 (6) (2007), pp. 1009--1028.

   \bibitem{BCF} {\sc  T. Barta, R. Chill,  E. Fa\v{s}angov\'a},  \textit{Every ordinary differential equation with a strict Lyapunov function is a gradient system}, Monatsh. Math., 166 (1) (2012), pp. 57--72.

\bibitem{BF} {\sc  T. B\'arta,   E. Fa\v{s}angov\'a}
  \textit{Convergence to equilibrium for solutions of an abstract wave equation with general damping function},
J. Differ. Equ., 260 (3) (2016), pp. 2259--2274.

\bibitem{BC}{\sc H. Bauschke, P. L. Combettes}, {\it  Convex Analysis and Monotone Operator Theory in Hilbert spaces}, CMS Books in Mathematics, Springer,  (2011).

\bibitem{BBJ}{\sc P. B\'egout, J. Bolte, M. Jendoubi},
{\it On damped second-order gradient systems}, J. Differ. Equ., 259 (7-8) (2015), pp. 3115--3143.

\bibitem{BDLM}{\sc J. Bolte, A. Daniilidis, O. Ley, and L. Mazet}, \textit{Characterizations of  Lojasiewicz inequalities: subgradient flows, talweg, convexity}, Trans. Amer. Math. Soc., 362 (6)  (2010), pp. 3319-3363.

\bibitem{BST}{\sc J. Bolte, S. Sabach, and M. Teboulle}, \textit{ Proximal alternating linearized minimization for nonconvex and nonsmooth problems}, Math. Program., 146 (1-2) (2014),
 pp. 459--494.

\bibitem{BotCseESAIM} {\sc R. I. Bo\c t, E. R. Csetnek},  {\it A forward-backward dynamical approach to the minimization of the sum of a nonsmooth convex with a smooth nonconvex function}, ESAIM COCV,  24(2) (2018), pp. 463--477.

\bibitem{BotCse} {\sc R. I. Bo\c t, E. R. Csetnek},  {\it Second order forward-backward dynamical systems for monotone inclusion problems}, SIAM J. Control  Optim., 54 (2016), pp. 1423-1443.

\bibitem{BotCseLaJEE}
{\sc R. I. Bo\c t, E. R. Csetnek, S.C. L\'{a}szl\'{o}},  {\it Approaching nonsmooth nonconvex minimization through second order proximal-gradient dynamical systems}, J. Evol. Equ., 18(3) (2018), pp. 1291--1318.

\bibitem{BCL}
{\sc R. I. Bo\c t, E. R. Csetnek, S.C. L\'{a}szl\'{o}},     \textit{Tikhonov regularization of a second order dynamical system with Hessian damping}, Math. Program., DOI:10.1007/s10107-020-01528-8.


\bibitem{BotCseLaEUR}  {\sc
R. I. Bo\c t, E. R. Csetnek, S.C. L\'{a}szl\'{o}},  {\it An inertial forward-backward algorithm for the minimization of the sum of two nonconvex functions}, EURO J. Comp. Optim., 4(1) (2016), pp. 3--25.

\bibitem{BotCseLaJDE}  {\sc
R. I. Bo\c t, E. R. Csetnek, S.C. L\'{a}szl\'{o}},  {\it A primal-dual dynamical approach to structured convex minimization problems}, J. Differ. Equ., 269(12) (2020), pp. 10717--10757.


%\bibitem{BN}
%{\sc R. I. Bo\c t, D. K. Nguyen},     \textit{The proximal alternating direction method of multipliers in the nonconvex setting:  convergence analysis and rates}, Math. Oper. Res., 45 (2020), pp. 682-712.

\bibitem{Brezis} {\sc H. Br\'ezis}, \textit{Op\'erateurs Maximaux Monotones dans les Espaces de Hilbert et \'Equations d'\'Evolution}, Lecture Notes 5, North Holland, 1972.

\bibitem{Bru} {\sc R. E. Bruck}, \textit{Asymptotic convergence of nonlinear
 contraction semigroups in Hilbert spaces}, J. Funct. Anal., 18 (1975), pp. 15--26.

\bibitem{CEG} {\sc
A. Cabot, H. Engler, S. Gadat}, \textit{On the long time behavior of second order differential equations with asymptotically small dissipation}, Trans. Amer. Math. Soc. 361 (2009), pp. 5983--6017.

\bibitem{CFr} {\sc A. Cabot, P. Frankel}, \textit{Asymptotics for some semilinear hyperbolic equations with non-autonomous damping}, J. Differ. Equ., 252 (2012), pp. 294--322.

\bibitem{CarlesGallo} {\sc R. Carles, C. Gallo},  \textit{Finite time extinction by nonlinear damping for the Schr\"odinger equation}, Commun. Part. Diff. Eq., 36 (6)  (2011), pp. 961--975.

\bibitem{CBFP}
{\sc C. Castera, J. Bolte, C. F\'evotte, E. Pauwels}, \textit{An inertial Newton algorithm for
deep learning}, (2019), https://hal.inria.fr/hal-02140748/.

\bibitem{CD}{\sc  A. Chambolle, Ch. Dossal}, {\it  On the convergence of the iterates of the Fast Iterative Shrinkage Thresholding Algorithm}, J. Opt. Theory Appl., 166 (2015), pp. 968--982.

\bibitem{Chergui}{\sc L. Chergui}, \textit{Convergence of global and bounded solutions of a second order gradient like system with nonlinear dissipation and analytic nonlinearity}, J. Dynam. Differential Eq., 20(3) (2008), pp. 643--652.

\bibitem{CF}{\sc R. Chill, E. Fasangova}, \textit{Gradient Systems}, Lecture Notes of the 13th International Internet Seminar, Matfyzpress, Prague, 2010.

\bibitem{CHJ}{\sc
R. Chill, A. Haraux, M. A. Jendoubi}, \textit{Applications of the Lojasiewicz-Simon gradient inequality to gradient-like evolution equations}, Anal. Appl., 7 (2009), pp. 351-372.

\bibitem{EZ}{\sc S. Ervedoza, E. Zuazua},
\textit{Uniformly exponentially stable approximations for a class
of damped systems}, J. Math. Pures Appl., 91 (2009), pp. 20--48.

\bibitem{GGH}{\sc M. Ghisi, M. Gobbino, A. Haraux},
{\it Universal bounds for a class of second order evolution equations and applications}, J. Math. Pures Appl., 142 (2020), pp. 184--203.

\bibitem{GCG}{\sc A. Ghose-Choudhury, P. Guha},
{\it An analytic technique for the solutions of nonlinear oscillators with damping using Abel Equation}, (2016), arXiv:1608.02324v1.

\bibitem{GT}{\sc I. M. Gelfand, M. L. Tsetlin},  \textit{The method of ravines}, Optimizatsii,
Dokl. AN SSSR, 137 (1961), pp. 295--298.

\bibitem{GP}{\sc M. Grasselli, M. Pierre},  \textit{Convergence to equilibrium of solutions of the backward Euler scheme for asymptotically autonomous second-order gradient-like systems}, Commun. Pure Appl. Anal., 11 (6) (2012), pp. 2393--2416.

\bibitem{Hale}{\sc J. Hale}, {\it Stability and gradient dynamical systems}, Rev. Mat. Complut., 17 (1) (2004),  pp. 7--57.

\bibitem{Haraux}{\sc  A. Haraux}, {\it  Syst\`emes Dynamiques Dissipatifs et Applications}, Recherches en Math\'ematiques Appliqu\'ees, RMA 17, Masson, 1991.

\bibitem{Haraux_PDE}{\sc  A. Haraux}, {\it  Nonlinear Evolution Equations - Global Behavior of Solutions}, Lecture Notes in Mathematics 841, Spinger-Verlag, 1981.

\bibitem{HJ1} {\sc A. Haraux, M. A. Jendoubi},
\textit{Convergence of solutions of second-order gradient-like systems with analytic nonlinearities},
J. Differ. Equ., 144 (2) (1998), pp. 313--320.

\bibitem{HJ2} {\sc A. Haraux, M. A. Jendoubi},
\textit{The Convergence Problem for Dissipative Autonomous Systems},
Classical Methods and Recent Advances, Springer, 2015.

\bibitem{HWL} {\sc B. Hu, S. Wright, L. Lessard},
\textit{Dissipativity theory for accelerating stochastic variance reduction: a unified analysis of SVRG and Katyusha using semidefinite programs}, ICML (2018).

\bibitem{Huang} {\sc S-Z. Huang},
\textit{Gradient Inequalities with Applications to Asymptotic Behavior and Stability of Gradient-Like Systems},
Math. Surveys and Monographs, Vol. 126, Amer. Math. Soc., 2006.

\bibitem{Ioffe} {\sc A.Ioffe}, \textit{An invitation to tame optimization},  SIAM J. Optim., 19(4) (2009), pp. 1894--1917.

\bibitem{IH} {\sc F. Iutzeler, J. M. Hendrickx},  \textit{A generic online acceleration scheme for optimization algorithms via relaxation and inertia}, Opt. Meth. Soft., 34 (2) (2019), pp. 383-405.

\bibitem{Kim}{\sc  D. Kim}, {\it Accelerated proximal point method for maximally monotone operators}, (2020), arXiv:1905.05149v3.

%\bibitem{LiPong}{\sc  G. Li,  T.K. Pong}, {\it Global convergence of splitting methods for nonconvex composite optimization}, SIAM J. Optim., 25 (4) (2015), pp.  2434–-2460.


\bibitem{LiPongFCM}
{\sc  G. Li,  T.K. Pong}, {\it Calculus of the exponent of Kurdyka-Łojasiewicz inequality and its applications to linear convergence of first-order methods}, Found. Comput. Math. 18 (5) (2018), pp. 1199–1232.


\bibitem{LJ}{\sc  T. Lin,  M.I. Jordan}, {\it A control-theoretic perspective on optimal high-order optimization}, (2019), preprint
arXiv:1912.07168v1.

\bibitem{May}{\sc R. May}, {\it Asymptotic for a second-order evolution equation with convex potential and vanishing damping term},  Turkish Journal of Math., 41 (3) (2017), pp. 681--685.

\bibitem{Nest1}{\sc  Y. Nesterov}, {\it   A method of solving a convex programming problem with convergence rate
O(1/k2)}, Soviet Mathematics Doklady,  27  (1983), pp. 372--376.

\bibitem{Nest2}{\sc  Y. Nesterov}, {\it  Introductory Lectures on Convex Optimization: A Basic Course},
Applied Optimization, Vol. 87,  Kluwer Academic Publishers, Boston, MA, (2004).


\bibitem{ipiano}{\sc P.  Ochs,  Y.  Chen,  T.  Brox,  T.  Pock}, \textit{iPiano:   Inertial  proximal  algorithm  for  non-convex optimization}, SIAM J. Imag. Sci., 7(2) (2014), pp. 1388–1419.


 \bibitem{Op}   {\sc Z. Opial}, \textit{Weak convergence of the sequence of successive approximations for nonexpansive mappings}, Bull. Amer. Math. Soc., 73 (1967), pp. 591--597.

\bibitem{PDM} {\sc J. Palis, W. De Melo},  \textit{Geometric Theory of Dynamical Systems - An Introduction}, (Translated
from the Portuguese by A. K. Manning), Springer-Verlag, New York-Berlin, 1982.

\bibitem{Pol} {\sc B.T. Polyak}, \textit{Some methods of speeding up the convergence
of iterative methods},  Z. Vylist Math. Fiz.,  4 (1964), pp. 1--17.

\bibitem{Polyak2}  {\sc B.T. Polyak}, {\it  Introduction to Optimization}. New York: Optimization Software, (1987).

\bibitem{PKD} {\sc B.T. Polyak, N. Kulakova, M. Danilova},
\textit{Accelerated gradient methods revisited}, Workshop ``Variational Analysis and Applications'', August 28--September 5, 2018, Erice.

\bibitem{PL}{\sc C. Poon, J. Liang},  \textit{Geometry of first-order methods and adaptive acceleration}, (2020), arXiv:2003.03910v1.

\bibitem{SDJS}{\sc
B. Shi, S.  S. Du,  M. I. Jordan,  W. J. Su}, {\it Understanding the acceleration phenomenon via high-resolution differential equations}, (2018), arXiv:submit/2440124

\bibitem{SBC}{\sc W. J. Su,  S. Boyd,  E. J. Cand\`es}, {\it A differential equation for modeling Nesterov's accelerated gradient method: theory and insights}, Neural Inf. Proc. Sys., 27 (2014), pp. 2510--2518.

\bibitem{vdDries}{\sc L.  van den Dries}, \textit{Tame Topology and o-Minimal Structures}, London Mathematical Society, Lecture Note Series, Vol. 248., Cambridge University Press, Cambridge, UK, (1998).

\bibitem{VSBV} {\sc S. Villa, S. Salzo, L. Baldassarres, A. Verri}, \textit{Accelerated and inexact forward-backward}, SIAM J. Optim.,  23 (3)  (2013),  pp. 1607--1633.



\end{thebibliography}
\end{document}